\documentclass[12pt]{amsart}

\usepackage{calligra,mathrsfs}
\usepackage[all]{xy}
\usepackage{float, comment}
\usepackage{mathtools}
\usepackage{amsmath}
\usepackage{amsthm}
\usepackage{amssymb}
\usepackage{amsbsy}
\usepackage{amstext}
\usepackage{amsopn}
\usepackage{mathrsfs} 
\usepackage{enumerate}
\usepackage{xcolor}
\usepackage{graphicx} 
\usepackage{microtype} 
\usepackage[margin=1in,marginparwidth=0.8in, marginparsep=0.1in]{geometry}
\usepackage[pagebackref, bookmarks=true, bookmarksopen=true, bookmarksdepth=3,bookmarksopenlevel=2, colorlinks=true, linkcolor=blue, citecolor=blue, filecolor=blue, menucolor=blue, urlcolor=blue]{hyperref}
\usepackage{tikz}
\usepackage{bbm}
\usepackage[all]{xy}
\usetikzlibrary{arrows,calc,positioning,decorations.pathreplacing} 

\newtheorem{thmx}{Theorem}

\usepackage[shortlabels]{enumitem}

\numberwithin{equation}{section}

\usepackage{todonotes}

\usepackage{xcolor}

\usepackage{stmaryrd}

\usepackage{tikz-cd}

\usepackage{enumitem}

\usepackage{quiver}

\numberwithin{equation}{section}
\newtheorem{thm}{Theorem}[section]
\newtheorem{prop}[thm]{Proposition}
\newtheorem{lem}[thm]{Lemma}
\newtheorem{cor}[thm]{Corollary}
\newtheorem{ex}[thm]{Example}
\newtheorem{rem}[thm]{Remark}

\newtheorem{conj}[thm]{{\bf Conjecture}}

\newtheorem*{cor*}{Corollary}

\newcommand{\nc}{\newcommand}

\nc{\bA}{\mathbb A}
\nc{\bC}{\mathbb C}
\nc{\bc}{{\bf c}}
\nc{\bD}{\mathbb D}
\nc{\bd}{\mathbb d}
\nc{\bG}{\mathbb G}
\nc{\bi}{\bold i}
\nc{\bL}{\mathbb L}
\nc{\bN}{\mathbb N}
\nc{\bP}{\mathbb P}
\nc{\bQ}{\mathbb Q}
\nc{\bR}{\mathbb R}
\nc{\bu}{\mathbb u}
\nc{\bv}{\bold v}
\nc{\bw}{\bold w}
\nc{\bW}{\mathbb W}
\nc{\bZ}{\mathbb Z}

\nc{\cA}{\mathcal A}
\nc{\cB}{\mathcal B}
\nc{\cC}{\mathcal C}
\nc{\cD}{\mathcal D}
\nc{\cE}{\mathcal E}
\nc{\cF}{\mathcal F}
\nc{\cG}{\mathcal G}
\nc{\cH}{\mathcal H}
\nc{\cI}{\mathcal I}
\nc{\cK}{\mathcal K}
\nc{\cL}{\mathcal L}
\nc{\cM}{\mathcal M}
\nc{\cN}{\mathcal N}
\nc{\cO}{\mathcal O}
\nc{\cP}{\mathcal P}
\nc{\cR}{\mathcal R}
\nc{\cT}{\mathcal T}
\nc{\cU}{\mathcal U}
\nc{\cV}{\mathcal V}
\nc{\cW}{\mathcal W}
\nc{\cY}{\mathcal Y}
\nc{\cX}{\mathcal X}

\nc{\al}{\alpha}
\nc{\be}{\beta}
\nc{\la}{\lambda}
\nc{\La}{\Lambda}
\nc{\ve}{\varepsilon}
\nc{\om}{\omega}
\nc{\bom}{\boldsymbol{\om}}

\nc{\gl}{\mathfrak{gl}}
\nc{\fsl}{\mathfrak{sl}}
\nc{\ff}{\mathfrak{f}}
\nc{\g}{\mathfrak{g}}
\nc{\gh}{\widehat\g}
\nc{\h}{\mathfrak{h}}
\nc{\fb}{{\mathfrak b}}
\nc{\fg}{{\mathfrak g}}
\nc{\fgh}{{\widehat{\mathfrak g}}}
\nc{\fh}{{\mathfrak h}}
\nc{\fl}{\mathfrak{l}}
\nc{\fL}{\mathfrak{L}}
\nc{\fm}{{\mathfrak m}}
\nc{\fM}{{\mathfrak M}}
\nc{\fp}{{\mathfrak p}}
\nc{\ft}{\mathfrak{t}}

\nc{\fn}{{\mathfrak n}}
\nc{\fQ}{\mathfrak{Q}}

\nc{\Aut}{\mathrm{Aut}}
\nc{\ch}{{\mathop {\rm ch}}}
\nc{\tr}{{\mathop {\rm tr}\,}}
\nc{\im}{\operatorname{im}}
\nc{\id}{{\mathop {\rm id}}}
\nc{\ad}{{\mathop {\rm ad}}}
\nc{\gr}{\mathrm{gr}}
\nc{\ord}{\mathrm{ord}}
\nc{\red}{\mathrm{red}}
\nc{\End}{\operatorname{End}}
\nc{\Spec}{\operatorname{Spec}}
\nc{\Spf}{\operatorname{Spf}}
\nc{\Proj}{\operatorname{Proj}}
\nc{\Pic}{\operatorname{Pic}}
\nc{\Lie}{\operatorname{Lie}}
\nc{\Der}{\operatorname{Der}}
\nc{\Coh}{\mathrm{Coh}}
\nc{\coh}{\mathrm{coh}}
\nc{\qcoh}{\mathrm{Qcoh }}
\nc{\Gal}{\operatorname{Gal}}
\nc{\Hom}{\mathrm{Hom}}
\nc{\Rhom}{\mathrm{RHom}}
\nc{\cHom}{\mathcal{Hom}}
\nc{\Ann}{\mathrm{Ann}}
\nc{\Vect}{\mathrm{Vect}}
\nc{\wt}{\mathrm{wt}}
\nc{\hw}{\mathrm{hw}}
\nc{\rk}{\operatorname{rank}}
\nc{\Gr}{{\mathrm {Gr}}}
\nc{\Fl}{\mathrm{Fl}}
\nc{\spn}{\mathrm{span}}
\nc{\Rep}{\operatorname{Rep}}
\nc{\Irrep}{\mathrm{Irrep }}
\nc{\supp}{\operatorname{supp}}
\nc{\tp}{\mathrm{top}}
\nc{\codim}{\mathrm{codim}}
\nc{\IC}{\operatorname{IC}}
\nc{\Res}{\mathrm{Res}}
\nc{\modules}{\mathrm{-mod}}
\nc{\Perv}{\mathrm{Perv}}
\nc{\Forg}{\operatorname{Forg}}
\nc{\Maps}{\mathrm{Maps}}
\nc{\Frac}{\operatorname{Frac}}

\nc{\GfL}{{(G \times \bC^\times, \fl \oplus \bC D)}}
\nc{\Gfl}{{(G \times \bC^\times, \fl \oplus \bC D)}}

\nc{\eO}{\EuScript{O}}

\nc{\bra}{\langle}
\nc{\ket}{\rangle}
\nc{\pa}{\partial}
\nc{\ld}{\ldots}
\nc{\cd}{\cdots}
\nc{\hk}{\hookrightarrow}
\nc{\T}{\otimes}
\nc{\ov}{\overline}
\nc{\wh}{\widehat}
\nc{\wti}{\widetilde}
\nc{\svee}{{\!\scriptscriptstyle\vee}}
\nc{\ula}{{\underline{\la}}}
\nc{\umu}{{\underline{\mu}}}
\nc{\conv}{{\widetilde \times}}
\nc{\lach}{{\la^\svee}}
\nc{\alch}{{\al^\svee}}
\nc{\omch}{{\omega^\svee}}
\nc{\much}{{\mu^\svee}}
\nc{\md}{\text {--mod}}
\nc{\pt}{\mathrm{pt}}
\nc{\torus}{\bC^\times}

\nc{\CF}{\mathcal{F}}

\nc{\GL}{\mathfrak{GL}}
\nc{\Tr}{{\mathop {\rm Tr}\,}}
\nc{\Id}{{\mathop {\rm Id}}}

\nc{\msl}{\mathfrak{sl}}
\nc{\mgl}{\mathfrak{gl}}

\nc{\U}{\mathrm U}

\nc{\Q}{\mathfrak Q}

\nc{\on}{\operatorname} \nc\ol{\overline} \nc\ul{\underline}

\nc{\BA}{{\mathbb{A}}} \nc{\BC}{{\mathbb{C}}} \nc{\BF}{{\mathbb{F}}}
\nc{\BD}{{\mathbb{D}}} \nc{\BG}{{\mathbb{G}}} \nc{\BQ}{{\mathbb{Q}}}
\nc{\BM}{{\mathbb{M}}} \nc{\BN}{{\mathbb{N}}} \nc{\BO}{{\mathbb{O}}}
\nc{\BP}{{\mathbb{P}}} \nc{\BR}{{\mathbb{R}}}
\nc{\BZ}{{\mathbb{Z}}} \nc{\BS}{{\mathbb{S}}} \nc{\BW}{{\mathbb{W}}}

\nc{\CA}{{\mathcal{A}}} \nc{\CL}{{\mathcal{L}}} \nc{\CV}{{\mathcal{V}}} \nc{\CW}{{\mathcal{W}}}
\nc{\CalD}{{\mathcal{D}}}

\nc{\sic}{{\on{sc}}}
\nc{\add}{{\on{add}}}

\newcommand\iso{\,\vphantom{j^{X^2}}\smash{\overset{\sim}{\vphantom{\rule{0pt}{0.20em}}\smash{\longrightarrow}}}\,}

\title{K-theoretic Hikita conjecture for quiver gauge theories}

\author{Ilya Dumanski}
\address{Ilya Dumanski:\newline
		Department of Mathematics, MIT, Cambridge, MA 02139, USA.
	}
	\email{ilyadumnsk@gmail.com}

\author{Vasily Krylov}
\address{Vasily Krylov: \newline
Department of Mathematics
Harvard University and CMSA
\newline
1 Oxford Street,
Cambridge, MA 02138,
USA}
\email{vkrylov@math.harvard.edu, krylovasya@gmail.com}

\begin{document}

\begin{abstract}
    We study variants of Hikita conjecture for Nakajima quiver varieties and corresponding Coulomb branches. First, we derive the equivariant version of the conjecture from the non-equivariant one for a set of gauge theories. Second, we suggest a variant of the conjecture, with K-theoretic Coulomb branches involved. We show that this version follows from the usual (homological) one for a set of theories. We apply this result to prove the conjecture in finite ADE types. In the course of the proof, we show that appropriate completions of K-theoretic and homological (quantized) Coulomb branches are isomorphic.
\end{abstract}

\maketitle

\tableofcontents

\section{Introduction}

This paper concerns the phenomenon of $3d$ mirror symmetry, also known as symplectic duality \cite{BLPW14, Kam22, WY23}. From mathematics perspective, it includes a set of examples of \textit{dual} pairs of symplectic singularities, with various expected relations between them. In this paper, we concentrate on pairs of varieties, associated with quiver gauge theories (Higgs  and Coulomb branches). 


\subsection{Homological Hikita conjecture}
Given a quiver $Q$, the dimension vector $\bv$ and framing vector $\bw$, one can build the Nakajima quiver varieties $\wti \fM_Q(\bv, \bw) \rightarrow \fM_Q(\bv, \bw)$ (Higgs branch of the corresponding quiver gauge theory) and Braverman--Finkelberg--Nakajima Coulomb branch $\cM_Q(\bv, \bw)$. These varieties are symplectic dual, with deep connections established between them, including Koszul duality for categories $\cO$ \cite{Web16}. We concentrate on a conjectural algebraic relation between these varieties, called the Hikita conjecture, originated in \cite{Hik17} (for other pairs of dual varieties). In its simplest form, it states that the following isomorphism of algebras should hold\footnote{Note that we do \underline{not} assume  that the natural $\mathbb{C}^\times$-action on $\cM_Q(\bv,\bw)$ is conical, i.e., that the corresponding quiver theory is  {{``good or ugly''}}. We believe (and prove in some cases) that the isomorphism~\eqref{eq: intro homo hikita} should hold without this additional assumption.}:
\begin{equation} \label{eq: intro homo hikita}
H^*(\wti \fM_Q(\bv, \bw),\mathbb{C}) \simeq \bC[\cM_Q(\bv, \bw)^{\nu(\bC^\times)}].
\end{equation}
Here $\bC^\times$ acts on $\cM_Q(\bv, \bw)$ through a generic cocharacter $\nu$ of the Hamiltonian torus action, and $\cM_Q(\bv, \bw)^{\nu(\bC^\times)}$ stands for schematic fixed points. 
In fact, on both sides of \eqref{eq: intro homo hikita}, there is a natural action of the polynomial algebra $\bC[\ft_\bv / S_\bv]$ (see next paragraph for definitions of $\ft_\bv, S_\bv$), and we require that \eqref{eq: intro homo hikita} is an isomorphism of $\bC[\ft_\bv / S_\bv]$-algebras (this condition actually determines the isomorphism uniquely).

Further, Nakajima suggested a deformation of the isomorphism (\ref{eq: intro homo hikita}). 
Let $F \curvearrowright \wti \fM_Q$ be an action of a torus $F$ (called the flavor torus) by Hamiltonian automorphisms, commuting with the  $\bC^\times$-action on $\wti \fM_Q$, and let  $\ff = \Lie F$. 
For example, one can take $F = T_\bw \subset GL_\bw$ (the maximal torus of the framing group).
Denote also by $T_\bv \subset GL_\bv$ the gauge group and its maximal torus, set $S_\bv$ to be its Weyl group, and $\ft_\bv = \Lie T_\bv$.
Then $\cM_Q(\bv, \bw)$ admits a natural Poisson deformation over $\ff$, denoted $\cM_Q(\bv, \bw)_{\ff}$ (see Section~\ref{subsection: homological Hikita} for details). The equivariant Hikita conjecture (a.k.a. Hikita--Nakajima conjecture) is the following isomorphism of $\bC[(\ft_\bv / S_\bv) \times \ff]$-algebras:
\begin{equation} \label{eq: intro equi homo hikita}
H^*_{F}(\wti \fM_Q(\bv, \bw)) \simeq \bC[\cM_Q(\bv, \bw)_{\ff}^{\nu(\bC^\times)}].
\end{equation}

The first result of this paper concerns a method to deduce the equivariant version \eqref{eq: intro equi homo hikita} from the non-equivariant one \eqref{eq: intro homo hikita}. 
For simplicity, in the Introduction we state this theorem for the case when $F = T_\bw$ is the framing torus. 
In the main body of the text, we drop this assumption.

\begin{thmx}[See Theorem~\ref{equivariant hikita from non-equivariant} for the full statement]\label{thm: intro equi-homo-hikita}
Fix a quiver $Q$ and let $F = T_\bw$. Suppose the homological Hikita conjecture~\eqref{eq: intro homo hikita} holds for any $(\bv, \bw)$. Then the homological equivariant Hikita conjecture \eqref{eq: intro equi homo hikita}  holds for any $(\bv, \bw)$.
\end{thmx}

Note that to deduce the equivariant conjecture (\ref{eq: intro equi homo hikita}) for a fixed $(\bv, \bw)$, it is not sufficient to know the non-equivariant conjecture (\ref{eq: intro homo hikita}) for the same $(\bv, \bw)$. Our argument is inductive, and uses all values of $v_i, w_i$, less than or equal to the required.

In particular, Theorem \ref{thm: intro equi-homo-hikita} allows us to establish previously not proven equivariant version of conjecture for ADE quivers under mild assumptions, see Corollary \ref{homo-hikita for ADE quivers}.  Recall that the Coulomb branch in this case is isomorphic to a generalized slice in the affine Grassmannian, see \cite{bfn_slices}.
Note that the non-equivariant version of the conjecture was proved in \cite[Theorem~8.1]{KTWWY19} for the case when the corresponding slice in affine Grassmannian is {\underline{non}}-generalized (and under the same assumption as in Corollary \ref{homo-hikita for ADE quivers}). As we pointed out above, our method requires knowing the validity of non-equivariant version for all $(\bv, \bw)$. So, we need the non-equivariant Hikita conjecture for generalized slices, and we prove it in Appendix, see Theorem~\ref{thm: hikita for generalized slices}. Also, in Appendix we provide a direct geometric argument for the equivariant conjecture, generalizing the method of \cite{KTWWY19} (and giving an alternative proof of Corollary~\ref{homo-hikita for ADE quivers}). We think that this proof is conceptually interesting. The connection between quiver and Coulomb sides there comes from an isomorphism of global Demazure and global Weyl modules in types ADE.

We should mention that the equivariant version has been proved previously in type A in \cite[Theorem 8.3.7]{Wee16}  (see also \cite[Remark 8.13]{KTWWY19}) as well as its weak form in DE types  (see \cite[Proposition 8.11]{KTWWY19} and \cite[Theorem 1.5]{KTWWY19b}).

\subsection{K-theoretic Hikita conjecture}
The main goal of the paper, however, is to study another version of the conjecture; we call this version the \textit{K-theoretic Hikita conjecture} (see \cite[Appendix B]{Zho23} where the hypertoric case is studied). The idea is as follows: in \eqref{eq: intro equi homo hikita}, on the quiver side, one needs to replace equivariant cohomology by equivariant K-theory, and on the Coulomb side, one needs to replace the Coulomb branch $\cM_Q(\bv, \bw)$ by the K-theoretic Coulomb branch $\cM^\times_Q(\bv, \bw)$. We identify $K_{G_{\bf{v}} \times F}(\operatorname{pt}) = \bC[(T_{\bf{v}}/S_{\bf{v}}) \times F]$.
\begin{conj}[See Conjecture \ref{k-hikita conjecture} for definitions and the full statement] \label{conj: intro K-hikita}
There is an isomorphism of $\bC[(T_\bv / S_\bv) \times F]$-algebras:
\begin{equation} \label{eq: intro k-hikita}
K^{F}(\wti \fM_Q(\bv, \bw)) \simeq \bC[\cM^\times_Q(\bv, \bw)_{F}^{\nu(\bC^\times)}].
\end{equation}
\end{conj}

We also suggest a quantized version of the conjecture, with added $\bC^\times$-action on the quiver side, and replacing $\cM^\times_Q(\bv, \bw)$ by its quantization on the Coulomb side, see Conjecture~\ref{k-hikita conjecture} (again, compare with  \cite[Appendix B]{Zho23}).
The present paper, however, deals with the non-quantized version.

Let us state our main result concerning the K-theoretic Hikita conjecture, and then explain our method (both for Theorems \ref{thm: intro equi-homo-hikita} and \ref{thm: intro k-hikita}). As above, for simplicity we assume that $F = T_\bw$ to state the result here, but in the main body of the text we do not have this assumption. 

\begin{thmx} [See Theorem \ref{k-hikita from homo-hikita}
for the full statement, allowing arbitrary $F$] \label{thm: intro k-hikita}
Fix a quiver $Q$, let $F = T_\bw$. Suppose the equivariant homological Hikita conjecture \eqref{eq: intro equi homo hikita} holds for any $(\bv, \bw)$. Then the equivariant K-theoretic Hikita conjecture (\ref{eq: intro k-hikita}) holds for any $(\bv, \bw)$.
\end{thmx}

This allows us to establish the equivariant K-theoretic Hikita conjecture for ADE quivers under a mild assumption (Corollary \ref{cor: k-hikita for ade}) and a weaker form of the conjecture for the Jordan quiver (Corollary \ref{cor: k-hikita jordan quiver}). 
Summarizing, we obtain the following result as a corollary of the above theorems. 
\begin{thmx}
Both K-theoretic and cohomological equivariant Hikita conjecture holds for arbitrary quiver $Q$ of type ADE and ${\bf{v}}, {\bf{w}}$ satisfying conditions as in Corollary \ref{homo-hikita for ADE quivers}. 
\end{thmx}

As one immediate corollary, we obtain a parametrization of $\nu(\mathbb{C}^\times)$-fixed points on (deformed) K-theoretic Coulomb branches for ADE quiver theories. In particular, we see that $(\mathcal{M}_Q^\times)^{\nu(\mathbb{C}^\times)}$ consits of one point if $V(\lambda)_\mu \neq 0$, and is empty otherwise (compare with \cite[Conjecture~3.25(1)]{bfn_slices}). Here $V(\la)$ is the irreducible representation of $\mathfrak{g}_{Q}$ with highest weight $\lambda=\sum_{i}w_i\omega_i$ and $\mu=\lambda-\sum_{i}v_i\alpha_i$.  We also conclude that the algebra of functions on schematic fixed points $\bC[(\mathcal{M}_Q^\times)^{\nu}_F]$ is flat over $F$.


Our method of dealing with  both homological and K-theoretic conjecture is by study of what one can call the {\it factorization property} of both sides of equivariant Hikita conjecture. Let us explain it in more detail.

\subsection{Outlines of proofs}

\subsubsection{Proof outline of Theorem \ref{thm: intro equi-homo-hikita}} \label{subsubseq: intro method for homo-hikita}


Note that \eqref{eq: intro equi homo hikita} is a deformation of \eqref{eq: intro homo hikita} over $\bC[\ff]$, meaning that the fiber over $0 \in \ff$ of \eqref{eq: intro equi homo hikita} is \eqref{eq: intro homo hikita}. Our idea is to look at fiber over an arbitrary point $t \in \ff$.

For quiver side, the localization theorem in equivariant cohomology reduces the computation of this fiber to the computation of $t$-fixed points of $\wti \fM_Q$; this technique goes back at least to \cite{Nak01b}. We suggest a variant of this result in Proposition~\ref{localization of quiver side of homological hikita}, Corollary~\ref{localization of quiver side of homological hikita over T_w}.

For the Coulomb side, the idea is similar. Localization technics for Coulomb branches appear already in \cite[5(i)]{BFN18}. 
We are interested in the fiber of $H^{F \times G_\bv}(\cR_{G_\bv, {\bf{N}}})$ at $t \in \ff$ as a module over $H_{F}(\pt)$, so the localization theorem should be applied non-directly.
Under certain assumption, we compute the fiber of the algebra of schematic fixed points, appearing in~\eqref{eq: intro equi homo hikita}, see Corollary~\ref{cor: fiber of b-algebra homo}.


Combining the results of two previous paragraphs together, one sees that taking fiber of the equivariant Hikita conjecture \eqref{eq: intro equi homo hikita} over the deformation, yields the (non-equivariant) Hikita conjecture \eqref{eq: intro homo hikita}, but for a different gauge theory (if $F = T_\bw$, this is a theory corresponding to the same quiver, but to different framing and dimension vectors). So, if we assume we know the non-equivariant version for enough cases, we obtain that all fibers of \eqref{eq: intro equi homo hikita} over $\bC[\ff]$ are isomorphic. Using also compatibility with ``Kirwan-type'' maps, this allows us to establish the required result, see proof of Theorem~\ref{equivariant hikita from non-equivariant}.

\subsubsection{Proof outline of Theorem \ref{thm: intro k-hikita}}

For our main result on the K-theoretic case, we argue similarly, but consider both sides of \eqref{eq: intro k-hikita} as modules over $\bC[(T_\bv / S_\bv) \times F]$ (not over $\bC[F]$, which would have been more similar to what we did in the homological case). 
We first show that both sides of \eqref{eq: intro k-hikita} are naturally quotients of $\bC[(T_\bv / S_\bv) \times F]$, see Corollary~\ref{corollary: surjections for K-hikita}. 
This follows from the fact that K-theoretic Coulomb branches are generated by dressed minuscule monopole operators, see Proposition~\ref{monopoles generate K-Coulomb branch}. 
We have the following surjective morphisms, and we want to show that their kernels coincide: 
\begin{equation} \label{eq: intro surjections from K-theory of point to both sides}
\begin{tikzcd}
	& {K^{G_{\bold v} \times F} (\mathrm{pt})} \\
	{K^{F}(\widetilde{\mathfrak M}(G_{\bold v}, {\bf{N}}))} && {\mathbb C[(\mathcal M({G_{\bold v}}, {\bf{N}})^\times_{F})^{\nu}]}
	\arrow[two heads, "{\phi^\times_1}"', from=1-2, to=2-1]
	\arrow[two heads, "{\phi^\times_2}", from=1-2, to=2-3].
\end{tikzcd}
\end{equation}

We take a point $(t_\bv, f) \in (T_\bv / S_\bv) \times F$ and compute formal completions of morphisms $\phi^\times_1, \phi^\times_2$ at $(t_\bv, f)$.
This is done by localization theorem in equivariant K-theory, similarly to what we described in Section~\ref{subsubseq: intro method for homo-hikita}. We again get (completed) K-theoretic Hikita conjecture for a different gauge theory.

But now the following slogan comes to help: ``for nice spaces, completion of equivariant K-theory is isomorphic to completion of equivariant homology''.

For the quiver side, this isomorphism is given by the (equivariant) Chern character, we check it in Lemma \ref{chern induces isomorphism in completion for quiver varieties}. For the Coulomb side, the situation is more subtle, and we dedicate the following Section \ref{subseq: intro isomorphism of completions} for its explanation.

Combining all of the above, this reduces K-theoretic Hikita conjecture to the homological one for a larger set of gauge theories.

\subsection{Isomorphism of completed homological and K-theoretic Coulomb branches} \label{subseq: intro isomorphism of completions}
\subsubsection{The result}
As we pointed out above, it is a general phenomenon that completions of equivariant K-theory and equivariant Borel--Moore homology are isomorphic. One result of geometric nature, related to this fact for Coulomb branches, dates back to \cite{BFM05}: there, Coulomb branches for pure gauge theory (meaning ${\bf{N}} = 0$) are identified with variants of universal centralizers, and one can see that formal neighborhoods of identities in group-group and group-algebra universal centralizers are indeed isomorphic. For more algebraic evidence, there is an isomorphism of completions of Yangian and quantum loop group \cite{GTL13} (we explain the relation of these results to Coulomb branches in Section \ref{subsubseq: ade quivers and completions}), see also DAHA-type examples in \cite[Section~4]{BEF20}.

We prove a general result of this sort:
\begin{thmx}[See Theorem~\ref{thm_iso_after_completions}] \label{thm: intro isomorphism of completions}
For any $(G, {\bf{N}})$ there is an isomorphism of completions of K-theoretic and homological Coulomb branches, as algebras over $K^{G}(\pt)^{\wedge 1} \simeq H^{G}(\pt)^{\wedge 0}$:
\begin{equation} \label{eq: intro isomorphism of completions}
\Upsilon\colon K^{G(\cO)}(\cR_{G, {\bf{N}}})^{\wedge 1} \xrightarrow{\sim} H^{G(\cO)}(\cR_{G, {\bf{N}}})^{\wedge 0}.
\end{equation}
Here completions are taken over $1 \in \Spec K_{G}(\pt)$ and $0 \in \Spec H_{G}(\pt)$.

The same holds for quantized Coulomb branches and for their deformations over a flavor torus $F$.
\end{thmx}

Let's elaborate on the definition of $\Upsilon$. Note that the usual Chern character does not commute with direct image, while $\cR_{G,{\bf{N}}}$ is defined as a direct image of schemes of infinite type. 
So, a modification in style of Riemann--Roch theorem is needed. Since $\cR_{G, {\bf{N}}}$ is defined as an inductive limit of projective limits of singular varieties, work is required to make the construction work. This is done in Section \ref{sec: Riemann-Roch}. The main technical tools for this were developed by Edidin--Graham \cite{eg0, eg1, eg2}. 

We also write an explicit formula for the value of $\Upsilon$ on dressed minuscule monopole operators (Section~\ref{subsubsec: formulas for upsilon on monopoles}). This gives a full description of the isomorphism~(\ref{eq: intro isomorphism of completions}) in case when these elements generate the Coulomb branch algebra. This is the case when $G$ is a torus (Section~\ref{subsubsec: formulas for upsilon in abelian case}), and also when $(G, {\bf{N}})$ is a quiver gauge theory (Proposition~\ref{monopoles generate K-Coulomb branch}). It is worth emphasizing that the geometric origin of the map $\Upsilon$ implies its compatibility with the ``abelianization'' map relating Coulomb branch for $(G,{\bf{N}})$ with the Coulomb branch for $(T,{\bf{N}})$, where $T \subset G$ is a maximal torus as well as with the map relating the Coulomb branch for $(G,{\bf{N}})$ with the Coulomb branch for the pure gauge theory $(G, 0)$.  

Theorem \ref{thm: intro isomorphism of completions} may have  algebraic applications, since many interesting algebras appear as quantizations of Coulomb branches, see \cite{bfn_slices, FT19a, BEF20}.

We believe that this theorem has a straightforward generalization for Coulomb branches with symmetrizers \cite{NW23} and for parabolic Coulomb branches \cite[Definition~2.2]{KWWY24}, which may give more algebraic applications.

\subsubsection{Relation of Theorem \ref{thm: intro isomorphism of completions} to previous results} \label{them D vs other results}
Theorem~\ref{thm: intro isomorphism of completions} goes back to \cite[Theorems~5.11.11,~6.2.4]{CG97}, where the finite-dimensional non-equivariant case is considered. Namely, under certain assumptions, the following isomorphism of algebras is constructed in \cite{CG97}: 
\begin{equation}\label{cg result rr}
K(X \times_Y X) \iso H_*(X \times_Y X),
\end{equation}
where $X$ is a smooth variety mapping to some (singular) variety $Y$. The convolution product on these algebras is defined via the closed embedding $X \times_Y X \hookrightarrow X \times X$, using that $X$ is smooth.
In \cite[Sketch of proof of Theorem 12.7]{Gi98} 
Ginzburg considers an equivariant version of the map (\ref{cg result rr}) for  $X=\widetilde{\mathcal{N}}$ and $Y=\mathcal{N}$ and it seems clear that he had an equivariant version of \cite[Theorem~5.11.11]{CG97} in mind (see also~\cite{Lus89}). 
Let us point out that both definition of the convolution products and the map in \eqref{cg result rr} depend on the embedding of $X \times_Y X$ in $X \times X$ and rely on the fact that $X$ is smooth. The map (\ref{cg result rr}) is ``in between'' the Chern character $\on{ch}_{X \times X}$ and the Riemann--Roch map $\tau_{X \times X}$ for $X \times X$ (see \cite{BFM75}). Namely, the map (\ref{cg result rr}) is given by $(1 \boxtimes \on{Td}_X) \on{ch}_{X \times X} = (\on{Td}_X^{-1} \boxtimes 1)\tau $ (an alternative candidate is $(\sqrt{\on{Td}_X} \boxtimes \sqrt{\on{Td}_X}) \on{ch}_{X \times X}$); here $\on{Td}_X$ is the Todd class of the tangent bundle to $X$.



In the Coulomb branch setting, we are dealing with the $G_{\mathcal{O}} \rtimes \bC^\times$-equivariant $K$-theory/ homology of the space $\mathcal{R}_{G,{\bf{N}}}$. Informally, this should be considered as $G_{\mathcal{K}} \rtimes \bC^\times$-equivariant $K$-theory/homology of $X \times_Y X$ for $X=\mathcal{T}$ and $Y={\bf{N}}_{\mathcal{K}}$ (see \cite[Remarks 3.9]{BFN18}). The later spaces are ``too infinite dimensional'' to make sense of their  $K$-theory/homology (see \cite[Section 5.2]{CW23} for an alternative approach). So, the presentation of algebras we are dealing with is not as in \cite{CG97}. 
Because of this, literally \cite[Theorem 5.11.11]{CG97} and its equivariant analogs are not applicable in our situation.

Let's point out that from the perspective of our proof of Theorem \ref{thm: intro isomorphism of completions}, the realization via $\mathcal{R}_{G,{\bf{N}}}$ actually {\emph{simplifies}} the construction of $\Upsilon$ (the analog of (\ref{cg result rr}) above). Namely, for ${\bf{N}}=0$, the isomorphism $\Upsilon$ is just the equivariant version of the morphism $\tau$ (constructed in \cite{eg1}) without any additional Todd class of tangent bundle corrections. In general, the morphism $\Upsilon$ is given by $\tau$ times the Todd class of $\mathcal{T} \rightarrow \operatorname{Gr}_G$ pulled back to $\mathcal{R}_{G,{\bf{N}}}$.

To summarize, we define $\Upsilon$ using the results of Edidin--Graham and then check that it is a homomorphism of algebras via ``abelianization'' of Coulomb branches (\cite[Section~5]{BFN18}) by reducing all the computations to the case $G=T$, ${\bf{N}}=0$.
Along the way, we relate maps $\Upsilon$ for Coulomb branches for different gauge theories. 
It would be interesting to use the Coulomb branch definition of \cite[Section 5.2]{CW23} and adapt the proof of \cite[Theorem~5.11.11]{CG97} to this setting. 

Completions of $K$-theoretic Coulomb branches are also studied in \cite{VV25} (see Theorem 4.7 in \textit{loc. cit.}), the main difference between Theorem \ref{thm: intro isomorphism of completions} and their result is that the completion is taken w.r.t. {\emph{different}} ideals. In the context of \cite{VV25}, the fixed points of $\mathcal{R}, \mathcal{T}$ w.r.t. the corresponding semi-simple element are finite dimensional, so \cite[Theorem~5.11.11]{CG97} is applicable as stated. In Theorem \ref{thm: intro isomorphism of completions} one can specialize $q=1, \hbar=0$ while \cite[Theorem~4.7]{VV25} is a purely quantum statement.

 Finally, let us mention that the abelian case of Theorem~\ref{thm: intro isomorphism of completions} (for $\hbar=0$) is \cite[Theorem~5.3]{GaMcWe19} (we are grateful to Ben Webster for pointing this out to us).










\subsection{Further directions and generalizations}
Most of methods of this paper are applicable for Higgs and Coulomb branches associated to an arbitrary gauge theory $(G, {\bf{N}})$ (not necessarily of quiver type). 
One does not expect that Hikita conjecture literally holds in this generality (see \cite[Section 6]{HKM24}), however it is an interesting question to investigate the extents of its validity (or to invent suitable modifications).

It would be also interesting to investigate multiplicative version of Hikita conjecture for the nilpotent cone (a multiplicative version of the nilpotent cone is the variety of unipotent elements in $G$), and, more generally, affinizations of coverings of nilpotent orbits (see \cite{HKM24} where the usual additive case is discussed).

We also expect an elliptic version of Hikita conjecture to exist (see \cite{LZ22}), involving the $B$-algebra of elliptic Coulomb branches. Elliptic BFN Coulomb branches are expected to be the Coulomb branches of $5d$ $\cN = 1$ gauge theories, but are very poorly studied at the moment, see \cite[Section~4]{FMP20}.  

Finally, there is even more deep variant of the Hikita conjecture, the so called \textit{quantum Hikita conjecture}, proposed in \cite{KMP21}. It involves quantum cohomology of a quiver variety (in the guise of \textit{specialized quantum D-module}) and the \textit{D-module of graded traces} for the quantized Coulomb branch.

Note that in \cite[Remark~1.4]{KMP21}, the authors suggest that one should be able to adapt this conjecture, replacing quantum cohomology by the quantum K-theory, and suggest that it ``in many respects proved to be an even richer object''. Below we mention a conjectural statement of such adaptation (joint with H.~Dinkins and I.~Karpov).

Similarly to the non-quantum form above, one should replace (quantum) cohomology by the (quantum) K-theory on the quiver side (defined as in \cite{KPSZ21}), and replace homological Coulomb branch by the K-theoretic one on the Coulomb side.

More formally, one should consider the $(\hbar=q)$-specialization of the quantum K-theoretic $D$-module. Conjecture claims that this specialization should be equal to the D-module of graded traces for the quantized K-theoretic Coulomb branch. Moreover, the analog of the diagram (\ref{eq: intro surjections from K-theory of point to both sides}) still exists and our D-modules should be equal as quotients of the ``master'' $D$-module $K_{G_{\bf{v}} \times F \times \bC^\times}(\operatorname{pt})[[z]]$ (see also \cite[Remark 1.11]{BL25}). The $(\hbar=q)$-specializations of Okounkov's vertex functions with descendants to torus fixed points of $\widetilde{\mathfrak{M}}_Q$ should recover (normalized) graded traces of Verma modules over quantized K-theoretic Coulomb branches.

We do not know if our method for the non-quantum conjecture can be extended to the quantum case.

\subsection{The paper is organized as follows} 

In Section~\ref{section: homo hikita}, we study the homological Hikita conjecture. The main general result is Theorem~\ref{equivariant hikita from non-equivariant}. Particular cases of the conjecture are obtained in Corollaries~\ref{cor: hiktia for A-quivers},~\ref{homo-hikita for ADE quivers}. 

In Section~\ref{sec: Riemann-Roch}, we study the equivariant Riemann--Roch theorem in the context of Coulomb branches, and prove the main isomorphism of completions result, Theorem~\ref{thm_iso_after_completions}. Explicit formulae for this isomorphism are given in Section~\ref{subsec: formulas for upsilon}.

In Section~\ref{sec: K-hikita}, we introduce and study the K-theoretic version of Hikita conjecture. In Section~\ref{subsec: monopoles generate k-branches}, we discuss generators of K-theoretic Coulomb branches. The main general result on K-theoretic conjecture is Theorem~\ref{k-hikita from homo-hikita}. It is derived for some particular cases in Corollaries~\ref{cor: k-hikita for ade},~\ref{cor: k-hikita jordan quiver}. Some corollaries of this conjecture are discussed in Section~\ref{subsec: corollaries of K-Hikita}.

In Appendix~\ref{appendix: ade}, we study the homological conjecture in types ADE by direct geometric analysis. The non-equivariant conjecture for generalized slices is proved in Theorem~\ref{thm: hikita for generalized slices}. The equivariant version is proved in Theorem~\ref{Equivariant Hikita conjecture for ADE quivers}.


\subsection*{Acknowledgments}
We are indebted to Alexander Braverman, who first explained to us that K-theoretic version of Hikita conjecture in this context should exist. We thank Michael Finkelberg for useful discussions, 
Joel Kamnitzer and Alex Weekes for valuable comments and for sharing their unpublished note, which helped us with the proof of Corollary~\ref{cor: fiber of b-algebra homo},
and Dinakar Muthiah for valuable discussions on Section~\ref{subsec: monopoles generate k-branches}. We are grateful to Hiraku Nakajima for valuable comments on an earlier version of this text and for suggesting relevant references. The second named author is supported by the Simons
Foundation Award 888988 as part of the Simons Collaboration on Global Categorical Symmetries.

\section{Homological Coulomb branches and homological Hikita conjecture} \label{section: homo hikita}

Throughout the paper, for a scheme $X$ with $G$-action, we use the notations $H_G(X) = H_G^*(X)$ for equivariant cohomology, $H^G(X) = H^G_*(X)$ for equivariant Borel--Moore homology, $K_G(X) = K_0(\Vect^G(X))$ for equivariant K-cohomology (K-theory), and $K^G(X) = K_0(\Coh^G(X))$ for equivariant K-homology (G-theory, sometimes voluntarily also called equivariant K-theory).

\subsection{Homological Hikita conjecture} \label{subsection: homological Hikita}

Let $Q$ be a finite oriented quiver, and $Q_0$ its set of vertices. 
Let $\bv = \{v_i\}_{i \in Q_0}$ be  dimension vector, $\bw = \{w_i \}_{i \in Q_0}$ be the framing vector.
We associate to each vertex $i \in Q_0$ the vector space $V_i$, $\dim V_i = v_i$, and the framing space $W_i$, $\dim W_i = w_i$.
We denote 
\[
{\bf{N}} = \bigoplus_{i \rightarrow j} \Hom (V_{i}, {V_{j}}) \oplus \bigoplus_{i \in Q_0} \Hom( {V_i}, {W_i}),
\]
$G_\bv = \prod_{i \in Q_0} GL_{v_i}$, $G_\bw = \prod_{i \in Q_0} GL_{w_i}$, let $T_\bv \subset G_\bv$, $T_\bw \subset G_\bw$ be   maximal tori, and let $S_\bv$, $S_\bw$ be the Weyl groups of $G_\bv$ and $G_\bw$ respectively. 
Let $\ft_\bv$, $\ft_\bw$ be the Lie algebras of $T_\bv$ and $T_\bw$ respectively.
Let a torus $F$ act on ${\bf{N}}$, such that its action commutes with $G_\bv$. We call $F$ the flavor torus. For example, one can take $F = T_\bw$, but if $Q$ has loops or multiple edges, $F$ may be chosen larger, see \cite[9.5.(i)]{BLPW14}. Denote $\ff = \Lie F$.

We can associate a pair of symplectic singularities to this data -- the Nakajima quiver variety $\fM(G_\bv, {\bf{N}})$ and the BFN Coulomb branch $\cM(G_\bv, {\bf{N}})$.

In many cases, affine quiver variety $\fM(G_\bv, {\bf{N}})$ can be resolved by the smooth quiver variety $\wti \fM(G_\bv, {\bf{N}})$. Variety $\wti \fM(G_\bv, {\bf{N}})$ depends on the choice of a regular character $\nu$ of $G_\bv$, and is defined the Hamiltonian GIT-reduction $({\bf{N}} \oplus {\bf{N}}^*) /\!\!/\!\!/^\nu G_\bv$. 
We fix $\nu$ and hence $\wti \fM(G_\bv, {\bf{N}})$. Variety $\wti \fM(G_\bv, {\bf{N}})$ admits the action of $F$, as well as the contracting action of torus, which we denote $\torus_\hbar$ (see \cite[Section~2.7]{Nak01a}). When the context is clear, we may denote $\wti \fM(G_\bv, {\bf{N}})$ just as $\wti \fM_Q$.

Let $\cK = \bC((t))$, $\cO = \bC[[t]]$. For a group $G$ we denote its affine Grassmannian $\Gr_G = G_\cK/ G_\cO$. The space $\Gr_G$ is stratified by smooth $G_{\cO}$-orbits $\Gr^\la$, parametrized by dominant coweights~$\la$; their closures $\ol \Gr^\la$ are called affine Schubert varieties.
Recall the BFN space of triples $\cR_{G_\bv, {\bf{N}}}$, associated with the group $G_\bv$ and its representation ${\bf{N}}$. It is defined by the Cartesian diagram
\begin{equation} \label{definition of bfn space of triples}
\begin{tikzcd}
	{\mathcal R_{G_{\bold v}, {\bf{N}}}} && {{\bf{N}}_{\mathcal O}} \\
	{G_{\bold v,\mathcal{K}} \times^{G_{\bv,\mathcal O}} {\bf{N}}_\mathcal O} && {{\bf{N}}_\mathcal K},
	\arrow[from=1-1, to=1-3]
	\arrow[hook, from=1-1, to=2-1]
	\arrow[hook, from=1-3, to=2-3]
	\arrow[from=2-1, to=2-3]
\end{tikzcd}
\end{equation}
see \cite{BFN18} for details. Its equivariant Borel--Moore homology and equivariant K-theory possess a natural multiplication structure. Homological Coulomb branch $\cM(G_\bv, {\bf{N}})$ is defined as $\Spec H^{G_\bv} (\cR_{G_\bv, {\bf{N}}})$. The K-theoretic Coulomb branch is defined as $\Spec K^{G_\bv}(\cR_{G_\bv, {\bf{N}}})$. In this Section, we deal with homological version, and return to K-theoretic one later. 

There is a Poisson deformation $\cM(G_\bv, {\bf{N}})_{\ff}$ of $\cM(G_\bv, {\bf{N}})$ over $\ff$, which is defined as $\Spec H^{G_\bv \times F} (\cR_{G_\bv, {\bf{N}}})$. 
Algebra $H^{G_\bv \times F} (\cR_{G_\bv, {\bf{N}}})$ is graded by $\pi_0(\cR_{G_\bv}) = \pi_1(G_\bv) = \bZ^{|Q_0|}$, we denote the corresponding (Hamiltonian) torus, acting on $\cM_{Q, F}$ by $H \simeq (\bC^\times)^{|Q_0|}$.

There is also an action of the multiplicative group on $\cM(G_\bv, {\bf{N}})_{\ff}$, coming from the homological grading of $H^{G_\bv \times F} (\cR_{G_\bv, {\bf{N}}})$. Sometimes, it is conical (then the corresponding gauge theory is called ``good or ugly''), but we do {\underline{not}} assume it here. We denote this torus by $\bC^\times_{\textrm{hom}}$, see \cite[3(v)]{BFN18}.

There is a quantization of $\cM(G_\bv, {\bf{N}})_{ \ff}$, which we denote $\cA(G_\bv, {\bf{N}})_{\ff}$. It is defined as $\cA(G_\bv, {\bf{N}})_{\ff} = H^{G_\bv \times F \times \bC^\times_{\hbar}}_* (\cR_{G_\bv, {\bf{N}}})$, where $\bC^\times_\hbar$ acts by the loop rotation.

When no confusion arise, we denote $\cM(G_\bv, {\bf{N}})$ by $\cM_Q$ and similarly to other varieties and algebras.

Recall that we picked a character $\nu$ of $G_\bv$, or equivalently a cocharacter of $H$ to be denoted by the same symbol.
For an algebra $A$, acted by a torus through a character $\nu$, its \textit{B-algebra} (also known as {\emph{Cartan subquotient}}) is defined as
\begin{equation*}
B^\nu(A) = A_0 \left/ \sum_{n \in \bZ_{<0}} A_n \cdot A_{-n} \right.,
\end{equation*}
where $A_n$ denotes the $n$-weight subspace of $\nu$-action. Note that when $A$ is commutative, this is equivalent to $B^\nu(A) = A \left/ ( A_{> 0} + A_{< 0} ) \right.$, so $B$-algebra is a quotient of $A$ (not just a subquotient), and it is nothing else but the algebra of functions on schematic fixed points under the $\nu$-action on $\Spec A$.

We now recall: 

\begin{conj}[Homological Hikita conjecture] \label{hikita conjecture}
There is an isomorphism of graded $\bC[(\ft_\bv / S_\bv) ~\times~\ff] \T \bC[\hbar]$-algebras
\begin{equation} \label{quantized hikita conjecture}
H^*_{F \times \torus_\hbar} (\wti \fM_Q) \simeq B^\nu(\cA_{Q, \ff}).
\end{equation} 
In particular, specializing at $\hbar = 0$, there is an isomorphism of graded $\bC[(\ft_\bv/S_\bv) \times \ff]$-algebras
\begin{equation} \label{deformed hikita conjecture}
H^*_{F} (\wti \fM_Q) \simeq \bC[\cM_{Q, \ff}^{\nu}].
\end{equation}
Further specializing at $0 \in \ff$, there is an isomorphism of graded $\bC[\ft_\bv / S_\bv]$-algebras
\begin{equation} \label{nonquantized nondeformed hikita}
H^* (\wti \fM_Q) \simeq \bC[\cM_{Q}^{\nu}].
\end{equation}
\end{conj}
Parameter $\hbar$ here should be thought of as a coordinate on $\Lie \torus_\hbar$, $\cM_{Q, \ff}^{\nu}$ stands for the schematics fixed points of $\cM_{Q, \ff}$ under the action of $\nu$.
The grading on the LHS is the cohomological grading, and the grading on the RHS comes from the $\nu$-commuting $\bC^\times_{\textrm{hom}}$-action on $\cM_{Q, \ff}$.

Historically, \eqref{nonquantized nondeformed hikita} is the version originally proposed by Hikita \cite{Hik17}, and \eqref{quantized hikita conjecture} is the strengthening, proposed by Nakajima. 
We refer to \eqref{quantized hikita conjecture} as to {\it quantized Hikita conjecture} (this should not be confused with \cite{KMP21}, where quantum cohomology are considered),  to \eqref{deformed hikita conjecture} as to {\it equivariant Hikita conjecture}, and to \eqref{nonquantized nondeformed hikita} as to {\it Hikita conjecture}.
In this paper, we do not consider the quantized version, and deal with \eqref{deformed hikita conjecture}, \eqref{nonquantized nondeformed hikita}.

Let us now comment on where does the $\bC[(\ft_\bv/S_\bv) \times \ff]$-action, mentioned in the statement of \eqref{deformed hikita conjecture}, come from on both sides. For the LHS, we recall that  $\wti \fM_Q$ is the GIT-quotient of $\mu_{G_\bv, {\bf{N}}}^{-1}(0)$ by the group $G_\bv$, where $\mu_{G_\bv, {\bf{N}}}\colon {\bf{N}} \oplus {\bf{N}}^* \rightarrow \g_\bv$ is the moment map. It is isomorphic to the geometric quotient of the stable locus $\mu_{G_\bv, {\bf{N}}}^{-1}(0)^{s}$ by the free action of $G_\bv$. Thus, the LHS of \eqref{deformed hikita conjecture} can be rewritten as $H^*_{F \times G_\bv} (\mu_{G_\bv, {\bf{N}}}^{-1}(0)^{s})$,
which is clearly a module over $H^*_{F \times G_\bv}(\pt) = \bC[(\ft_\bv / S_\bv) \times \ff]$. 

The RHS of \eqref{quantized hikita conjecture} $H^{G_\bv \times F }_* (\cR_{G_\bv, {\bf{N}}})$ is also clearly a module over $H_{G_\bv \times F}^* (\pt) = \bC[(\ft_\bv / S_\bv) \times \ff]$.

Note that the homomorphism 
\begin{equation} \label{surjection of homology of point to quiver}
\phi_1\colon \bC[(\ft_\bv / S_\bv) \times \ff] \twoheadrightarrow H^*_{F} (\wti \fM_Q)
\end{equation}
is surjective by \cite{MN18} (this is the so-called {\it Kirwan surjectivity}). 
The homomorphism 
\begin{equation} \label{surjection from homology of point to b-algebra}
\phi_2\colon \bC[(\ft_\bv / S_\bv) \times \ff] \twoheadrightarrow \bC[\cM_{Q, \ff}^H]
\end{equation}
is also surjective (see \cite[Proposition 8.7]{KS25} and references therein).



Thus, Conjecture~\eqref{deformed hikita conjecture} is equivalent to the claim $\ker \phi_1 = \ker \phi_2$.

\begin{rem} \label{rem: hikita implies fixed point is unique}
Note that $H^*(\widetilde{\mathfrak{M}}_Q)=H^*(\mu_{G_\bv, {\bf{N}}}^{-1}(0)^{s}/G_{\bf{v}})$ considered as a module over $\bC[\mathfrak{t}_{\bf{v}}/S_{\bf{v}}]$ is either zero (if $\widetilde{\mathfrak{M}}_Q = \varnothing$) or supported at the point $\{0\}$ (this follows from the fact that the action $G_{\bf{v}} \curvearrowright \mu_{G_\bv, {\bf{N}}}^{-1}(0)^{s}$ is free). So, assuming that isomorphism (\ref{nonquantized nondeformed hikita}) holds, we conclude that $\mathcal{M}_Q^{\nu}$ as a set is empty if    $\widetilde{\mathfrak{M}}_Q = \varnothing$ and is a single point otherwise. For quivers without loops this is precisely \cite[Conjecture 3.25(1)]{bfn_slices} so the Hikita conjecture should be considered as an ``upgraded'' version of this conjecture. 
\end{rem}



\subsection{Localization of  Hikita conjecture}\label{subsec: localization of homo}


For a maximal ideal $\fm \subset A$ and an $A$-module $M$, we denote by $M_{\fm}$ the localization of $M$ at $\fm$. $M^{\wedge \fm}$ denotes the completion of $M$ at $\fm$.

We restrict our attention to the non-quantum case of the homological Hikita conjecture~\eqref{deformed hikita conjecture}.
As explained in Section \ref{subsection: homological Hikita}, both sides of \eqref{deformed hikita conjecture} are modules over $H_{G_\bv \times F}(\pt) = \bC[(\ft_\bv / S_\bv) \times \ff]$. In this Section, we study the localizations of both sides of \eqref{deformed hikita conjecture} at a point of $H_{G_\bv \times F}(\pt)$, as well as over its subalgebra $H_{F}(\pt)$.

We begin with the quiver variety side.

\begin{prop} \label{localization of quiver side of homological hikita}
For any $(t_\bv, f) \in (\ft_\bv / S_\bv) \times \ff$, there is an isomorphism of algebras
\begin{equation*}
H_{F} (\wti \fM(G_\bv, {\bf{N}}))_{(t_\bv, f)} \simeq H_{F} (\wti \fM({Z_{G_\bv}(t_\bv), {\bf{N}}^{(t_\bv, f)}}))_{(0, 0)}.
\end{equation*}

\end{prop}
It is clear that $Z_{G_\bv}(t_\bv)$ is a product of general linear groups, but it is not immediately clear that ${\bf{N}}^{(t_\bv, f)}$ is its representation, coming from some (framed) quiver. This is indeed the case, as shown in Proposition ~\ref{prop: fixed points fo quiver for arbitrary F} below. For the moment, we simply define
\begin{equation*}
\wti \fM({Z_{G_\bv}(t_\bv), {\bf{N}}^{(t_\bv, f)}}) = ({\bf{N}}^{(t_\bv, f)} \oplus ({\bf{N}}^{(t_\bv, f)})^*) /\!\!/\!\!/^\nu Z_{G_\bv}(t_\bv).
\end{equation*}

\begin{proof}[Proof of Proposition \ref{localization of quiver side of homological hikita}]
By the localization theorem for equivariant cohomology, we get
\begin{equation} \label{localization theorem for quiver variety}
H_{F} (\fM(G_\bv, {\bf{N}}))_{(t_\bv, f)} = H_{F \times G_\bv} (\mu_{G_\bv, {\bf{N}}}^{-1}(0)^{s})_{(t_\bv, f)} = H_{F \times Z_{G_\bv}(t_\bv)} ((\mu_{G_\bv, {\bf{N}}}^{-1}(0)^{s})^{(t_\bv, f)})_{ (0,0)}.
\end{equation}

The rest of the proof is the computation of the torus-fixed points, similar to \cite[Lemma~3.2]{Nak01b}.

First, let us analyze $(\mu_{G_\bv, {\bf{N}}}^{-1}(0))^{(t_\bv, f)}$ (without taking the stable locus). It is the (scheme-theoretic) intersection of 
\[
({\bf{N}} \oplus {\bf{N}}^*)^{(t_\bv, f)} = ({\bf{N}}^{(t_\bv, f)} \oplus ({\bf{N}}^{(t_\bv, f)})^*)\]
with $\mu_{G_\bv, {\bf{N}}}^{-1}(0)$.  Note that ${\bf{N}}^{(t_\bv, f)} \oplus ({\bf{N}}^{(t_\bv, f)})^*$ carries a Hamiltonian action of the group $Z_{G_\bv}(t_\bv)$, and the diagram
\[\begin{tikzcd}
	{{\bf{N}} \oplus {\bf{N}}^*} &&& {\mathfrak g_{\bold v}} \\
	{{\bf{N}}^{(t_{\bold v}, t_{f})} \oplus ({\bf{N}}^{(t_{\bold v}, t_{f})})^*} &&& {\operatorname{Lie} Z_{G_{\bold v}}(t_{\bold v})}
	\arrow["{\mu_{G_{\bold v}, {\bf{N}}}}", from=1-1, to=1-4]
	\arrow[hook, from=2-1, to=1-1]
	\arrow["{\mu_{Z_{G_{\bold v}}(t_{\bold v}), {\bf{N}}^{(t_{\bold v}, t_{f})}}}", from=2-1, to=2-4]
	\arrow[hook, from=2-4, to=1-4]
\end{tikzcd}\]
is commutative (here horizontal morphisms are the moment maps, and the bottom part is obtained from the top part by taking $(t_\bv, f)$-invariants). It follows that 
\[
(\mu_{G_\bv, {\bf{N}}}^{-1}(0))^{(t_\bv, f)} = \mu_{Z_{G_{\bold v}}(t_{\bold v}), {\bf{N}}^{(t_{\bold v}, t_{f})}}^{-1}(0).
\]
We claim that this isomorphism restricts to the isomorphism of stable loci:
\begin{equation} \label{isomorphism of stable loci}
(\mu_{G_\bv, {\bf{N}}}^{-1}(0)^s)^{(t_\bv, f)} = \mu_{Z_{G_{\bold v}}(t_{\bold v}), {\bf{N}}^{(t_{\bold v}, t_{f})}}^{-1}(0)^s,
\end{equation}
where on the left-hand side we mean the stability condition corresponding to  $\nu\colon G_\bv \rightarrow \bC^\times$, and on the right-hand side we mean the stability condition for  the restriction $\nu|_{Z_{G_\bv}(t_\bv)}$.
The equality (\ref{isomorphism of stable loci}) follows directly from the combinatorial description of stability \cite[Definition~2.7]{Nak01b}, and is implicit in the proof of \cite[Lemma~3.2]{Nak01b}.

Combining \eqref{localization theorem for quiver variety} and \eqref{isomorphism of stable loci}, we get:
\begin{equation*}
H_{F} (\fM(G_\bv, {\bf{N}}))_{ (t_\bv, f)} \simeq H_{F \times Z_{G_\bv}(t_\bv)} (\mu_{Z_{G_{\bold v}}(t_{\bold v}), {\bf{N}}^{(t_{\bold v}, f)}}^{-1}(0)^s)_{ (0,0)} \simeq H_{F} (\fM({Z_{G_\bv}(t_\bv), {\bf{N}}^{(t_\bv, f)}}))_{(0, 0)},
\end{equation*}
which finishes the proof.
\end{proof}

\begin{rem}\label{rem:twist iso fibers}
The RHS of Proposition \ref{localization of quiver side of homological hikita} is the localization at $(0,0)$. In the course of the proof, we use that this localization naturally identifies with the localization of the same space at $(t_{\mathbf{v}},f)$. This identification is standard (see, for example, \cite[Section 5.2]{eg2} in the $K$-theoretic setting), and we will occasionally refer to it as a twisting. 
\end{rem}

We now describe fibers of the same space, but as a module over $H_{F}(\pt)$, instead of $H_{F \times G_\bv}(\pt)$. For $f \in \ff$, denote by $\bC_f$ the one-dimensional module over $\bC[\ff]$ --- quotient by the maximal ideal, corresponding to $f$.

\begin{cor} \label{localization of quiver side of homological hikita over T_w}
For any $f \in \ff$, there is an isomorphism of algebras
\begin{equation*}
H_{F} (\wti \fM(G_\bv, {\bf{N}})) \T_{H_F(\pt)} \bC_{f} \simeq \bigoplus_{t_\bv \in \ft_\bv / S_\bv} H^* (\wti \fM({Z_{G_\bv}(t_\bv), {\bf{N}}^{(t_\bv, f)}})),
\end{equation*}
and only finite number of summands is nonzero.
\end{cor}

\begin{proof}

We have $\bC[\ff] \subset \bC[\ff \times (\ft_\bv / S_\bv)]$, and the desired fiber over $f$ is a module over $\bC[\ft_\bv / S_\bv]$. Note that $H_{F} (\wti \fM(G_\bv, {\bf{N}}))$ is finitely generated over $H_F(\pt)$, 
hence its fiber at $f$ is supported at a finite number of points of $\ft_\bv / S_\bv$.

Thus, the required fiber is isomorphic to the direct sum over all points of $\ft_\bv / S_\bv$ of its formal completions at these points. Hence, we obtain 
\begin{multline*}
H_{F} (\wti \fM(G_\bv, {\bf{N}})) \T_{H_F(\pt)} \bC_{f} \simeq \bigoplus_{t_\bv \in \ft_\bv / S_\bv} \big(H_{F} (\wti \fM(G_\bv, {\bf{N}})) \T_{H_F(\pt)} \bC_{f} \big)^{\wedge (t_\bv, f)} \\
\simeq \bigoplus_{t_\bv \in \ft_\bv / S_\bv} H^* (\wti \fM({Z_{G_\bv}(t_\bv), {\bf{N}}^{(t_\bv, f)}}))^{\wedge t_\bv} \simeq \bigoplus_{t_\bv \in \ft_\bv / S_\bv} H^* (\wti \fM({Z_{G_\bv}(t_\bv), {\bf{N}}^{(t_\bv, f)}})),
\end{multline*}
where the second isomorphism follows by Proposition~\ref{localization of quiver side of homological hikita} (namely its {\emph{untwisted}} version when we do not pass from the fiber over $(t_{\bf{v}},f)$ to the fiber over $(0,0)$, see Remark \ref{rem:twist iso fibers}), and the last one is justified as follows: a finite-dimensional algebra over $\bC[\ft_\bv/S_\bv]$, supported at $t_\bv$, is isomorphic to its completion at this point.
\end{proof}

We now turn to the Coulomb branch side. We first prove the following fact about the affine Grassmannian.


\begin{lem} \label{fixed points on affine grassmannian}
Let $G$ be a connected reductive group. 
For any semi-simple $t \in G$, one has an isomorphism of reduced ind-schemes:
\begin{equation*}
((\Gr_G)^t)_\red = (\Gr_{Z_G(t)})_\red.
\end{equation*}
\end{lem}

Note that for $t$ being a generic element of a cocharacter of $G$, this is well-known, even without taking reduced parts (see, e.g., \cite[Proposition~3.4]{HR21}).
We need it further for any semi-simple element of $G$ (for example, of finite order), so we include a proof.    

\begin{proof}
Clearly, we have a closed embedding  $(\Gr_{Z_G(t)})_\red \hookrightarrow ((\Gr_G)^t)_\red$. To prove that it is an isomorphism it is enough to prove that it is bijective on $\bC$-points.  

Recall the decomposition of $\operatorname{Gr}_G(\mathbb{C})$ into the disjoint union:
\begin{equation*}
\operatorname{Gr}_G(\mathbb{C}) = \bigsqcup_{\eta \in \Lambda} (U((z)) \cdot z^\eta)(\mathbb{C}),
\end{equation*}
where $U \subset G$ is the unipotent radical of a fixed Borel $B \subset G$. Orbits $U((z)) \cdot z^\eta$ are called {\emph{semi-infinite}} orbits and we have a canonical identification:
\begin{equation*}
 U((z)) \cdot z^\eta = z^\eta U((z))/U[[z]] = z^\eta U[z^{-1}]_1,
\end{equation*}
where $U[z^{-1}]_1 \subset U[z^{-1}]$ is the kernel of the ``evaluation at infinity'' homomorphism $U[z^{-1}] \rightarrow U$.

So, it remains to compute $t$-fixed points on each $z^\nu U[z^{-1}]_1$. The action of $t$ on $z^\eta U[z^{-1}]_1$ is given by $t \cdot (z^\eta u) = z^\eta \operatorname{Ad}_t(u)$. Set $\mathfrak{u}=\operatorname{Lie}U$. We have an exponent isomorphism $\operatorname{exp}\colon z^{-1}\mathfrak{u}[z^{-1}] \iso U[z^{-1}]_1$. The adjoint action of $t$ on $U[z^{-1}]_1$ corresponds to the adjoint $t$-action on $z^{-1}\mathfrak{u}[z^{-1}]$. Clearly:
\begin{equation*}
(z^{-1}\mathfrak{u}[z^{-1}])^{\operatorname{ad}(t)} = \bigoplus_{\alpha(t)=1}z^{-1}\mathfrak{u}_\alpha[z^{-1}]=z^{-1}\operatorname{Lie}(U \cap Z_G(t))[z^{-1}],
\end{equation*}
where $\al$ runs over roots of $\Lie G$.
Exponentiating we conclude that 
\begin{equation*}
(U[z^{-1}]_1)^t \subset (U \cap Z_G(t))[z^{-1}]_1,
\end{equation*}
so $z^\eta (U[z^{-1}]_1)^t \subset z^\eta (U \cap Z_G(t))[z^{-1}]_1$, hence, $((U((z)) \cdot z^\eta)(\mathbb{C}))^t \subset \operatorname{Gr}_{Z_G(t)}(\mathbb{C})$ as desired.
\end{proof}

We now turn to the BFN space of triples.

\begin{lem} \label{fixed points on BFN space}
For any $(t_\bv, f) \in T_\bv \times F$, one has an isomorphism of reduced ind-schemes
\begin{equation*}
((\mathcal R_{G_{\bold v}, {\bf{N}}})^{(t_{\bold v}, {f})})_\red = (\cR_{Z_{G_\bv}(t_\bv), {\bf{N}}^{(t_\bv, f)}})_\red
\end{equation*}
\end{lem}

For $t_\bv$ being a generic element of a cocharacter of $T_\bv$ and $F = \id$ this is \cite[Lemma~5.1]{BFN18}. The proof of general case is the same, using Lemma \ref{fixed points on affine grassmannian}.

\begin{proof}

As before, we have a closed embedding and verify the required isomorphism on $\mathbb{C}$-points.
In diagram \eqref{definition of bfn space of triples}, all maps are naturally $(G_\bv(\cO) \times F)$-equivariant. Let us investigate what are the fixed points:
\[
(G_{\bold v}(\mathcal K) \times^{G_\bv(\mathcal O)} {\bf{N}}(\mathcal O))^{(t_\bv, f)} \hookrightarrow (\Gr_{G_\bv} \times {\bf{N}}(\cK))^{(t_\bv, f)}
\]
(the embedding is given by $(g, n) \mapsto ([g], gn)$ and is  $G_\bv(\cO) \times F$-equivariant).
Note that in the last product, $\Gr_{G_\bv}$ is acted by $G_\bv(\cO)$ only, while ${\bf{N}}(\cK)$ is acted by both factors of $G_\bv(\cO) \times F$. Using Lemma \ref{fixed points on affine grassmannian}, we have: 
\[
(\Gr_{G_\bv} \times {\bf{N}}(\cK))^{(t_\bv, f)} = \Gr_{Z_{G_\bv}(t_\bv)} \times {\bf{N}}^{(t_\bv, f)}(\cK) 
\]
up to taking reduced parts. 

From here, one easily sees, that there is a commutative diagram of $(t_\bv, f)$-invariants of~\eqref{definition of bfn space of triples}:
\[\begin{tikzcd}
	{(\mathcal R_{G_\bv, {\bf{N}}})^{(t_{\bold v}, {f})}} && {{\bf{N}}^{(t_{\bold v}, {f})}(\mathcal O)} \\
	{(Z_{G_\bv}(t_{\bold v})(\mathcal K) \times^{Z_{G_\bv}(t_{\bold v})(\mathcal O)} {\bf{N}}^{(t_{\bold v}, {f})}(\mathcal O)} && {{\bf{N}}^{(t_{\bold v}, {f})}(\mathcal K)}
	\arrow[from=1-1, to=1-3]
	\arrow[hook, from=1-1, to=2-1]
	\arrow[hook, from=1-3, to=2-3]
	\arrow[from=2-1, to=2-3].
\end{tikzcd}\]
It follows that $(\mathcal R_{G_\bv, {\bf{N}}})^{(t_{\bold v}, {f})} = \cR_{Z_{G_\bv}(t_\bv), {\bf{N}}^{(t_\bv, f)}}$ (up to taking reduced parts) by definition of the latter.
\end{proof}


\begin{prop} \label{localization of homological coloumb branch}
For any $(t_\bv, f) \in (\ft_\bv / S_\bv) \times \ff$, there is an isomorphism of algebras, localized over $\bC[(\ft_\bv / S_\bv) \times \ff]$:
\begin{equation*}
\bC[\cM(G_\bv, {\bf{N}})_{\ff}]_{(t_\bv, f)} \simeq \bC[\cM(Z_{G_\bv}(t_\bv), {\bf{N}}^{(t_\bv, f)})_{\ff}]_{(0, 0)}
\end{equation*}
\end{prop}

\begin{proof}
By the localization theorem for equivariant homology and Lemma \ref{fixed points on BFN space}, we have:
\begin{multline*}
\bC[\cM(G_\bv, {\bf{N}})_{\ff}]_{(t_\bv, f)} =
H^{G_\bv \times F} (\cR_{G_\bv, {\bf{N}}})_{ (t_\bv, f)} 
= H^{Z_{G_\bv}(t_\bv) \times F} ((\cR_{G_\bv, {\bf{N}}})^{(t_\bv, f)})_{(0, 0)}=\\
= H^{Z_{G_\bv}(t_\bv) \times F} (\cR_{Z_{G_\bv}(t_\bv), {\bf{N}}^{(t_\bv, f)}})_{(0, 0)}
= \bC[\cM(Z_{G_\bv}(t_\bv), {\bf{N}}^{(t_\bv, f)})_{\ff}]_{(0, 0)}.
\end{multline*}
As it is shown in \cite[Section~5]{BFN18}, localization commutes with multiplication in Coulomb branches, hence this is indeed an isomorphism of algebras.
\end{proof}

From this, the localization of schematic fixed points algebra can be described
\begin{cor} \label{cor: localization of b-algebra}
For any $(t_\bv, f) \in (\ft_\bv / S_\bv) \times \ff$ there is an isomorphism of algebras, localized over $\bC[(\ft_\bv / S_\bv) \times \ff]$:
\begin{equation*}
\bC[\cM(G_\bv, {\bf{N}})_{\ff}^\nu]_{(t_\bv, f)} \simeq \bC[\cM(Z_{G_\bv}(t_\bv), {\bf{N}}^{(t_\bv, f)})_{\ff}^\nu]_{(0, 0)}.
\end{equation*}
\end{cor}
\begin{proof}
Recall that the action of $\nu$ on $\cM(G_\bv, {\bf{N}})_{\ff}$ comes from the connected components decomposition of $\cR_{G_\bv, {\bf{N}}}$. Hence, $\nu$ acts trivially on $\bC[(\ft_\bv/S_\bv) \times \ff]$, and taking $\nu$-fixed points commutes with taking localization 
(note also that the localization functor is exact). Similarly for $\cM(Z_{G_\bv}(t_\bv), {\bf{N}}^{(t_\bv, f)})_{\ff}$. Now the results follow from Proposition~\ref{localization of homological coloumb branch}.
\end{proof}

Now we describe the fiber of the schematic fixed points algebra over a point of $\bC[\ff]$. 

\begin{cor} \label{cor: fiber of b-algebra homo}

Suppose $\cM(G_\bv, {\bf{N}})^\nu(\bC)$ is finite\footnote{In fact, this assumption is always satisfied, as follows from an unpublished work of Kamnitzer--Weekes. The argument is not very complicated.}. 
Then for any $f \in \ff$, there is an isomorphism of algebras
\begin{equation*}
\bC[\cM(G_\bv, {\bf{N}})_{\ff}^\nu] \T_{\bC[\ff]} \bC_{f} \simeq \bigoplus_{t_\bv \in \ft_\bv / S_\bv} \bC[\cM(Z_{G_\bv}(t_\bv), {\bf{N}}^{(t_\bv, f)})_{\ff}^\nu],
\end{equation*}
and only finite number of summands is nonzero.
\end{cor}
\begin{proof}
We argue as in the proof of Corollary~\ref{localization of quiver side of homological hikita over T_w}, using Proposition~\ref{localization of homological coloumb branch}.

Our assumption implies that fiber at $0 \in \ff$ of $\bC[\cM(G_\bv, {\bf{N}})_{\ff}^\nu]$ is finite-dimensional, hence the graded Nakayama lemma implies that fiber over any point $f$ is finite-dimensional. Note that fiber over $f$ is a module over $\bC[\ft_\bv]$, hence it is supported at a finite number of points. Thus, it is isomorphic to the direct sum of formal completions over all points, ad we get:
\begin{multline*}
\bC[\cM(G_\bv, {\bf{N}})_{\ff}^\nu] \T_{\bC[\ff]} \bC_{f} \simeq \bigoplus_{t_\bv \in \ft_\bv / S_\bv} \big( \bC[\cM(G_\bv, {\bf{N}})_{\ff}^\nu] \T_{\bC[\ff]} \bC_{f} \big)^{\wedge (t_\bv, f)} \\ 
\simeq \bigoplus_{t_\bv \in \ft_\bv / S_\bv} \bC[\cM(Z_{G_\bv}(t_\bv), {\bf{N}}^{(t_\bv, f)})_{\ff}^\nu]^{\wedge t_\bv} \simeq \bigoplus_{t_\bv \in \ft_\bv / S_\bv} \bC[\cM(Z_{G_\bv}(t_\bv), {\bf{N}}^{(t_\bv, f)})_{\ff}^\nu],
\end{multline*}
where the second isomorphism follows from Proposition~\ref{localization of homological coloumb branch}, and the last is justified in the same way as in the proof of Corollary~\ref{localization of quiver side of homological hikita over T_w}.
\end{proof}

\begin{rem}
As we defined above, the cocharacter of the Hamiltonian torus $\nu$ comes from a character of $G_\bv$ corresponding to generic stability condition. At a fiber over some $(t_\bv, f) \in T_\bv \times F$, it descends to a character of $Z_{G_\bv}(t_\bv)$. 
We expect that the corresponding cocharacter of the Hamiltonian torus is generic. One distinguished choice for $\nu$ is the one corresponding to the character of $G_{\bf{v}}$ given by the product of determinants (compare with \cite[Section~3(viii)]{bfn_slices} where it is denoted by $\chi$). 
This particular character $\nu$ is always generic. This follows from the fact that positive and negative dressed minuscule monopole operators generating $\mathbb{C}[\mathcal{M}(G_{\bf{v}},{\bf{N}})]$ over $H^*_{G_{\bf{v}}}(\operatorname{pt})$ have non-zero degree w.r.t $\nu$ (see Proposition~\ref{monopoles generate K-Coulomb branch} below and \cite[proof of Proposition~3.1]{Wee19}). 
\end{rem}

\subsection{Fixed points of a quiver theory} \label{subsec: fixed points of quiver theory}
We see that on both sides it is the representation ${\bf{N}}^{(t_\bv, f)}$ of the group $Z_{G_\bv}(t_\bv)$ that appears in localizations of Hikita conjecture, so let us investigate this representation in terms of quiver $Q$.

We first assume that $F = T_\bw$.
Take $(t_\bv, t_\bw) \in  T_\bv \times T_\bw$.
Its action on ${\bf{N}}$ determines decompositions $V_i = \bigoplus_{\la \in \bC^\times} V_i^\la$, $W_i = \bigoplus_{\la \in \bC^\times} W_i^\la$ for each $i \in Q_0$, where $t_\bv$ acts by the scalar $\la$ on $V_i$ and $t_\bw$ acts by the scalar $\la$ on $W_i^\la$. Then
\begin{equation*}
{\bf{N}}^{(t_\bv, t_\bw)} = \bigoplus_{\la \in \bC^\times} \left( \bigoplus_{i \rightarrow j} \Hom (V_{i}^\la, {V_{j}^\la}) \oplus \bigoplus_{i \in Q_0} \Hom( {V_i^\la}, {W_i^\la}) \right).
\end{equation*}
Moreover, we clearly have 
\begin{equation*}
Z_{G_\bv}(t_\bv) = \prod_{\la \in \mathbb{C}^\times} \left( \prod_{i \in Q_0}  GL_{\dim V_i^\la} \right).
\end{equation*}
So, we have the following:

\begin{prop} \label{fixed points of quiver theory is quiver theory}
Fix a quiver $Q$ and dimension vectors $\bv = \{ v_i \}$,  $\bw = \{ w_i \}$. 
For any $(t_\bv, t_\bw) \in T_\bv \times T_\bw$, the theory $(Z_{G_\bv}(t_\bv), {\bf{N}}^{(t_\bv, t_\bw)})$ is the product of quiver gauge theories for the same quiver $Q$ and dimension vectors $\bv^\la = \{ \dim V_i^\la \}$, $\bw^\la = \{ \dim W_i^\la \}$, where for each $i$, $v_i = \sum_\la \dim V_i^\la$, $w_i = \sum_\la \dim W_i^\la$. 
\end{prop}


Now we consider the case of arbitrary $F$. Then the fixed points are still a (product of) quiver theories, but possibly for different quivers:
\begin{prop} \label{prop: fixed points fo quiver for arbitrary F}
Let $(G_\bv, {\bf{N}})$ be a quiver theory, and $F$ is a torus, acting on $N$, commuting with $G_\bv$. Take $(t_\bv, f) \in T_\bv \times F$. Then $(Z_{G_\bv}(t_\bv), {\bf{N}}^{(t_\bv, f)})$ comes from some quiver.
\end{prop}
\begin{proof}
We do the Crawley-Boevey trick \cite[p.261]{CB}. Then the summands of ${\bf{N}}$ correspond to arrows in the quiver, and are of three sorts: $\Hom(V_i, V_i)$, $\Hom(V_i, V_j)$ (with $i \neq j$), and $\Hom(V_i, \bC)$ ($\bC$ being a one-dimensional framing space). The second and the third of them are irreducible as $G_\bv$-representations, while the first one decomposes as $\fsl(V_i) \oplus \bC$. By the Schur lemma, an operator on ${\bf{N}}$, commuting with $G_\bv$, acts by a scalar on each irreducible summand. It follows that action of $F$ on ${\bf{N}}$ factors through the torus $F_{\mathrm{max}}$, which scales the summand corresponding to each arrow of the quiver (we ignore the trivial summands $\bC$ here as they do not affect anything on the Coulomb side and simply multiply the Higgs side by $T^*\mathbb{A}^1$). Thus we can assume $f \in F_{\mathrm{max}}$.

Element $t_\bv \in T_{\bf{v}}$ acts diagonally on $V_i$, defining the decomposition $V_i = \bigoplus_{\la \in \bC^\times} V_i^\la$. Consider an arrow $i \rightarrow j$ (possibly $i=j$), let $f$ act on the summand $\Hom(V_i, V_j)$ corresponding to this arrow by a constant $\eta_{i \rightarrow j}$. Then $\al \in \Hom(V_i, V_j)$  lies in ${\bf{N}}^{(t_\bv, f)}$ if and only if $\al(V_i^\la) \subset V_j^{\la \cdot \eta_{i \rightarrow j}}$ for any~$\la$. Similarly, for a framing summand $\Hom(V_i, \bC)$, on which $f$ acts by a scalar $\eta$, we have $\al \in \Hom(V_i, \bC)$ lies in ${\bf{N}}^{(t_\bv, f)}$ if and only if $\al(V_i^\la) = 0$ for all $\la \neq \eta^{-1}$.

It follows that $(Z_{G_\bv}(t_\bv), {\bf{N}}^{(t_\bv, f)})$ corresponds to a quiver, whose vertices are labeled by $(i, \la)$ such that $V_i^\la \neq 0$; arrows from $(i, \la)$ to $(j, \mu)$ are parametrized by arrows from $i$ to $j$ in $Q$ such that $\mu = \la \cdot \eta_{i \rightarrow j}$ (in notations as above); framing arrows from $(i, \la)$ are parametrized by framing arrows from $i$ in $Q$, such that $\la = \eta^{-1}$ (in notations as above).
\end{proof}

\subsection{Proofs of homological Hikita conjecture in some cases} \label{subsec: proofs of homo-hikita}
We are ready to present some new results (and new proofs of old results) on homological Hikita conjecture.

Namely, we use our results in previous subsections to deduce the equivariant Hikita conjecture from non-equivariant for different set of gauge theories. Conversely, knowing the equivariant Hikita conjecture for some theory $(G_\bv, {\bf{N}})$, one can deduce it for a different theory.

We begin with the following
\begin{prop} \label{prop: hikita from one theory to its centralizer}
Suppose one has an isomorphism of $H_{F \times G_\bv}(\pt)$-algebras
\begin{equation} \label{hikita from old theory to new, equation}
H_{F}(\wti \fM(G_\bv, {\bf{N}})) \simeq \bC[\cM({G_\bv}, {\bf{N}})_{\ff}^{\nu}]
\end{equation}
(equivariant homological Hikita conjecture for $(G_\bv, {\bf{N}})$). Then for any $(t_\bv, f) \in (\ft_\bv / S_\bv) \times \ff$ there is an isomorphism of $H_{F \times Z_{G_\bv}(t_\bv)}(\pt)$ algebras:
\begin{equation*}
H_{F} (\wti \fM(Z_{G_\bv}(t_\bv), {\bf{N}}^{(t_\bv, f)})) \simeq \bC[\cM(Z_{G_\bv}(t_\bv), {\bf{N}}^{(t_\bv, f)})_{\ff}^{\nu}]
\end{equation*}
(equivariant homological Hikita conjecture for $(Z_{G_\bv}(t_\bv), {\bf{N}}^{(t_\bv, f)})$). 
\end{prop}

\begin{proof}
Take the completion of \eqref{hikita from old theory to new, equation} over $H_{G_\bv \times F}(\pt)$ at the maximal ideal, corresponding to $(t_\bv, f)$. Taking completions of Proposition \ref{localization of quiver side of homological hikita}, Corollary \ref{cor: localization of b-algebra}, we obtain the identification
\begin{equation}\label{eq: hikita centralizer}
H_{F} (\wti \fM(Z_{G_\bv}(t_\bv), {\bf{N}}^{(t_\bv, f)}))^{\wedge (0, 0)} \simeq \bC[\cM(Z_{G_\bv}(t_\bv), {\bf{N}}^{(t_\bv, f)})_{\ff}^{\nu}]^{\wedge (0, 0)}.
\end{equation}
This is an isomorphism of non-negatively graded algebras (both of them are non-negatively graded as quotients of $H_{F \times Z_{G_{\bf{v}}}(t_{\bf{v}})} (\operatorname{pt})^{\wedge (0, 0)}$ and the identification (\ref{eq: hikita centralizer}) is graded as it commutes with $H_{F \times Z_{G_{\bf{v}}}(t_{\bf{v}})}(\operatorname{pt})^{\wedge (0, 0)}$-action),
complete with respect to the grading. Taking direct sums of graded components on both sides 
gives the result.
\end{proof}

Proposition \ref{prop: hikita from one theory to its centralizer} may be used as follows. In \cite{KS25}, the second-named author and Shlykov proved the equivariant Hikita conjecture (\ref{eq: intro equi homo hikita}) for the Gieseker variety. They deduced it from knowing it over the generic point of $\ff$ together with flatness of both of the sides of (\ref{eq: intro equi homo hikita}) over $\ff$, checked by computing the fiber at $0 \in \ff$. It turns out that taking the fiber at an appropriate non-generic nonzero point yields the equivariant Hikita conjecture for arbitrary type A quivers. This has been already proved in \cite[Theorem~8.3.7]{Wee16}, so the result is not new, but we think that this proof is interesting on its own. Let's emphasize that the computation needed for the argument in \cite{KS25} is quite simple and only involves a representation theory of  ``classical'' cyclotomic rational Cherednik algebras (the ones with large center).  

\begin{cor} \label{cor: hiktia for A-quivers}
Equivariant homological Hikita conjecture holds for type A quivers.
\end{cor}

\begin{proof}
Let $Q$ be the quiver with one vertex and one loop. Take the dimension and framing numbers $v, w$, consider ${\bf{N}} = \End (\bC^v) \oplus \Hom(\bC^v, \bC^w)$.
Take the flavor torus $F = \torus_{\mathrm{loop}} \times T_w$ (here $T_w = (\torus)^w \subset GL_w$), which acts on ${\bf{N}}$ by the formula
\begin{equation*}
(t, g) (X, \al) = (t X, \al \circ g^{-1}),
\end{equation*}
see \cite{KS25} for details.
The main result of \cite{KS25} is the equivariant Hikita conjecture for this quiver and this flavor torus (which is the maximal possible).

Take an element $(t_v, t_\ell, t_w) \in \Lie (T_v \times \torus_{\mathrm{loop}} \times T_w)$ such that the corresponding one-parameter subgroup of $T_v \times \torus_{\mathrm{loop}} \times T_w$ is of the form
\begin{equation*}
((\underbrace{s, \hdots, s}_{v_1}, \hdots, \underbrace{s^{k}, \hdots, s^{k}}_{v_k}), (s), (\underbrace{s, \hdots, s}_{w_1}, \hdots, \underbrace{s^{k}, \hdots, s^{k}}_{w_k})),
\end{equation*}
$s \in \torus$. Then it is straightforward to see that $(Z_{G_v}(t_v), {\bf{N}}^{(t_v, t_\ell, t_w)})$ is the quiver gauge theory for type $A_k$ quiver, with dimension vector $(v_1, \hdots, v_k)$ and framing vector $(w_1, \hdots, w_k)$. Hence, Proposition \ref{prop: hikita from one theory to its centralizer} implies the claim.
\end{proof}

Now we propose a way to deduce the equivariant Hikita conjecture from non-equivariant.
\begin{thm} \label{equivariant hikita from non-equivariant}
Let $(G_\bv, {\bf{N}})$ be a quiver gauge theory. Suppose for any $(t_\bv, f) \in (T_\bv / S_\bv) \times F$, there is an isomorphism of $H_{Z_{G_\bv}(t_\bv)}(\pt)$-algebras
\begin{equation} \label{eq: hikita on fiber}
H^* (\wti \fM(Z_{G_\bv}(t_\bv), {\bf{N}}^{(t_\bv, f)})) \simeq \bC[\cM(Z_{G_\bv}(t_\bv), {\bf{N}}^{(t_\bv, f)})^{\nu}].
\end{equation}
(homological Hikita conjecture for $(Z_{G_\bv}(t_\bv), {\bf{N}}^{(t_\bv, f)})$). 

Then, there is an isomorphism of $H_{{G_\bv} \times F}(\pt)$-algebras:
\begin{equation} \label{equivariant hikita from non-equivariant, formula for equivariant}
H^*_{F} (\wti \fM({G_\bv}, {\bf{N}})) \simeq \bC[\cM({G_\bv}, {\bf{N}})_{\ff}^{\nu}].
\end{equation}
(homological equivariant Hikita conjecture for $({G_\bv}, {\bf{N}})$).
\end{thm}

\begin{proof}

Take any $f \in \ff$. By Corollary \ref{localization of quiver side of homological hikita over T_w}, the fiber of LHS of \eqref{equivariant hikita from non-equivariant, formula for equivariant} at $f$ over $H_{F}(\pt)$ is isomorphic to 
\begin{equation*} \label{eq: quiver fiber over a point}
\bigoplus_{t_\bv} H^* (\wti \fM({Z_{G_\bv}(t_\bv), {\bf{N}}^{(t_\bv, f)}})).
\end{equation*}
By our assumption, Hikita conjecture holds for $(G_\bv, \bold N)$ (let $(t_\bv, f) = (0, 0)$ in \eqref{eq: hikita on fiber}), hence Corollary~\ref{cor: fiber of b-algebra homo} is applicable, and the fiber of RHS of \eqref{equivariant hikita from non-equivariant, formula for equivariant} over $f$ is isomorphic to
\begin{equation*} \label{eq: coulomb fiber over a point}
\bigoplus_{t_\bv} \bC[\cM(Z_{G_\bv}(t_\bv), {\bf{N}}^{(t_\bv, f)})^\nu].
\end{equation*}
By our assumptions, these two algebras are isomorphic, and in particular, their dimensions coincide. Now, $H^*_{F} (\wti \fM({G_\bv}, {\bf{N}}))$ is free of finite rank over $H_{F}(\pt)$ by \cite[Theorem~7.3.5]{Nak01a}, hence fibers over all points $f \in \mathfrak{f}$ have equal dimensions. Thus, the same holds for $\bC[\cM({G_\bv}, {\bf{N}})_{\ff}^{\nu}]$ and this module is also free over $H_{F}(\pt)$.

Let $k = \Frac H_{F}(\pt)$. We have the following diagram:
\begin{equation} \label{eq: diagram in the proof of equi-homo-hikita from non-equi}
\begin{tikzcd}[column sep=tiny]
	& {H^*_{F}(\widetilde{\mathfrak M}(G_{\bold v}, {\bf{N}}))} && {H^*_{F}(\widetilde{\mathfrak M}(G_{\bold v}, {\bf{N}})) \otimes_{H_{ F}(\mathrm{pt})} k} \\
	{H_{G_{\bold v} \times F}(\mathrm{pt})} && {H_{G_{\bold v} \times F}(\mathrm{pt})\otimes_{H_F(\operatorname{pt})} k} \\
	& {\mathbb C[\mathcal M({G_{\bold v}}, {\bf{N}})_{\mathfrak f}^{\nu}]} && {\mathbb C[\mathcal M({G_{\bold v}}, {\bf{N}})_{\mathfrak f}^{\nu}]\otimes_{H_{F}(\mathrm{pt})} k}
	\arrow["\psi_1", hook, from=1-2, to=1-4]
	\arrow["\theta", "\simeq"', tail reversed, from=1-4, to=3-4]
	\arrow["\phi_1", two heads, from=2-1, to=1-2]
	\arrow[hook, from=2-1, to=2-3]
	\arrow["\phi_2" ', two heads, from=2-1, to=3-2]
	\arrow[two heads, from=2-3, to=1-4]
	\arrow[two heads, from=2-3, to=3-4]
	\arrow["\psi_2" ', hook, from=3-2, to=3-4]
\end{tikzcd}
\end{equation}
Here surjections $\phi_1$ and $\phi_2$ are constructed in \eqref{surjection of homology of point to quiver}, \eqref{surjection from homology of point to b-algebra}; morphisms $\psi_1, \psi_2$ are injective because of the freeness, explained above; let us explain the construction of $\theta$.

Take a generic element $\al \in \ff$. Then canonically
\begin{multline*} \label{eq: quiver side at generic point}
H^*_{F}(\widetilde{\mathfrak M}(G_{\bold v}, {\bf{N}})) \otimes_{H_{F}(\mathrm{pt})} k \simeq H^*_{F}(\widetilde{\mathfrak M}(G_{\bold v}, {\bf{N}})^\al) \otimes_{H_{ F}(\mathrm{pt})} k \simeq \\
H^*(\wti \fM(G_\bv, {\bf{N}})^{\al}) \T k \simeq \left( \bigoplus_{t_\bv} H^*(\wti \fM(Z_{G_\bv}(t_\bv), {\bf{N}}^{(t_\bv, \al)})) \right) \T k,
\end{multline*} 
where we used the localization theorem, the fact that $F$ acts trivially on $\wti \fM(G_\bv, {\bf{N}})^{\al}$ (since $\al$ is generic), and Corollary \ref{localization of quiver side of homological hikita over T_w}.
Similarly, for Coulomb side, using Corollary~\ref{cor: fiber of b-algebra homo}:
\begin{equation*} \label{eq: coulomb side at generic point}
\mathbb C[\mathcal M({G_{\bold v}}, {\bf{N}})_{\mathfrak f}^{\nu}]\otimes_{H_{F}(\mathrm{pt})} k \simeq \mathbb C[(\mathcal M({G_{\bold v}}, {\bf{N}})^\nu_{\mathfrak{f}})^{\al}] \otimes  k \simeq \left( \bigoplus_{t_\bv} \bC[\cM(Z_{G_\bv}(t_\bv), {\bf{N}}^{(t_\bv, \al)})^\nu] \right) \T k.
\end{equation*}
By our assumptions, these algebras are isomorphic as quotients of $\bigoplus_{t_\bv} H_{G_\bv}(\pt) \T k$, which defines $\theta$. Moreover, composing with the diagonal inclusion
$H_{G_\bv}(\pt) \T k \rightarrow \prod_{t_\bv} H_{G_\bv}(\pt) \T k$ also tells us that the rightmost triangle in \eqref{eq: diagram in the proof of equi-homo-hikita from non-equi} is commutative.


The top and bottom parallelograms in \eqref{eq: diagram in the proof of equi-homo-hikita from non-equi} are also obviously commutative.
Hence, the whole diagram is commutative, and $\psi_1 \circ \phi_1$ and $\psi_2 \circ \phi_2$ both are presentations of the same morphism from ${H_{G_{\bold v} \times F}(\mathrm{pt})}$ as composition of surjection and injection. Since such presentation is unique, we get $\im \phi_1 \simeq \im \phi_2$, which is the desired isomorphism.
\end{proof}

Note that if one knows the freeness of $\bC[\cM({G_\bv}, {\bf{N}})_{\ff}^{\nu}]$ over $\bC[\ff]$, then the argument in the proof of Theorem~\ref{equivariant hikita from non-equivariant} shows that the equivariant Hikita conjecture follows from non-equivariant over generic point  $f \in \ff$.

\begin{cor} \label{homo-hikita for ADE quivers}
Let $Q$ be of type ADE, and $(\bv, \bw)$ is such that:
\begin{itemize}
\item If $Q$ is of type E$_7$, then $w_3 = w_4 = w_5 = 0$;
\item If $Q$ is of type E$_8$,  then $w_2 = w_3 = w_4 = w_5 = w_6 =w_7 = 0$.
\end{itemize}
Then the equivariant Hikita conjecture holds for $Q$ and $(\bv, \bw)$.
\end{cor}

\begin{proof}
In \cite[Theorem 8.1]{KTWWY19}, the non-equivariant Hikita conjecture for types ADE is proved for the case when the corresponding slice in affine Grassmannian is non-generalized (note that although not explicitly claimed in {\it loc. cit.}, the proof shows an isomorphism of algebras over $H_{G_\bv}(\pt)$, see Section 
 \ref{app: u(t) vs cohom} for the details). 
It is deduced in {\it loc. cit.} from a result of Zhu \cite{Zhu09}, which is an isomorphism:
\begin{equation} \label{Zhu isomorphism for T-fixed points}
\Gamma(\ol \Gr^\la, \cO(1)) \simeq \Gamma((\ol \Gr^\la)^T, \cO(1))
\end{equation}
(see \cite{Zhu09} for notations). This result is proved in \cite{Zhu09} in types A and D, in types $E_7$ and $E_8$ under the assumptions of this Corollary, and in type $E_6$ under assumption $w_4 = 0$ (see \cite[Proposition 2.2.17]{Zhu09}). This last assumption in type $E_6$ was removed in \cite[Theorem 5.1]{BH20}.

In Theorem \ref{thm: hikita for generalized slices} in Appendix, we prove non-equivariant Hikita conjecture  for generalized slices, hence completing the proof of non-equivariant Hikita conjecture in types ADE under these assumptions on $\bw$.

By Proposition \ref{fixed points of quiver theory is quiver theory}, for a theory $(G_\bv, {\bf{N}})$ as in assumptions of this Corollary, theories $(Z_{G_\bv}(t_\bv), {\bf{N}}^{(t_\bv, t_\bw)})$ are sums of theories of the same kind for any $(t_\bv, t_\bw)$. Hence, the result follows from Theorem \ref{equivariant hikita from non-equivariant}.
\end{proof}

\begin{rem}
It is claimed in \cite{KTWWY19} that non-equivariant Hikita conjecture is proved for quivers of types $ADE$ with no restrictions which we assumed in Corollary \ref{homo-hikita for ADE quivers}. As we point out in the proof above, it is not the case, since the Zhu isomorphism \eqref{Zhu isomorphism for T-fixed points} has not yet been proved in full generality, although it is definitely expected to be true.
\end{rem}

\section{Equivariant Riemann--Roch for Coulomb branches} \label{sec: Riemann-Roch}

The main result of this Section is Theorem~\ref{thm_iso_after_completions}. We first recall generalities on equivariant Riemann--Roch theorem.

\subsection{Equivariant Riemann--Roch isomorphism}
Let $X$ be a variety over complex numbers equipped with an algebraic action of a reductive group  $G$.

\subsubsection{Equivariant Chow groups}\label{subsec_equivar_chow}
In \cite{eg0} (see also \cite[Section 1.2]{eg1}) the notion of equivariant  
Chow groups $A_{i}^G(X)$ is defined. Let us recall the definition. Choose an $\ell$-dimensional representation $V$ of $G$ that contains an open $G$-invariant subset $U \subset V$ such that the action $G \curvearrowright U$ is free and the complement $V \setminus U$ has codimension greater than $\operatorname{dim}X-i$. Then $A_i^{G}(X):=A_{i + \ell - g}(X_G)$, where $X_G = X \times^{G} U$, and $g=\operatorname{dim}G$.

\subsubsection{Equivariant Chow groups via equivariant Borel--Moore homology}\label{sub: ch vs bm} We assume that $X$ has an algebraic cell decomposition, which is invariant under the maximal torus $T \subset G$. 
This implies that the natural cycle map $A_*^{G}(X) \rightarrow H_*^G(X)$ is an isomorphism. Indeed, we have the natural identifications 
\begin{align}
A_*^G(X) &\simeq A_*^T(X)^W, & H_*^G(X) &\simeq H_*^T(X)^W
\end{align}
and now the claim follows from the fact that the cycle morphism $A_*^{T}(X) \rightarrow H_*^{T}(X)$ is an isomorphism (same proof as the one of \cite[Lemma 5.1.1]{CG97} reduces the claim to the case of affine space for which this is immediate).

\subsubsection{Equivariant Chern character} The equivariant Chern character (see, for example, \cite[Definition 3.1]{eg1}) is a homomorphism of algebras:  
\begin{equation*}
\operatorname{ch}^G\colon K_G(X) \rightarrow \prod_{i=0}^{\infty}H^i_G(X).
\end{equation*}
For $X=\operatorname{pt}$, the homomorphism $\operatorname{ch}^G$ is given by: 
\begin{equation*}
\mathbb{C}[T]=K_T(\on{pt}) \ni \chi \mapsto \on{exp}(d\chi) \in \mathbb{C}[\mathfrak{t}]^{\wedge 0} = H_*^T(\on{pt})^{\wedge 0}.
\end{equation*}

\subsubsection{The equivariant Riemann--Roch map $\tau^G$} \label{equivariant_RR_sec}
In \cite{eg1}, an isomorphism
\begin{equation}\label{eq: iso tau EG}
\tau^G\colon K^{G}(X)^{\wedge 1} \iso \prod_{i=\operatorname{dim}X}^{-\infty}A_i^{G}(X)
\end{equation}
is constructed. For our convenience, we will normalize the map $\tau^G$ considered in \cite{eg1} as follows: the map that we will consider is equal to their divided by $\tau^G([\mathcal{O}_{\operatorname{pt}}])$. I.e., for $X=\operatorname{pt}$, we want $\tau^G([\mathcal{O}_{\operatorname{pt}}])=[\mathcal{O}_{\operatorname{pt}}]$.

\begin{rem}
The fact that $K^G(X)^{\wedge 1}$ is isomorphic to the completion of $K^G(X)$ considered in \cite[Section 2]{eg1} follows from \cite[Theorem 6.1 (a)]{eg1}. 
\end{rem}

\begin{rem}
In \cite{eg1} the authors use ``codimension $i$'' equivariant Chow groups $CH^i_G$ (see \cite[Section 1.2]{eg1}) instead of ``dimension $i$'' equivariant Chow groups $A_i^{G}$. We prefer to work with $A_i^{G}$, all the definitions remain the same, the only difference is that $A_*^G(X)$ lives in degrees $\operatorname{dim}X,\operatorname{dim}X-1,\operatorname{dim}X-2,\ldots $, while $CH^i_G$ live in degrees $0,1,\ldots$.
\end{rem}

We assume that $X$ is as in Section (\ref{sub: ch vs bm}) above and identify: 
\begin{equation}\label{identi_A_vs_H}
A_i^{G}(X) \iso H_{i}^{G}(X).
\end{equation}

\begin{lem}
There are isomorphisms     
\begin{align}\label{prod_hom_via_coml}
\prod_{i=\operatorname{dim}X}^{-\infty} H_i^{G}(X) &\simeq H_*^G(X)^{\wedge 0}; & \prod_{i=0}^{\infty}H^i_G(X) &\simeq H^*_G(X)^{\wedge 0}.
\end{align}
\end{lem}
\begin{proof}
The claim follows from the following fact: if $M$ is a finitely generated graded module over $H^*_G(\operatorname{pt})$ such that $M_i = 0$ for $i \ll 0$, then $\prod_i M_i \simeq M^{\wedge 0}$ (compare with the proof of \cite[Proposition 2.1]{eg1}). Note that the action of $H^*_G(X)$ on $H_*^G(X)$ lowers the degree.   
\end{proof}

Combining the identifications (\ref{eq: iso tau EG}), (\ref{identi_A_vs_H}), (\ref{prod_hom_via_coml}) we obtain the identification: 
\begin{equation}\label{equiv_RR_via_compl_0_1}
\tau^G\colon K^G(X)^{\wedge 1} \iso H_*^{G}(X)^{\wedge 0}.
\end{equation}

The main properties of the map $\tau^G$ are the following (see \cite[Theorem 3.1]{eg1}):

\begin{itemize}
    \item For $\alpha \in K_G(X)$ and $x \in K^G(X)$, we have $\tau^{G}(\alpha \cdot x)=\operatorname{ch}^{G}(\alpha) \cdot \tau^G(x)$.
    \item If $f\colon X \rightarrow Y$ is a proper $G$-equivariant morphism, then $\tau^{G}$ commutes with $f_*$.
    \item If $f\colon X \rightarrow Y$ is a smooth and equivariantly quasi-projective $G$-equivariant morphism and $x \in K^G(X)$, then $\tau^G(f^*x)=\operatorname{Td}^G(T_f) \cdot f^*(\tau^G(x))$, where $T_f \in K^G(X)$ is the relative tangent element of the morphism $f$ and $\operatorname{Td}^G(-)$ is the equivariant Todd class (see \cite[Definition 3.1]{eg1}). 
\end{itemize}

If $X$ is nonsingular, then, after the identification $K^G(X) = K_G(X)$, we can  describe $\tau^G$ as follows. Applying the property above to $f\colon X \rightarrow \operatorname{pt}$ and using our normalization $\tau^G([\mathcal{O}_{\operatorname{pt}}])=[\mathcal{O}_{\operatorname{pt}}]$ we conclude that $\tau^G([\mathcal{O}_X])=\operatorname{Td}^G(T_X)$. It then follows from above that: 
\begin{equation*} 
\tau^G (-) = \operatorname{Td}^G(T_X) \operatorname{ch}^G(-) \cdot [X].
\end{equation*}

Assume now that $L \subset G$ is a subgroup and $L$ acts on $X$. Consider the natural identifications $K^L(X) \iso K^G(G \times^L X)$, $H_*^L(X) \iso H_*^G(G \times^L X)$ to be denoted $\operatorname{Ind}_H^G(-)$, $\operatorname{ind}_H^G(-)$ respectively.
The following is \cite[Proposition 3.2 (b)]{eg1} (note that we have an additional factor of $\operatorname{Td}^G(T_{G/L})$ because of our normalization):
\begin{equation}\label{tau_vs_ind}
\tau^G(\operatorname{Ind}_L^G(x)) = \operatorname{Td}^G(T_{G/L}) \cdot \operatorname{ind}_L^G(\tau^L(x)),~\forall x \in K^L(X).
\end{equation}

Assume now that $G$ acts on $X$, then we have restriction maps $K^G(X) \rightarrow K^L(X)$, $H_*^G(X) \rightarrow H_*^L(X)$ to be denoted $\operatorname{Res}_L^G$, $\operatorname{res}_L^G$ respectively. The following is \cite[Proposition 3.2 (c)]{eg1}:
\begin{equation}\label{tau_vs_resrtriction}
\tau^L(\operatorname{Res}_L^G(x))=\operatorname{res}^G_L(\tau^G(x)).
\end{equation}




\subsection{Coulomb branches}

In this section, we work with Coulomb branches in a more general setup then in the rest of the paper. Namely, we work with general Coulomb branches for the pair $(G, \bold {\bf{N}})$, not necessarily of quiver type. We first recall all definitions and constructions of \cite{BFN18}, needed for the proof of Theorem \ref{thm_iso_after_completions} below.

\subsubsection{Space of triples}

Let $G$ be a reductive group and let ${\bf{N}}$ be its finite dimensional representation. Assume also that the action  $G \curvearrowright {\bf{N}}$ extends to the action of some larger group $\widetilde{G} \curvearrowright N$ containing $G$ as a normal subgroup and such that $F:=\widetilde{G}/G$ is a torus. 

Let $\mathcal{R}$ be the space of triples for $(G,{\bf{N}})$. Group $\widetilde{G}_{\mathcal{O}}$ acts naturally on $\mathcal{R}$. We also have an action $\mathbb{C}^\times \curvearrowright \mathcal{R}$ by loop rotation.

We have a natural morphism $\mathcal{R} \rightarrow \operatorname{Gr}_{G}$. Recall that $\operatorname{Gr}_{G}=\bigsqcup_{\lambda \in \Lambda^+}\operatorname{Gr}_{G}^{\lambda}$ and $\overline{\operatorname{Gr}}_{G}^{\lambda}=\bigsqcup_{\mu \leqslant \lambda}\operatorname{Gr}_{G}^{\mu}$. Following \cite[2(ii)]{BFN18}, we denote by $\mathcal{R}_{\leqslant \lambda}$ the preimage of $\overline{\operatorname{Gr}}_{G}^{\lambda}$. Scheme $\mathcal{R}_{\leqslant \lambda}$ is the inverse limit of the system $\mathcal{R}_{\leqslant \la}^d \rightarrow \mathcal{R}^{e}_{\leqslant \la}$ for $d \geqslant e$ (see \cite[Section 2(i)]{BFN18}). 

The following lemma will be important as we want to apply the results of Section \ref{subsec_equivar_chow} to spaces $\cR^{d}_{\leqslant \lambda}$. 

\begin{lem} \label{lem: coulomb is free over cartan}
Space $\cR^d_{\leqslant \lambda}$ has a $\widetilde{T} \times \bC^\times$-invariant algebraic cell decomposition. In particular, $H^{\wti G \times \torus}(\cR)$ is free over $H^{\wti G \times \torus}(\pt)$, and $K^{\wti G \times \torus}(\cR)$ is free over $K^{\wti G \times \torus}(\pt)$.
\end{lem}
\begin{proof}
Same as \cite[Lemma 4.1]{BEF20}.
\end{proof}

\subsubsection{Homological Coulomb branch}
Following \cite[Section 2(ii)]{BFN18} but slightly changing the grading convention we define: 
\begin{equation*}
\mathcal{A}_{\leqslant \lambda} = H_*^{\widetilde{G}_{\cO} \rtimes \bC^\times}(\mathcal{R}_{\leqslant \lambda}) = H^{*}_{\widetilde{G}_{\cO} \rtimes \bC^\times}(\mathcal{R}^d_{\leqslant \lambda}, \boldsymbol{\omega})[2\operatorname{dim}({\bf{N}}_{\mathcal{O}}/z^d{\bf{N}}_{\mathcal{O}})] ,
\end{equation*}
where $\bom$ is the dualizing sheaf of $\mathcal{R}^d_{\leqslant \lambda}$. Then $\mathcal{A} = H_*^{\widetilde{G}_{\cO} \rtimes \bC^\times}(\mathcal{R})$ is defined as the limit of $H_*^{\widetilde{G}_{\cO} \rtimes \bC^\times}(\mathcal{R}_{\leqslant \la})$ under pushforward maps. $\cA$ can be equipped with an algebra structure via convolution $*$ (see \cite[Section 3(iii)]{BFN18}).

Note that the algebra $H_*^{\widetilde{G}_{\cO} \rtimes \bC^\times}(\cR)$ is naturally graded, its $i$'th  degree  component is:
\begin{equation*}
\cA_i := \underset{\longrightarrow}{\operatorname{lim}}\, H^{i}_{\widetilde{G}_{\cO} \rtimes \bC^\times}(\cR^d_{\leqslant \lambda},\boldsymbol{\omega})[2\operatorname{dim}({\bf{N}}_{\cO}/z^d{\bf{N}}_{\cO})].
\end{equation*}

\subsubsection{K-theoretic Coulomb branch} 
As in the homological case, $K^{\widetilde{G}_{\cO} \rtimes \bC^\times}(\mathcal{R}_{\leqslant \lambda})$ is defined as $K^{\widetilde{G}_{\cO} \rtimes \bC^\times}(\mathcal{R}^d_{\leqslant \lambda})$ for $d$ large enough (the resulting K-groups identify canonically for all $d$'s using pullbacks under the flat morphisms $\mathcal{R}^d_{\leqslant \lambda} \rightarrow \mathcal{R}^e_{\leqslant \lambda}$). Using the closed embeddings $\mathcal{R}^d_{\leqslant \mu} \hookrightarrow \mathcal{R}^d_{\leqslant \la}$, $\cA^\times = K^{\widetilde{G}_{\cO} \rtimes \bC^\times}(\mathcal{R})$ is defined as the limit w.r.t. the pushforward maps.  
$K^{G_{\mathcal{O}} \rtimes \bC^\times}(\mathcal{R})$ is an algebra with the convolution $\star$ operation \cite[Remarks~3.9]{BFN18}.

We can also consider the classical limits of algebras $\cA$, $\cA^\times$ as above:
\begin{align*}
\mathbb{C}[\mathcal{M}_{F}^\times]&=K^{\widetilde{G}_{\cO}}(\mathcal{R}), & \mathbb{C}[\mathcal{M}_{\mathfrak{f}}]&=H_*^{\widetilde{G}_{\cO}}(\mathcal{R}).
\end{align*}
These are algebras of functions on the corresponding deformed Coulomb branches $\mathcal{M}_{\mathfrak{f}},\,\mathcal{M}_{F}^\times$.

\begin{rem}{}
More generally, one can consider parabolic versions of the algebras $\mathcal{A}$, $\mathcal{A}^\times$ as well as their classical analogs, see \cite[Definition 2.2]{KWWY24}. Results of this section should be valid for this more general situation.
\end{rem}


\subsubsection{Multiplication}
Let us describe the properties of convolution product that determine it uniquely. 

We start with the case $N=0$ and $G=T$. In this case, the algebra $K^{\widetilde{T} \times \bC^\times}(\mathcal{R})$ can be explicitly described as follows. Identify
$K^{\widetilde{T} \times \bC^\times}(\operatorname{pt})=\bC[\widetilde{T}][q^{\pm 1}]$. 
For a character $\chi \colon \widetilde{T} \rightarrow \bC^\times$ let $f_\chi$ be the corresponding element of $\bC[\widetilde{T}][q^{\pm 1}]=K^{\widetilde{T} \times \bC^\times}(\operatorname{pt})$.  
Then, $K^{\widetilde{T}}(\mathcal{R})$ is a free (left) module over $K^{\widetilde{T} \times \bC^\times}(\operatorname{pt})$ with generators $\{r_\lambda\,|\, \lambda \in X^*(T)\}$. Multiplication is uniquely determined by the following formulae:
\begin{align*}
r_\lambda \star r_\mu &= r_{\lambda+\mu},&  r_\lambda \star f_\chi &= q^{\pi_1(\chi)}f_\chi r_{\lambda},
\end{align*}
where $\pi_1(\chi)$ is the $\mathbb{Z}$-valued function given by $\pi_0(\mathcal{R})=\pi_1(T) \xrightarrow{\pi_1(\chi|_{T})} \pi_1(\mathbb{C}^\times)=\mathbb{Z}$.

Let's now describe the multiplication $\star$ on $K^{\widetilde{G}_{\cO} \rtimes \bC^\times}(\operatorname{Gr}_G)$ (i.e., we still assume $N=0$ but put no restrictions on $G$). We have the natural morphism $\iota\colon \operatorname{Gr}_T \hookrightarrow \operatorname{Gr}_G$, it induces an injective homomorphism of algebras over $K_{\widetilde{T}_{\cO} \rtimes \bC^\times}(\operatorname{pt})^{W} = K_{\widetilde{G}_{\cO} \rtimes \bC^\times}(\operatorname{pt})$
\begin{equation*}
\iota_*\colon K^{\widetilde{T}_{\cO} \rtimes \bC^\times}(\operatorname{Gr}_T)^{W} \hookrightarrow K^{\widetilde{G}_{\cO} \rtimes \bC^\times}(\operatorname{Gr}_{G}),
\end{equation*}
which becomes an isomorphism after localization 
(see \cite[Remark~5.23]{BFN18}), 
hence the algebra structure on $K^{\widetilde{G}_{\cO} \rtimes \bC^\times}(\operatorname{Gr}_{G})$ is uniquely determined by the algebra structure on $K^{\widetilde{T}_{\cO} \rtimes \bC^\times}(\operatorname{Gr}_T)^{W} \hookrightarrow K^{\widetilde{T}_{\cO} \rtimes \bC^\times}(\operatorname{Gr}_T)$.

Finally, let's describe the algebra structure on $K^{\widetilde{G}_{\cO} \rtimes \bC^\times}(\cR)$ in general.  We have morphisms:
\begin{equation} \label{eq: R --> T --> Gr}
\cR \xrightarrow{i} \cT \xrightarrow{\pi} \operatorname{Gr}_G
\end{equation}
(see \cite[Section 2]{BFN18}). Morphism $\pi$ is an infinite rank vector bundle (with fibers being isomorphic to ${\bf{N}}_{\cO}$), so the pullback $\pi^*$ induces an isomorphism on both K-theory and homology. Morphism $i$ is a closed embedding.  The same argument as in \cite[Lemma~5.11]{BFN18} shows that composing $i_*$ and $(\pi^*)^{-1}$, we obtain an injective algebra homomorphism: 
\begin{equation} \label{eq: embedding from K(R) to K(Gr)}
(\pi^{*})^{-1} \circ i_*  \colon K^{\widetilde{G}_\mathcal{O} \rtimes \bC^\times}(\mathcal{R}) \hookrightarrow K^{\widetilde{G}_\mathcal{O} \rtimes \bC^\times}(\operatorname{Gr}_G).
\end{equation}
So, the convolution product on $K^{\widetilde{G}_\mathcal{O} \rtimes \bC^\times}(\mathcal{R})$ is uniquely determined by the convolution product on $K^{\widetilde{G}_\mathcal{O} \rtimes \bC^\times}(\operatorname{Gr}_G)$.

\subsection{Riemann--Roch isomorphism for Coulomb branches}

\subsubsection{Completions of Coulomb branch algebras}

Algebra $\mathcal{A}$ is a (left) module over 
\begin{equation}\label{cohom_point}
H^*_{\widetilde{G} \times \mathbb{C}^\times}(\operatorname{pt})=\mathbb{C}[\widetilde{\mathfrak{g}},\hbar]^{\widetilde{G}}=\mathbb{C}[(\widetilde{\mathfrak{t}} \oplus \mathbb{C})]^{W}
\end{equation}
and $\mathcal{A}^\times$ is a module over 
\begin{equation}\label{K_th_pt}
K_{\widetilde{G} \times \mathbb{C}^\times}(\operatorname{pt})=\mathbb{C}[\widetilde{G},q^{\pm 1}]^{\widetilde{G}}=\mathbb{C}[\widetilde{T} \times \mathbb{C}^\times]^{W}.
\end{equation}
Recall that 
\begin{align*}
\cA^{\times} &= \underset{\longrightarrow}{\operatorname{lim}}\, K^{\widetilde{G}_{\cO} \rtimes \bC^\times}(\cR_{\leqslant \lambda}), & \cA &= \underset{\longrightarrow}{\operatorname{lim}}\, H_*^{\widetilde{G}_{\cO} \rtimes \bC^\times}(\cR_{\leqslant \lambda}).
\end{align*}
Let $H_*^{\widetilde{G}_{\cO} \rtimes \bC^\times}(\cR_{\leqslant \lambda})^{\wedge {0}}$, $K^{\widetilde{G}_{\cO} \rtimes \bC^\times}(\cR_{\leqslant \lambda})^{\wedge 1}$ be the completions of  $H_*^{\widetilde{G}_{\cO} \rtimes \bC^\times}(\cR_{\leqslant \lambda})$, $K^{\widetilde{G}_{\cO} \rtimes \bC^\times}(\cR_{\leqslant \lambda})$ at the augmentation ideals of $H^*_{\widetilde{G} \times \bC^\times}(\operatorname{pt})$, $K_{\widetilde{G} \times \bC^\times}(\operatorname{pt})$ corresponding to $0 \in (\widetilde{\mathfrak{t}} \oplus \bC)/W$ and $1 \in (\widetilde{T} \times \bC^\times)/W$. 
Set
\begin{align} \label{eq: colimit of completions of coulombs}
\widehat{\mathcal{A}^\times} &:= \underset{\longrightarrow}{\operatorname{lim}}\, K^{\widetilde{G}_{\cO} \rtimes \bC^\times}(\cR_{\leqslant \lambda})^{\wedge 1}, & \widehat{\mathcal{A}} &:= \underset{\longrightarrow}{\operatorname{lim}}\,H_*^{\widetilde{G}_{\cO} \rtimes \bC^\times}(\cR_{\leqslant \lambda})^{\wedge 0} .
\end{align}

\begin{lem}{}
The convolution product on $\cA$, $\cA^\times$ induces the product on the corresponding completions     $\widehat{\mathcal{A}}$, $\widehat{\mathcal{A}^\times}$.
\end{lem}
\begin{proof}
Let's prove the claim for $\widehat{\cA^\times}$, the argument for $\widehat{\cA}$ is similar. Let $\mathfrak{m}_1 \subset K^{\widetilde{G}_{\cO} \rtimes \bC^\times}(\operatorname{pt})$ be the augmentation ideal. It is enough to check that for every $a \in \cA^{\times}$, $a\mathfrak{m}_1^k \subset \mathfrak{m}_1^{k}\cA^{\times}$. This is equivalent to $[a,\mathfrak{m}_1^k] \subset \mathfrak{m}_1^{k}\cA^{\times}$. Note now that the quotient $\cA^\times/(q-1)$ is commutative. It follows that for any $x \in \cA^{\times}$ we have 
$[a,x] \in (q-1)\cA^{\times}$, so $\frac{[a,x]}{q-1}$ is a well-defined element of $\cA^{\times}$ (recall that $\cA^{\times}$ has no zero divisors, see \cite[Corollary 5.22]{BFN18}).
We now prove by the induction on $k$ that $a\mathfrak{m}_1^k \subset \mathfrak{m}_1^{k}\cA^{\times}$. For $k=0$ this is clear. Let's prove the induction step. 
Pick $x_1,\ldots,x_k \in \mathfrak{m}_1$. It is enough to check that $[a,x_1\ldots x_k] \in \mathfrak{m}_1^k\cA^{\times}$. Setting $a_i := \frac{[a,x_i]}{q-1}$ and using the Leibnitz rule we get:
\begin{multline*}
[a,x_1 \ldots x_k] = \\
 (q-1)(a_1x_2 \ldots x_k + x_1 a_2 x_3 \ldots x_k + \ldots + x_1\ldots x_{i-1}  a_i x_{i+1} \ldots x_k + \ldots x_1 \ldots x_{k-1}  a_k).
\end{multline*}
By the induction hypothesis, we conclude that $[a,x_1\ldots x_k] \in (q-1)\mathfrak{m}_1^{k-1}\cA^{\times} \subset \mathfrak{m}_1^k\cA^{\times}$.
\end{proof}

\subsubsection{Isomorphism between completed Coulomb branch algebras}

The equivariant Chern character gives an identification:
\begin{equation*}
\operatorname{ch}^{\widetilde{G} \times \mathbb{C}^\times}\colon K_{\widetilde{G} \times \mathbb{C}^\times}(\operatorname{pt})^{\wedge 1} \iso H^*_{\widetilde{G} \times \mathbb{C}^\times}(\operatorname{pt})^{\wedge 0}.
\end{equation*}

The main result of this section is the following theorem.

\begin{thm}\label{thm_iso_after_completions}
There exists an isomorphism of algebras:
\begin{equation*}
\Upsilon\colon \widehat{\mathcal{A}^\times} \iso \widehat{\mathcal{A}}
\end{equation*}
compatible with the actions of $K_{\widetilde{G} \times \mathbb{C}^\times}(\operatorname{pt})^{\wedge 1} \iso H^*_{\widetilde{G} \times \mathbb{C}^\times}(\operatorname{pt})^{\wedge 0}$. Same holds for classical 
Coulomb branch algebras.
\end{thm}

Recall that by the definition we have 
\begin{align} \label{actual_def_coulomb}
\mathcal{A}^\times &= \underset{\longrightarrow}{\operatorname{lim}}\, K^{\widetilde{G} \rtimes \mathbb{C}^\times}(\mathcal{R}_{\leqslant \lambda}^d), & \mathcal{A} &= \underset{\longrightarrow}{\operatorname{lim}}\, H_*^{\widetilde{G} \times \mathbb{C}^\times}(\mathcal{R}_{\leqslant \lambda}^d),
\end{align}
where $\mathcal{R}_{\leqslant \lambda}^d$ are finite dimensional schemes. The  equivariant Riemann--Roch map (\ref{equiv_RR_via_compl_0_1}) provides an isomorphism:
\begin{equation*}
\tau^{\widetilde{G} \times \mathbb{C}^\times} \colon K^{\widetilde{G} \times \mathbb{C}^\times}(\mathcal{R}_{\leqslant \lambda}^d)^{\wedge 1} \iso H^{\widetilde{G} \times \mathbb{C}^\times}(\mathcal{R}_{\leqslant \lambda}^d)^{\wedge 0} 
\end{equation*}
over $K_{\widetilde{G} \times \mathbb{C}^\times}(\operatorname{pt})^{\wedge 1} \iso H^*_{\widetilde{G} \times \mathbb{C}^\times}(\operatorname{pt})^{\wedge 0}$. Morphism $\tau^{\widetilde{G} \times \mathbb{C}^\times}$ commutes with proper push forwards (\cite[Theorem 3.1]{eg1}, see also Section \ref{equivariant_RR_sec}), so passing to the limit we obtain an isomorphism:
\begin{equation*}
\tau^{\widetilde{G}_{\mathcal{O}} \rtimes \mathbb{C}^\times}\colon \widehat{\mathcal{A}^\times} \iso \widehat{\mathcal{A}}
\end{equation*}
of modules over $K_{\widetilde{G} \rtimes \mathbb{C}^\times}(\operatorname{pt})^{\wedge 1} \iso H^*_{\widetilde{G} \rtimes \mathbb{C}^\times}(\operatorname{pt})^{\wedge 0}$.
We will see that for $\bold N = 0$ this map is an isomorphism of algebras (i.e., is compatible with convolution). In general, we need to modify it as follows. 
Consider the closed embedding $\mathcal{R}_{\leqslant \la}^d \hookrightarrow \mathcal{T}_{\leqslant \la}^d$. Scheme $\mathcal{T}_{\leqslant \la}^d$ is a vector bundle over $\overline{\operatorname{Gr}}^\lambda_G$ with fibers being $\bold N_{\mathcal{O}}/z^d \bold N_{\mathcal{O}}$. 
Let $\operatorname{Td}^{\widetilde{G} \rtimes \mathbb{C}^\times}(\mathcal{T}_{\leqslant \la}^d) \in H^*_{G \rtimes \mathbb{C}^\times}(\overline{\operatorname{Gr}}^\lambda_G)^{\wedge 0}$ be the equivariant Todd class of $\mathcal{T}_{\leqslant \la}^d$ (see \cite[Definition~3.1]{eg1}). 
Pulling $\operatorname{Td}^{\widetilde{G} \times \mathbb{C}^\times}(\mathcal{T}_{\leqslant \la}^d)$ back under the map $\pi\colon \mathcal{R}^d_{\leqslant \la} \rightarrow \ol{\operatorname{Gr}}^\la_G$, we define:
\begin{equation*}
\Upsilon^d_{\leqslant \lambda} := (\pi^*\operatorname{Td}^{\widetilde{G} \times \mathbb{C}^\times}(\mathcal{T}_{\leqslant \la}^d)^{-1} \cap -) \circ \tau^{\widetilde{G} \times \mathbb{C}^\times}\colon K^{\widetilde{G} \times \mathbb{C}^\times}(\mathcal{R}_{\leqslant \lambda}^d)^{\wedge 1} \iso H^{\widetilde{G} \times \mathbb{C}^\times}(\mathcal{R}_{\leqslant \lambda}^d)^{\wedge 0}. 
\end{equation*}
For $d' \geqslant d$, the natural morphism $\tilde{p}^{d'}_d \colon \mathcal{R}^{d'}_{\leqslant \lambda} \rightarrow \mathcal{R}^{d}_{\leqslant \lambda}$ is a vector bundle with fibers being isomorphic to $z^d \bold N_{\mathcal{O}}/z^{d'} \bold N_{\mathcal{O}}$. 
It then follows from \cite[Theorem 3.1 (d)]{eg1} (see also Section~\ref{equivariant_RR_sec}) that morphisms $\Upsilon^{d}_{\leqslant \lambda}$, $\Upsilon^{d'}_{\leqslant \la}$ are compatible after the identifications induced by $(\tilde{p}^{d'}_d)^*$ (we use that the Todd class is multiplicative).

This allows us to take the limit of $\Upsilon^d_{\leqslant \lambda}$ and obtain the desired identification:
\begin{equation*}
\Upsilon\colon \widehat{\mathcal{A}^\times} \iso \widehat{\mathcal{A}}.
\end{equation*}
In what follows we check that $\Upsilon$ is indeed a homomorphism of algebras.
We start with:

\begin{lem}\label{lemma_tau_homom_abelian}
For $G=T$ a torus and $\bold N=0$, the map $\tau^{\widetilde{T}_{\mathcal{O}}\times \mathbb{C}^\times}$ is an isomorphism of algebras.
\end{lem}
\begin{proof}
For $\widetilde{G}=\widetilde{T}$, ${\bf{N}}=0$ the Coulomb branches $\mathcal{A}$, $\mathcal{A}^\times$ are generated over $H^*_{\widetilde{T}_{\mathcal{O}} \rtimes \BC^\times}(\operatorname{pt})$, $K_{\widetilde{T} \times \BC^\times}(\operatorname{pt})$  by $\iota_{\lambda,*}[z^{\lambda}]$, $[\iota_{\lambda,*}\mathcal{O}_{z^\lambda}]$ respectively, where $\la \in \operatorname{Hom}(\mathbb{C}^\times,T)$ and $\iota_\la\colon \{z^\lambda\} \hookrightarrow \operatorname{Gr}_T$ is the natural embedding. For a character $\chi\colon \widetilde{T} \rightarrow \mathbb{C}^\times$, let $\mathbb{C}_\chi$ be the corresponding one-dimensional representation of $\widetilde{T}_{\mathcal{O}} \rtimes \mathbb{C}^\times$ considered as a $\widetilde{T}_{\mathcal{O}} \rtimes \mathbb{C}^\times$-equivariant line bundle on a point. Then, we can consider $c_1(\mathbb{C}_\chi) \in H^*_{\widetilde{T}_{\mathcal{O}} \rtimes \mathbb{C}^\times}(\operatorname{pt})$ and we have: 
\begin{equation}\label{commut_coulomb_abel}
\bullet * c_1(\mathbb{C}_\chi)1 = (c_1(\mathbb{C}_\chi)+\hbar \pi_1(\chi)) \cdot \bullet,
\end{equation}
\begin{equation}\label{commut_mult_coulomb_abel}
     \bullet \star [\mathbb{C}_{\chi}] = q^{\pi_1(\chi)} [\mathbb{C}_\chi] \cdot \bullet.
\end{equation}
Note now that 
\begin{equation}\label{chern_of_product}
\operatorname{ch}(q^{\pi_1(\chi)}[\mathbb{C}_\chi])=e^{\hbar \pi_1(\chi)}e^{c_1(\mathbb{C}_\chi)}=e^{c_1(\mathbb{C}_\chi)+\hbar\pi_1(\chi)}. 
\end{equation}
It also follows from (\ref{commut_coulomb_abel}) that:
\begin{equation}\label{commut_exp_coulomb_abel}
\bullet * e^{c_1(\mathbb{C}_\chi)}1 = e^{c_1(\mathbb{C}_\chi) + \hbar \pi_1(\chi)}  \cdot \bullet.
\end{equation}
Combining (\ref{commut_mult_coulomb_abel}), (\ref{commut_exp_coulomb_abel}) and (\ref{chern_of_product}) we conclude that:
\begin{equation*}
\tau(\bullet * [{\mathbb{C}}_\chi]1) = \tau(q^{\pi_1(\chi)}[\mathbb{C}_\chi] \cdot \bullet) = \operatorname{ch}(q^{\pi_1(\chi)}[\mathbb{C}_\chi]) \cdot \tau(\bullet) =  e^{c_1(\mathbb{C}_\chi)+\hbar \pi_1(\chi)} \cdot \tau(\bullet) = \tau(\bullet) * e^{c_1(\mathbb{C}_\chi)}1
\end{equation*}
so $\tau$ is a homomorphism of bimodules.

It remains to note that:
\begin{multline*}
\tau(\iota_{\la,*}[\mathcal{O}_{z^\lambda}] \star \iota_{\eta,*}[\mathcal{O}_{z^\eta}]) = \tau(\iota_{\la+\eta,*}[\mathcal{O}_{z^{\lambda+\eta}}])=\\
\iota_{\la+\eta,*} \tau([\mathcal{O}_{z^{\lambda+\eta}}])=\iota_{\la+\eta,*}[z^{\la+\eta}]=\iota_{\lambda,*}[z^\lambda] * \iota_{\eta,*}[z^\eta]=\tau(\iota_{\lambda,*}\mathcal{O}_{\lambda}) * \tau(\iota_{\eta,*}\mathcal{O}_{\eta}).
\end{multline*}
\end{proof}

\begin{lem}
For $\bold N=0$, the map $\tau^{\widetilde{G}_{\mathcal{O}} \rtimes \mathbb{C}^\times}$ is an isomorphism of algebras.    
\end{lem}
\begin{proof}
By \cite[Lemma 5.10]{BFN18}  pushforward along the natural embedding $\iota\colon \operatorname{Gr}_{T} \hookrightarrow \operatorname{Gr}_G$ defines homomorphisms of algebras:
\begin{align*}
H_*^{\widetilde{T}_{\mathcal{O}} \rtimes \mathbb{C}^\times}(\operatorname{Gr}_T)^{W} &\rightarrow H_*^{\widetilde{G}_{\mathcal{O}} \rtimes \mathbb{C}^\times}(\operatorname{Gr}_G), & K^{\widetilde{T}_{\mathcal{O}} \rtimes \mathbb{C}^\times}(\operatorname{Gr}_T)^{W} &\rightarrow K^{\widetilde{G}_{\mathcal{O}} \rtimes \mathbb{C}^\times}(\operatorname{Gr}_G).
\end{align*}
These homomorphisms are $H^*_{\widetilde{T} \times \mathbb{C}^\times}(\operatorname{pt})^{W}$ and $K_{\widetilde{T} \times \mathbb{C}^\times}(\operatorname{pt})^{W}$-linear.  Moreover, they become isomorphisms after appropriate localizations of $H^*_{\widetilde{T} \times \mathbb{C}^\times}(\operatorname{pt})^W$ (resp. $K_{\widetilde{T} \times \mathbb{C}^\times}(\operatorname{pt})$). Namely, we should localize $H^*_{\widetilde{T} \times \mathbb{C}^\times}(\operatorname{pt})^W$ at the elements $\alpha \in \Delta_G$ and we should localize $K_{\widetilde{T} \times \mathbb{C}^\times}(\operatorname{pt})$ at the elements $1-f_{-\alpha}$, $\alpha \in \Delta_G$. The image of  $\frac{1}{1-f_{-\alpha}}$ under the Chern character map is equal to $\frac{1}{1-e^{-\alpha}}=\frac{\alpha/(1-e^{-\alpha})}{\alpha}$ and the element $\alpha/(1-e^{-\alpha})$ already lives in $H^*_{\widetilde{G} \times \mathbb{C}^\times}(\operatorname{pt})^{\wedge 0}$.

Note now that $\iota_*$ commutes with $\tau^T$ 
so our claim follows from Lemma \ref{lemma_tau_homom_abelian} together with the equality (\ref{tau_vs_resrtriction}) (claiming that maps $\tau^G, \tau^T$ are compatible).
\end{proof}

Recall the maps $i, \pi$ in \eqref{eq: R --> T --> Gr}, embedding \eqref{eq: embedding from K(R) to K(Gr)} given by $(\pi^{*})^{-1} i_*$ and its homological version \cite[Lemma~5.11]{BFN18}.



We are ready to finish the proof of Theorem \ref{thm_iso_after_completions}.

\begin{proof}[Proof of Theorem \ref{thm_iso_after_completions}]
Recall that $\wh \cA,\,\wh{\cA^\times}$ are defined as inductive limits of completions of finite $H_{\wti G \times \torus}(\pt)$ and $K^{\wti G \times \torus}(\pt)$-modules respectively (see \eqref{eq: colimit of completions of coulombs}). In the category of finite modules over a Noetherian ring, completion is an {{exact functor}}, so it remains to check that the following diagram is commutative:
\[\begin{tikzcd}[sep=large]
	{\widehat{\mathcal{A}^\times}} & {K^{\widetilde{G}_{\mathcal{O}} \rtimes \mathbb{C}^\times}(\operatorname{Gr}_G)^{\wedge 1}} \\
	{\widehat{\mathcal{A}}} & {H^{\widetilde{G}_{\mathcal{O}} \rtimes \mathbb{C}^\times}_*(\operatorname{Gr}_G)^{\wedge 0}}
	\arrow["{(\pi^{*})^{-1} i_*}", hook, from=1-1, to=1-2]
	\arrow["\Upsilon", from=1-1, to=2-1]
	\arrow["{\tau^{\widetilde{G}_{\mathcal{O}} \rtimes \mathbb{C}^\times}}", from=1-2, to=2-2]
	\arrow["(\pi^{*})^{-1} i_*", hook, from=2-1, to=2-2]
\end{tikzcd}\]
This follows from the definition of $\Upsilon$ together with \cite[Theorem 3.1 (b), (d)]{eg1}.
\end{proof}

\subsection{Explicit description of $\Upsilon$ for quiver gauge theories and for $G=T$} \label{subsec: formulas for upsilon}

\subsubsection{Formulae for $\Upsilon$ on generators} \label{subsubsec: formulas for upsilon on monopoles}
Let $\boldsymbol{\lambda}$ be a minuscule coweight for $G$. 
Then there is an isomorphism:
\begin{equation*}
\overline{\operatorname{Gr}^{\boldsymbol{\lambda}}_{G}} = \operatorname{Gr}^{\boldsymbol{\lambda}}_{G} \simeq G/P_{\boldsymbol{\la}},
\end{equation*}
where $P_{\boldsymbol{\lambda}} \subset G$ is the parabolic subgroup generated by $\mathfrak{g}_{\alpha}$ such that $\langle \al,{\boldsymbol{\lambda}}\rangle \leqslant 0$.

It follows that $\cR_{\leqslant \boldsymbol{\lambda}}=\cR_{ \boldsymbol{\lambda}}$ is a vector bundle over $\operatorname{Gr}^{\boldsymbol{\la}}_{G}$ so we obtain the identification: 
\begin{equation}\label{K_th_R_min}
K^{\widetilde{G}_{\cO} \rtimes \bC^\times}(\cR_{\leqslant \lambda}) \simeq K^{\widetilde{G}_{\cO} \rtimes \bC^\times}(\operatorname{Gr}^{\lambda}) \simeq K^{\widetilde{P}_{\boldsymbol{\lambda}} \times \bC^\times}(\{z^{\boldsymbol{\lambda}}\})=\bC[\widetilde{T}]^{W_{\boldsymbol{\lambda}}} \otimes \bC[q^{\pm 1}],
\end{equation}
where $W_{\boldsymbol{\lambda}} \subset W$ is the stabilizer of $\boldsymbol{\lambda}$ in $W$.

For every $p \in \bC[\widetilde{T}]^{W_{\boldsymbol{\la}}} \otimes \bC[q^{\pm 1}]$, we denote by $M^\times_{{\boldsymbol{\lambda}},p} \in \cA^\times$ the corresponding element. It is called {\it dressed minuscule monopole operator} (dressing refers to $p$). We similarly define the dressed minuscule monopole operators $M_{\boldsymbol{\la},p} \in \cA$ (here $p \in \bC[\widetilde{\mathfrak{t}}]^{W_{\boldsymbol{\la}}} \otimes \bC[\hbar]$). 

Let $t_k$, $k=1,\ldots,d=\operatorname{dim}\widetilde{T}$ be some coordinates on $\widetilde{T}$ and $x_k$ be the corresponding coordinates on $\widetilde{\mathfrak{t}}$. For $\chi \in X^*(\widetilde{T})$ we will denote by $\langle \boldsymbol{\lambda},\chi\rangle$ the pairing of $\boldsymbol{\lambda}$ with $\chi|_T$.

\begin{prop}\label{upsilon_on_minuscule_gen}
 We have:
 \begin{equation}\label{image_minuscule}
\Upsilon(M^\times_{\boldsymbol{\la},p(t_1,\ldots,t_d,q)}) = M_{\boldsymbol{\la},a \cdot p(e^{x_1},\ldots,e^{x_d},e^\hbar)},
 \end{equation}
 where $a:=\prod_{\chi \in X^*(\widetilde{T})}\prod_{k=\langle {\boldsymbol{\la}},\chi\rangle,\ldots,-1} \Big(\frac{1-e^{-\chi-k\hbar}}{\chi+k\hbar}\Big)^{\operatorname{dim}{\bf{N}}(\chi)} \cdot \prod_{\langle \alpha,{\boldsymbol{\lambda}} \rangle > 0} (\frac{\alpha}{1-e^{-\alpha}})$.
\end{prop}

\begin{proof}
For minuscule ${\boldsymbol{\la}}$, the isomorphism
\begin{equation*}
\Upsilon^d_{\leqslant \boldsymbol{\lambda}} \colon K^{\widetilde{G}_{\mathcal{O}} \rtimes \mathbb{C}^\times}(\mathcal{R}_{\boldsymbol{\lambda}}^d)^{\wedge 1} \iso H_*^{\widetilde{G}_{\mathcal{O}} \rtimes \mathbb{C}^\times}(\mathcal{R}_{\boldsymbol{\lambda}}^d)^{\wedge 0}
\end{equation*}
can be described explicitly. Recall that we have identifications (see (\ref{K_th_R_min})):
\begin{align*}
K^{\widetilde{G}_{\mathcal{O}} \rtimes \mathbb{C}^\times}(\mathcal{R}_{\la}^d)^{\wedge 1} &\simeq K^{\widetilde{P}_{\boldsymbol{\lambda}} \times \mathbb{C}^\times}(\{z^\la\})^{\wedge 1},& H_*^{\widetilde{G}_{\mathcal{O}} \rtimes \mathbb{C}^\times}(\mathcal{R}_{\la}^d)^{\wedge 0} &\simeq H_*^{\widetilde{P}_\lambda \times \mathbb{C}^\times}(\{z^\la\})^{\wedge 0}.
\end{align*}
Now, using (\ref{tau_vs_ind})  and \cite[Theorem 3.1 (d)]{eg1}  we see that up to the multiplication by a Todd class of a certain explicit vector bundle, morphism $\Upsilon$ is given by:
\begin{equation*}
\operatorname{Td}^{\widetilde{P}_\lambda \times \mathbb{C}^\times}(\mathfrak{g}/\mathfrak{p}_{\boldsymbol{\lambda}}) \cdot \tau^{\widetilde{P}_{\boldsymbol{\lambda}} \times \mathbb{C}^\times}(-)\colon K^{\widetilde{P}_\lambda \times \mathbb{C}^\times}(\{z^{\boldsymbol{\la}}\})^{\wedge 1} \iso  H_*^{\widetilde{P}_{\boldsymbol{\lambda}} \times \mathbb{C}^\times}(\{z^{\boldsymbol{\la}}\})^{\wedge 0},
\end{equation*}
which is the equivariant Chern character (given by $p(t_1,\ldots,t_d,q) \mapsto p(e^{x_1},\ldots,e^{x_d},e^\hbar)$) times $\operatorname{Td}^{\widetilde{P}_\lambda \times \mathbb{C}^\times}(\mathfrak{g}/\mathfrak{p}_{\boldsymbol{\lambda}})=\prod_{\langle \alpha,{\boldsymbol{\lambda}} \rangle > 0} (\frac{\alpha}{1-e^{-\alpha}})$.

So, it remains to compute the Todd class correction involved in the definition of $\Upsilon^d_{\leqslant {\boldsymbol{\la}}}$ (see \cite[Definition 3.1]{eg1} for the explicit formula for the Todd class). 
Recall that  $\mathcal{T}^d_{\leqslant {\boldsymbol{\la}}}=\mathcal{T}^d_{{\boldsymbol{\la}}}$ is a vector bundle over $\operatorname{Gr}^{\boldsymbol{\lambda}}_{G}$. 
This bundle is 
$\widetilde{G}_{\cO} \rtimes \bC^\times$-equivariant with fiber over $z^{\boldsymbol{\la}}$ equal to ${\bf{N}}/z^d{\bf{N}}_{\cO}$. It has a vector subbundle $\cR^d_{\boldsymbol{\lambda}}$. Our correction comes from the Todd class of the vector bundles quotient $\mathcal{T}^d_{{\boldsymbol{\lambda}}}/\cR^d_{{\boldsymbol{\lambda}}}$. 
It follows from \cite[proof of Lemma 2.2]{BFN18} that its fiber over $z^{\boldsymbol{\la}}$ is equal to $z^{\boldsymbol{\la}}{\bf{N}}_{\cO}/({\bf{N}}_{\cO} \cap z^{\boldsymbol{\la}}{\bf{N}}_{\cO})$ (as a $\widetilde{P}_{\boldsymbol{\la}} \times \bC^\times$-module). So, Chern roots of this bundle are weights of $\widetilde{T} \times \bC^\times$ acting on the quotient $z^{\boldsymbol{\la}}{\bf{N}}_{\cO}/({\bf{N}}_{\cO} \cap z^{\boldsymbol{\la}}{\bf{N}}_{\cO})$.

Fix a weight $\chi \in X^*(\widetilde{T})$. Passing to the $\chi$-weight space of the quotient above, we obtain  
\begin{equation*}
z^{\langle {\boldsymbol{\la}},\chi\rangle} {\bf{N}}_{\cO}(\chi) / ({\bf{N}}_{\cO}(\chi) \cap z^{\langle\boldsymbol{\lambda},\chi\rangle}{\bf{N}}_{\cO}(\chi))
\end{equation*}
(considered as $\bC^\times$-module).
For $\langle {\boldsymbol{\la}},\chi \rangle < 0$ its weights are:
\begin{equation*}
\chi+\langle {\boldsymbol{\la}},\chi \rangle\hbar, \chi+(\langle {\boldsymbol{\la}},\chi \rangle+1)\hbar,\hdots,\chi-\hbar
\end{equation*}
with multiplicity $\operatorname{dim}{\bf{N}}(\chi)$, otherwise the quotient above is zero.
\end{proof}

Assume that $(G,{\bf{N}})$ comes from a quiver theory. We show in Proposition~\ref{monopoles generate K-Coulomb branch} that dressed minuscule monopole operators generate the Coulomb branch algebra (see \cite[Section~3.1]{Wee19} for description of minuscule weights and dressings in this case). Proposition  \ref{upsilon_on_minuscule_gen} thus gives an explicit formula for $\Upsilon$ on the generators.

\begin{rem}
Let us see that the formula \eqref{image_minuscule} is compatible with the Atiyah--Bott localization theorem. Computations for arbitrary ${\bf{N}}$ and ${\bf{N}}=0$ are the same so we assume that ${\bf{N}}=0$.
We can assume that $M^\times_{{\boldsymbol{\lambda},p}}$ is the class of some $\widetilde{G} \times \mathbb{C}^\times$-equivariant vector bundle $\mathcal{E}$ on $\operatorname{Gr}^{\boldsymbol{\lambda}}_{G}=G/P_{\boldsymbol{\lambda}}$. Then: 
\begin{equation*}\label{local_equality}
\Upsilon([\mathcal{E}])=\operatorname{Td}^{\widetilde{G} \times \mathbb{C}^\times}(T\operatorname{Gr}^{\boldsymbol{\lambda}}_G)\operatorname{ch}^{\widetilde{G} \times \mathbb{C}^\times}([\mathcal{E}]) \cdot [\operatorname{Gr}^{\boldsymbol{\lambda}}_G].
\end{equation*}
On the other hand, by the localization theorem:
\begin{equation*}
[\mathcal{E}] = \sum_{c \in (\operatorname{Gr}^{\boldsymbol{\lambda}}_{G})^{T}} \frac{[\mathcal{E}|_c]}{\operatorname{eu}_{\widetilde{T} \times \mathbb{C}^\times}(T_c\operatorname{Gr}_G^{\boldsymbol{\lambda}})}.
\end{equation*}
The map $\Upsilon$ then sends it to:
\begin{equation*}
\sum_{c \in (\operatorname{Gr}^{\boldsymbol{\lambda}}_{G})^{T}} \frac{\operatorname{ch}^{\widetilde{T} \times \mathbb{C}^\times}([\mathcal{E}|_c]) \cdot [c]}{\operatorname{ch}^{\widetilde{T} \times \mathbb{C}^\times}( \operatorname{eu}_{\widetilde{T} \times \mathbb{C}^\times}(T_c\operatorname{Gr}_G^{\boldsymbol{\lambda}}))} = \sum_{c \in (\operatorname{Gr}^{\boldsymbol{\lambda}}_{G})^{T}} \frac{\operatorname{Td}^{\widetilde{T} \times \mathbb{C}^\times}(T_c \operatorname{Gr}^{\boldsymbol{\lambda}}_G)}{\operatorname{eu}_{\widetilde{\mathfrak{t}} \oplus \mathbb{C}}(T_c\operatorname{Gr}^{\boldsymbol{\lambda}}_G)}\operatorname{ch}^{\widetilde{T} \times \mathbb{C}^\times}([\mathcal{E}|_c]) \cdot [c], 
\end{equation*}
as desired.
\end{rem}

\subsubsection{Abelian example} \label{subsubsec: formulas for upsilon in abelian case}

Assume that $G=T$, let's recall the explicit description of the algebras $\cA^\times$, $\cA$ in this case. The elements $r^{\times}_{\boldsymbol{\la}}$, $r_{\boldsymbol{\la}}$ (${\boldsymbol{\la}} \in X^*(T)$) form a basis of $\cA^\times$, $
\cA$ over $K_{\widetilde{T} \times \bC^\times}(\operatorname{pt})$, $H_{\widetilde{T} \times \bC^\times}(\operatorname{pt})$ respectively. Let us recall the relations. First of all recall that 
\begin{align*}
r_{\boldsymbol{\la}}^\times &= [\cO_{\cR_{\boldsymbol{\la}}}], & r_{\boldsymbol{\la}}=[\cR_{\boldsymbol{\la}}] &= [z^{\boldsymbol{\la}}{\bf{N}}_{\cO} \cap {\bf{N}}_{\cO}].
\end{align*}
We consider $\cR_{\boldsymbol{\lambda}}$ as a subspace of $\mathcal{T}_{\boldsymbol{\lambda}}=z^{\boldsymbol{\lambda}}{\bf{N}}_{\cO}$. Embeddings 
\begin{align*}
\cA^\times &\hookrightarrow K^{\widetilde{T} \times \bC^\times}(\operatorname{Gr}_T),& \cA &\hookrightarrow H^{\widetilde{T} \times \bC^\times}_*(\operatorname{Gr}_T)
\end{align*}
are given by 
\begin{align*}
r_{\boldsymbol{\lambda}}^\times &\mapsto \operatorname{eu}_{\widetilde{T} \times\bC^\times}(z^{\boldsymbol{\lambda}}{\bf{N}}_{\cO}/(z^{\boldsymbol{\la}}{\bf{N}}_{\cO} \cap {\bf{N}}_{\cO})) \cdot  [\cO_{z^{\boldsymbol{\la}}}],& r_{\boldsymbol{\lambda}} &\mapsto \operatorname{eu}_{\widetilde{\mathfrak{t}} \oplus \bC}(z^{\boldsymbol{\lambda}}{\bf{N}}_{\cO}/(z^{\boldsymbol{\la}}{\bf{N}}_{\cO} \cap {\bf{N}}_{\cO})) \cdot [{z^{\boldsymbol{\la}}}],
\end{align*}
where by $\operatorname{eu}_{\widetilde{T} \times \bC^\times}(?)$,  we denote the product of $1-f_{-\chi} \in K_{\widetilde{T} \times \bC^\times}(\on{pt})$, where $\chi$ runs through $\widetilde{T} \times \bC^\times$-weights of $?$ (compare with \cite[Lemma 5.4.9]{CG97}).
By 
$\operatorname{eu}_{\widetilde{\mathfrak{t}} \oplus \bC}(?)$  we denote the product of 
$\widetilde{\mathfrak{t}} \oplus \bC$-weights of $?$ considered as elements of $H^*_{\widetilde{T} \times \bC^\times}(\on{pt})$
(compare with \cite[Proposition 2.6.43]{CG97}).

Set:
\begin{equation*}
E(\boldsymbol{\la}) := \operatorname{eu}_{\widetilde{T} \times \bC^\times}(z^{\boldsymbol{\la}}{\bf{N}}_{\cO}/({\bf{N}}_{\cO} \cap z^{\boldsymbol{\la}}{\bf{N}}_{\cO})) \in K_{\widetilde{T} \times \bC^\times}(\operatorname{pt}), 
\end{equation*}
\begin{equation*}
e(\boldsymbol{\la}) := \operatorname{eu}_{\widetilde{\mathfrak{t}} \oplus \bC}(z^{\boldsymbol{\la}}{\bf{N}}_{\cO}/({\bf{N}}_{\cO} \cap z^{\boldsymbol{\la}}{\bf{N}}_{\cO})) \in H_{\widetilde{T} \times \bC^\times}(\operatorname{pt}).
\end{equation*}
Also, for ${\boldsymbol{\la}} \in X^*(\widetilde{T})$ let's introduce the ``shift'' automorphisms:
\begin{align*}
S_{\boldsymbol{\la}} &\curvearrowright K_{\widetilde{T} \times \bC^\times}(\on{pt}),& s_{\boldsymbol{\la}} &\curvearrowright H_{\widetilde{T} \times \bC^\times}(\on{pt}),
\end{align*}
uniquely determined by:
\begin{align*}
S_{\boldsymbol{\la}}(f_\chi) &= q^{\langle\chi,\boldsymbol{\la}\rangle}f_{\chi},& s_{\boldsymbol{\la}}(\chi) &= \chi+\hbar\langle\chi,\boldsymbol{\la}\rangle.
\end{align*}

It follows from definitions that for $f \in K_{\widetilde{T} \times \bC^\times}(\operatorname{pt})$, $x \in H^*_{\widetilde{T} \times \bC^\times}(\on{pt})$ we have:
\begin{align*}
r_{\boldsymbol{\la}}^\times \star f &= S_{\boldsymbol{\la}}(f)r_{\boldsymbol{\la}}^\times,& r_{\boldsymbol{\la}} * x &= s_{\boldsymbol{\la}}(x)r_{\boldsymbol{\la}}.
\end{align*}
In other words, automorphism $S_{\boldsymbol{\la}}$ is the conjugation by $r_{\boldsymbol{\la}}^\times$ and the automorphism $s_{\boldsymbol{\la}}$ is the conjugation by $r_{\boldsymbol{\la}}$.

We then conclude that: 
\begin{align}\label{product_coulomb_geom}
r_{\boldsymbol{\la}}^\times \star r_{\boldsymbol{\mu}}^\times &= 
\Big(\frac{E(\boldsymbol{\la}) S_{\boldsymbol{\la}}(E(\boldsymbol{\mu}))}{E(\boldsymbol{\lambda}+\boldsymbol{\mu})}\Big) \cdot r_{\boldsymbol{\la}+\boldsymbol{\mu}}^\times,& r_{\boldsymbol{\la}} * r_{\boldsymbol{\mu}} &= \Big(\frac{e(\boldsymbol{\la})s_{\boldsymbol{\la}}(e(\boldsymbol{\mu}))}{e(\boldsymbol{\la}+\boldsymbol{\mu})}\Big) \cdot r_{\boldsymbol{\la}+\boldsymbol{\mu}}.
\end{align}

\begin{rem}
Explicitly, the relations for $\cA$ are given by (see \cite[Section 4(iii)]{BFN18}): 
\begin{equation*}
r_{\boldsymbol{\la}} * r_{\boldsymbol{\mu}} = \prod_{\langle \chi,\boldsymbol{\la}\rangle > 0 > \langle \chi,\boldsymbol{\mu}\rangle} \prod_{j=1}^{d(\langle \chi,{\boldsymbol{\la}}\rangle, \langle \chi,{\boldsymbol{\mu}}\rangle)} (\chi+(\langle \chi,{\boldsymbol{\la}}\rangle -j)\hbar) \prod_{\langle \chi,\boldsymbol{\la}\rangle < 0 < \langle \chi,\boldsymbol{\mu}\rangle}\prod_{j=0}^{d(\langle \chi,{\boldsymbol{\la}}\rangle, \langle \chi,{\boldsymbol{\mu}}\rangle)-1}(\chi+(\langle\chi,\boldsymbol{\la}\rangle+j)\hbar)r_{\boldsymbol{\la}+\boldsymbol{\mu}},
\end{equation*}
\begin{equation*}
r_{\boldsymbol{\la}} \chi = (\chi+\langle\chi,\boldsymbol{\la}\rangle\hbar)r_{\boldsymbol{\la}}=s_{\boldsymbol{\la}}(\chi)r_{\boldsymbol{\la}}.
\end{equation*}
Here the first and third products range over weights $\chi$ of ${\bf{N}}$, with multiplicity. These are weights for the action of $\widetilde{T}$. Also, $d\colon \bZ \times \bZ \rightarrow \bZ_{\geqslant 0}$ is the function defined by:
\begin{equation*}
d(a,b) = \begin{cases} 
 0, & \text{if}~a,b~\text{have the same sign;} \\
  \operatorname{min}(|a|,|b|), & \text{otherwise}.
\end{cases}
\end{equation*}

The relations for $\cA^\times$ are given by:
\begin{align*}
r_{\boldsymbol{\la}}^\times \star r_{\boldsymbol{\mu}}^\times = \\
=\prod_{\langle \chi,\boldsymbol{\la}\rangle > 0 > \langle \chi,\boldsymbol{\mu}\rangle} \prod_{j=1}^{d(\langle \chi,{\boldsymbol{\la}}\rangle, \langle \chi,{\boldsymbol{\mu}}\rangle)} (1-f_{-\chi} & \cdot q^{(\langle -\chi,{\boldsymbol{\la}}\rangle + j)}) \prod_{\langle \chi,\boldsymbol{\la}\rangle < 0 < \langle \chi,\boldsymbol{\mu}\rangle}\prod_{j=0}^{d(\langle \chi,{\boldsymbol{\la}}\rangle, \langle \chi,{\boldsymbol{\mu}}\rangle)-1}(1-f_{-\chi} \cdot q^{(\langle-\chi,\boldsymbol{\la}\rangle-j)})r_{\boldsymbol{\la}+\boldsymbol{\mu}}^\times, \\
r_{\boldsymbol{\la}}^\times \star f_\chi = f_\chi & \cdot q^{\langle\chi,\boldsymbol{\la}\rangle}r_{\boldsymbol{\la}}^\times=S_{\boldsymbol{\la}}(f_\chi)r_{\boldsymbol{\la}}^\times.
\end{align*}
\end{rem}

Let us now set:
\begin{equation*}
\operatorname{Td}({\boldsymbol{\la}})=\operatorname{Td}^{\widetilde{T} \times \bC^\times}(z^{\boldsymbol{\la}}{\bf{N}}_{\cO}/({\bf{N}}_{\cO} \cap z^{\boldsymbol{\la}}{\bf{N}}_{\cO})).
\end{equation*}

\begin{rem}
Explicitly, we have:
\begin{equation*}
\operatorname{Td}({\boldsymbol{\la}})^{-1} = \prod_{\langle \chi,{\boldsymbol{\la}}\rangle < 0} \prod_{j=1}^{-\langle\chi,\boldsymbol{\la}\rangle-1} \frac{\chi+(\langle \chi,\boldsymbol{\la}+j)\hbar}{1-e^{-\chi-(\langle \chi,\boldsymbol{\la}+j)\hbar}},~\operatorname{Td}({\boldsymbol{\la}})^{-1} = \prod_{\langle \chi,{\boldsymbol{\la}}\rangle < 0} \prod_{j=1}^{-\langle\chi,\boldsymbol{\la}\rangle-1} \frac{1-e^{-\chi-(\langle \chi,\boldsymbol{\la}+j)\hbar}}{\chi+(\langle \chi,\boldsymbol{\la}+j)\hbar}.
\end{equation*}
\end{rem}

It follows from definitions that the map $\Upsilon$ is given by 
\begin{equation*}
\Upsilon(r_{\boldsymbol{\la}}^\times) = \operatorname{Td}(\boldsymbol{\la})^{-1} \cdot r_{\boldsymbol{\la}}.
\end{equation*}

Let's check ``by hands'' that $\Upsilon$ is indeed a homomorphism of algebras.

By the very definition we have:
\begin{equation}\label{main_prop_Td}
\operatorname{ch}^{\widetilde{T} \times \bC^\times}(E(\boldsymbol{\la})) = \operatorname{Td}(\boldsymbol{\la})^{-1} \cdot e(\boldsymbol{\la}).
\end{equation}
We also have
\begin{equation}\label{comp_S_s}
\on{ch}^{\widetilde{T} \times \bC^\times}(S_{\boldsymbol{\la}}(f))=s_{\boldsymbol{\la}}(\on{ch}^{\widetilde{T} \times \bC^\times}(f)),~f \in K_{\widetilde{T} \times \bC^\times}(\operatorname{pt}).
\end{equation}
To prove (\ref{comp_S_s}), it's enough to assume that $f=f_\chi$ for some $\chi \in \widetilde{T}$ and then:
\begin{multline*}
\on{ch}^{\widetilde{T} \times \bC^\times}(S_{\boldsymbol{\la}}(f_\chi))=\on{ch}^{\widetilde{T} \times \bC^\times}(q^{\langle \chi,\boldsymbol{\la}\rangle}f_\chi)=\on{ch}^{\widetilde{T} \times \bC^\times}(q^{\langle \chi,\boldsymbol{\la}\rangle})\on{ch}^{\widetilde{T} \times \bC^\times}(f_\chi)=
\\
=e^{\langle \chi,\boldsymbol{\la}\rangle\hbar}e^{\chi}=e^{\chi+\langle \chi,\boldsymbol{\la}\rangle\hbar}=s_{\boldsymbol{\la}}(\operatorname{ch}^{\widetilde{T} \times \bC^\times}(\chi)).
\end{multline*}

\begin{rem}
Note that the equation (\ref{main_prop_Td}) determines $\operatorname{Td}({\boldsymbol{\la}})$ uniquely. 
For any $\widetilde{T}$-equivariant line      bundle $\mathcal{L}$ one has
\begin{equation*}
1-\operatorname{ch}^{\widetilde{T}}(\mathcal{L}^\vee) = \operatorname{Td}(\mathcal{L})^{-1}   \cdot c_1^{\widetilde{T}}(\mathcal{L}),
\end{equation*}
so $\operatorname{Td}(-)$ ``measures'' the difference between  $\operatorname{ch}^{\widetilde{T}}(-)$ and $c_1^{\widetilde{T}}(-)$.  
\end{rem}

Using (\ref{product_coulomb_geom}), (\ref{main_prop_Td}), and (\ref{comp_S_s}) we see that:
\begin{multline*}
\Upsilon(r_{\boldsymbol{\la}}^\times \star r_{\boldsymbol{\mu}}^{\times})=\Upsilon\Big(\frac{E(\boldsymbol{\la})S_{\boldsymbol{\la}}(E(\boldsymbol{\mu}))}{E(\boldsymbol{\la}+\boldsymbol{\mu})}\cdot r^{\times}_{\boldsymbol{\la}+\boldsymbol{\mu}}\Big)=\\
=\frac{\operatorname{ch}^{\widetilde{T} \times \bC^\times}(E(\boldsymbol{\la})) \operatorname{ch}^{\widetilde{T} \times \bC^\times}(S_{\boldsymbol{\la}}(E(\boldsymbol{\mu})))}{\operatorname{ch}^{\widetilde{T} \times \bC^\times}(E(\boldsymbol{\la}+\boldsymbol{\mu}))}\on{Td}(\boldsymbol{\la}+\boldsymbol{\mu})^{-1} \cdot r_{\boldsymbol{\la}+\boldsymbol{\mu}}= \\
= \frac{\operatorname{ch}^{\widetilde{T} \times \bC^\times}(E(\boldsymbol{\la})) \operatorname{ch}^{\widetilde{T} \times \bC^\times}(S_{\boldsymbol{\la}}(E(\boldsymbol{\mu})))}{e(\boldsymbol{\la}+\boldsymbol{\mu})}\cdot r_{\boldsymbol{\la}+\boldsymbol{\mu}},
\end{multline*}
\begin{multline*}
\Upsilon(r_{\boldsymbol{\la}}^\times) * \Upsilon(r_{\boldsymbol{\mu}}^\times)=(\on{Td}(\boldsymbol{\la})^{-1} \cdot r_{\boldsymbol{\la}}) * (\on{Td}(\boldsymbol{\mu})^{-1} \cdot r_{\boldsymbol{\mu}}) = \frac{\on{Td}(\boldsymbol{\la})^{-1}e(\boldsymbol{\la}) \cdot s_{\boldsymbol{\la}}( \on{Td}(\boldsymbol{\mu})^{-1}e(\boldsymbol{\mu}))}{e(\boldsymbol{\la}+\boldsymbol{\mu})} \cdot r_{\boldsymbol{\la}+\boldsymbol{\mu}}= \\
=\frac{\operatorname{ch}^{\widetilde{T} \times \bC^\times}(E(\boldsymbol{\la})) s_{\boldsymbol{\la}}(\operatorname{ch}^{\widetilde{T} \times \bC^\times}(E(\boldsymbol{\mu})))}{e(\boldsymbol{\la}+\boldsymbol{\mu})} \cdot r_{\boldsymbol{\la}+\boldsymbol{\mu}}=\frac{\operatorname{ch}^{\widetilde{T} \times \bC^\times}(E(\boldsymbol{\la})) \operatorname{ch}^{\widetilde{T} \times \bC^\times}(S_{\boldsymbol{\la}}(E(\boldsymbol{\mu})))}{e(\boldsymbol{\la}+\boldsymbol{\mu})}\cdot r_{\boldsymbol{\la}+\boldsymbol{\mu}}.\end{multline*}
We conclude that $\Upsilon(r_{\boldsymbol{\la}}^\times \star r_{\boldsymbol{\mu}}^{\times})$ is indeed equal to $\Upsilon(r_{\boldsymbol{\la}}^\times) * \Upsilon(r_{\boldsymbol{\mu}}^\times)$.

\subsubsection{Case of ADE quivers} \label{subsubseq: ade quivers and completions}
Recall that homological Coulomb branches for finite type ADE quivers are isomorphic to truncated shifted Yangians \cite[Appendix~B]{bfn_slices}, while K-theoretic ones are isomorphic to truncated shifted quantum affine groups \cite{FT19a, FT19b}.

Note that for the non-shifted case, Gautam and Toledano Laredo constructed an isomorphism between completions of Yangian and the quantum loop group \cite{GTL13}. 
It would be very interesting to investigate the relationship between our construction and theirs.
We expect that the isomorphism of \cite{GTL13} restricts to truncations and should be closely related to ours.

{An analogous statement for the shifted case has not appeared in the literature to the best of our knowledge. We expect that the results of this section could help study this question.}

It would be also very interesting to investigate shifted analogs of  \cite{GTL17}.

Note that the comparison of categories $\cO$ for homological and K-theoretic Coulomb branches could be deduced without usage of completions, combining results of \cite{Web24} and \cite{VV25} (see also \cite[Appendix~B]{NW23}).


Finally, let us note that it should be straightforward to extend Theorem \ref{thm_iso_after_completions} to Coulomb branches with symmetrizers \cite{NW23}. This would help to investigate analogous connection between Yangians and quantum loop groups for non-simply laced types.



\section{K-theoretic Hikita conjecture} \label{sec: K-hikita}

\subsection{Statement of the conjecture} \label{subsection: K-hikita}

We now suggest the multiplicative (K-theoretic, trigonometric) variant of the Hikita conjecture. Note that similar ideas to modify the Hikita conjecture in this way appeared in \cite{Zho23, LZ22}.
Roughly speaking, in Conjecture \eqref{hikita conjecture}, in the LHS one needs to replace the equivariant cohomology by the equivariant K-theory, and in the RHS, one needs to replace the Coulomb branch by the K-theoretic Coulomb branch.

We return to notations of Section~\ref{section: homo hikita} and work with $(G_\bv, {\bf{N}})$ of quiver type.
The K-theoretic Coulomb branch is defined as $\cM^\times(G_\bv, {\bf{N}}) = \Spec K^{G_\bv} (\cR_{G_\bv, {\bf{N}}})$. From physics perspective, it stands for the Coulomb branch of $4d$ $\cN = 2$ supersymmetric gauge theory (as opposed to homological variant, which stands for $3d$ $\cN = 4$ theory). It admits a Poisson deformation over a flavor torus $F$, defined as $\cM^\times(G_\bv, {\bf{N}})_{F} = \Spec K^{G_\bv \times F} (\cR_{G_\bv, {\bf{N}}})$.
It also admits the quantization $\cA^\times(G_\bv, {\bf{N}})_{F} = K^{G_\bv \times F \times \torus_\hbar} (\cR_{G_\bv, {\bf{N}}})$. Same as in the homological case, the torus $H \simeq (\torus)^{\pi_0(\cR_{G_\bv, {\bf{N}}})} \simeq (\torus)^{|Q_0|}$ acts on $\cA^\times(G_\bv, {\bf{N}})_{F}$. We keep having the chosen $H$-character~$\nu$, used to construct the resolved quiver variety $\wti \fM_Q$.

When no confusion arise, we denote $\cM^\times(G_\bv, {\bf{N}})$ by $\cM^\times_Q$ and similarly to other related varieties and algebras.

\begin{conj}[K-theoretic Hikita conjecture] \label{k-hikita conjecture}
There is an isomorphism of algebras over $\bC[(T_{\bf{v}} / S_\bv) \times F] \T \bC[q^{\pm 1}]$
\begin{equation} \label{quantized K-theoretic hikita conjecture}
K^{F \times \torus_\hbar} (\wti \fM_Q) \simeq B^\nu(\cA_{Q, F}^\times).
\end{equation} 
In particular, specializing at $q = 1$, there is an isomorphism of algebras over $\bC[(T_\bv / S_\bv) \times F]$
\begin{equation} \label{deformed K-theoretic hikita conjecture}
K^{F} (\wti \fM_Q) \simeq \bC[(\cM^\times_{Q, F})^{\nu}].
\end{equation}
Further specializing at $1 \in F$, there is an isomorphism of algebras over $\bC[T_\bv / S_\bv]$
\begin{equation} \label{nonquantized nondeformed K-theoretic hikita}
K(\wti \fM_Q) \simeq \bC[(\cM^\times_{Q})^{\nu}].
\end{equation}
\end{conj}

The variable $q$ in \eqref{quantized K-theoretic hikita conjecture} should be thought as the coordinate on $\torus_\hbar$.


Similarly to the homological case, we refer to \eqref{quantized K-theoretic hikita conjecture} as to {\it quantized K-theoretic Hikita conjecture}, to \eqref{deformed K-theoretic hikita conjecture} as to {\it equivariant K-theoretic Hikita conjecture}, and to \eqref{nonquantized nondeformed K-theoretic hikita} as to {\it K-theoretic Hikita conjecture}.

The $K_{G_\bv \times F \times \torus_\hbar} (\pt) = \bC[(T_\bv / S_\bv) \times F] \T \bC[q^{\pm 1}]$-action, mentioned in Conjecture \ref{k-hikita conjecture}, appears on both sides similarly to the homological case, discussed in Section \ref{subsection: homological Hikita} in detail.

In the present paper, we deal with the equivariant (non-quantized) version of the conjecture, that is \eqref{deformed K-theoretic hikita conjecture}. We hope to return to the quantized case one day.

\subsection{Formal completions of K-theoretic Hikita conjecture} \label{subsec: completions of K-hikita}

In this section, we investigate what are the formal completions of both sides of \eqref{deformed K-theoretic hikita conjecture} over a maximal ideal of $K_{G_\bv \times F}(\pt)$.

Recall that for a maximal ideal $\fm$ of an algebra $A$ and an $A$-module $M$, we denote by $M^{\wedge \fm}$ the completion of $M$ at $\fm$. By $M_{\fm}$ we denote the localization of $M$ at $\fm$.

We begin with the quiver variety side.
First we relate the completion of K-theory of a quiver variety to the completion of cohomology (an analogous statement for the Coulomb side is the main result of Section \ref{sec: Riemann-Roch}):

\begin{lem} \label{chern induces isomorphism in completion for quiver varieties}
The Chern character induces an isomorphism of algebras
\begin{equation*}
K^{F}(\wti \fM(G_\bv, {\bf{N}}))^{\wedge (1, 1)} \simeq H_{F}(\wti \fM(G_\bv, {\bf{N}}))^{\wedge (0, 0)}.
\end{equation*}
The completions are at $(1, 1) \in (T_\bv / S_\bv) \times F$ on the LHS and at $(0, 0) \in (\ft_\bv / S_\bv) \times \ff$ on the RHS.
\end{lem}
\begin{proof}
It follows from \cite{eg1} that there exists an isomorphism 
\begin{equation*}
\tau^{F}\colon K^{F}(\wti \fM(G_\bv, {\bf{N}}))^{\wedge (1, 1)} \iso H^{F}_*(\wti \fM(G_\bv, {\bf{N}}))^{\wedge (0, 0)}.
\end{equation*}
Note now that $\wti \fM(G_\bv, {\bf{N}})$ is smooth, so $\tau^{F}$ differs from $\operatorname{ch}^{F}$ by the multiplication by some invertible class (here we also use the identifications $H^{F}_*(\wti \fM(G_\bv, {\bf{N}}))^{\wedge (0, 0)} \simeq H_{F}^{*}(\wti \fM(G_\bv, {\bf{N}}))^{\wedge (0, 0)}$, $K^F(\wti \fM(G_\bv, {\bf{N}}))^{\wedge (1,1)}=K_F(\wti \fM(G_\bv, {\bf{N}}))^{\wedge (1,1)}$). The fact that $\tau^{F}$ is an isomorphism then implies that $\operatorname{ch}^{F}$ is an isomorphism.  
\end{proof}

\begin{prop} \label{completion of K-theory of quiver variety at any point}
For any $(t_\bv, f) \in (T_\bv / S_\bv) \times F$, there is an isomorphism of algebras
\begin{equation*}
K^{F} (\wti \fM(G_\bv, {\bf{N}}))^{\wedge (t_\bv, f)} \simeq H_{F} (\wti \fM({Z_{G_\bv}(t_\bv), {\bf{N}}^{(t_\bv, f)}}))^{\wedge (0, 0)}.
\end{equation*}
\end{prop}
\begin{proof}
\begin{multline*} 
K^{F} (\wti \fM(G_\bv, {\bf{N}}))^{\wedge (t_\bv, f)} = K^{G_\bv \times F} (\mu_{G_\bv, {\bf{N}}}^{-1}(0)^{s})^{\wedge (t_\bv, f)} = 
K^{Z_{G_\bv}(t_\bv) \times F} ((\mu_{G_\bv, {\bf{N}}}^{-1}(0)^{s})^{(t_\bv, f)})^{\wedge (1,1)} \simeq \\ 
K^{F} (\wti \fM({Z_{G_\bv}(t_\bv), {\bf{N}}^{(t_\bv, f)}}))^{\wedge (1, 1)} \simeq
H_{F} (\wti \fM({Z_{G_\bv}(t_\bv), {\bf{N}}^{(t_\bv, f)}}))^{\wedge (0, 0)}.
\end{multline*}
We used the localization theorem, computation of $(\mu_{G_\bv, {\bf{N}}}^{-1}(0)^{s})^{(t_\bv, f)}$ done during the proof of Proposition~\ref{localization of quiver side of homological hikita}, and Lemma \ref{chern induces isomorphism in completion for quiver varieties}. We also use the twisting by the central element as in \cite[Section 5.2]{eg2} (compare with Remark \ref{rem:twist iso fibers} above).
\end{proof}

Now let us turn to the Coulomb branches side.

\begin{rem}
In what follows, we consider completions of Coulomb branch algebras. Coulomb branch algebras are defined as the inductive limit of homology or K-theory of schemes $\cR_{\leqslant \la}$. As in~\eqref{eq: colimit of completions of coulombs}, by a completion of a Coulomb algebra, we mean the colimit of completions (as oppose to completion of colimit). We omit this in our notations, and just write $\bC[\cM(G_\bv, {\bf{N}})^\times_{F}]^{\wedge (t_\bv, f)}$ and similar, but one should keep in mind this subtlety. 
\end{rem}

\begin{prop} \label{completion of K-coulomb branch}
For any $(t_\bv, f) \in (T_\bv / S_\bv) \times F$, there is an isomorphism of algebras over $K^{G_\bv \times F}(\pt)^{\wedge (t_\bv, f)} \simeq K^{Z_{G_\bv}(t_\bv) \times F}(\pt)^{\wedge (1, 1)}$:
\begin{equation*}
\bC[\cM(G_\bv, {\bf{N}})^\times_{F}]^{\wedge (t_\bv, f)} \simeq \bC[\cM(Z_{G_\bv}(t_\bv), {\bf{N}}^{(t_\bv, f)})_{F}]^{\wedge (0, 0)}.
\end{equation*}
\end{prop}
Note that it is K-theoretic Coulomb branch on the LHS of the above claim, and homological one on the RHS.
\begin{proof}
By the localization theorem combined with the twisting \cite[Section 5.2]{eg2}, we have: 
\begin{equation*}
K^{G_\bv \times F} (\cR_{G_\bv, {\bf{N}}})^{\wedge (t_\bv, f)} = K^{Z_{G_\bv}(t_\bv) \times F} ((\cR_{G_\bv, {\bf{N}}})^{(t_\bv, f)})^{\wedge (1, 1)}.
\end{equation*}
Now apply Lemma \ref{fixed points on BFN space} and Theorem~\ref{thm_iso_after_completions}.
\end{proof}


\begin{cor} \label{completion of B-algebra at any point}
For any $(t_\bv, f) \in (T_\bv / S_\bv) \times F$ there is an isomorphism of algebras
\begin{equation*}
\bC[(\cM(G_\bv, {\bf{N}})^{\times}_{F})^\nu]^{\wedge (t_\bv, f)} \simeq \bC[(\cM(Z_{G_\bv}(t_\bv), {\bf{N}}^{(t_\bv, f)})^{\times}_{F})^\nu]^{\wedge (1, 1)}.
\end{equation*}
\end{cor}
\begin{proof}
Unlike localization, the completion is not exact in general, so we need to be more careful than in the proof of Corollary \ref{cor: localization of b-algebra}. We show that taking B-algebra commutes with taking completion.

Let $J$ be the ideal of $\bC[\cM(G_\bv, {\bf{N}})^{\times}_{F}]$, generated by all elements of the form $(a - \nu(a))$, so that $\bC[\cM(G_\bv, {\bf{N}})^{\times}_{F}] / J \simeq \bC[(\cM(G_\bv, {\bf{N}})^{\times}_{F})^\nu]$ (coinvariants ideal). Let $\{a_i\}$ be the generators of this ideal (one can take $a_i$ to be dressed minuscule monopole operators, but we do not use it). Then we have a right exact sequence
\begin{equation} \label{eq: ses for taking b-algebra}
\bigoplus_i \bC[\cM(G_\bv, {\bf{N}})^{\times}_{F}] \rightarrow \bC[\cM(G_\bv, {\bf{N}})^{\times}_{F}] \rightarrow \bC[(\cM(G_\bv, {\bf{N}})^{\times}_{F})^\nu] \rightarrow 0
\end{equation}
Note that the first and the second terms in \eqref{eq: ses for taking b-algebra} are free over $K^{G_\bv \times F}(\pt)$ by Lemma~\ref{lem: coulomb is free over cartan}, while the third term is finitely generated over it. Hemce, for all three terms, completion at $(t_\bv, f)$ coincides with taking the tensor product $- \otimes_{K_{G_\bv \times T_\bw}(\pt)} K_{G_\bv \times T_\bw}(\pt)^{\wedge (t_\bv, f)}$. Since tensor product is right exact, we obtain that the completion of \eqref{eq: ses for taking b-algebra} is right exact.

Since we take completion over $K_{G_\bv \times F}(\pt)$, which is a subalgebra lying in $\nu$-weight 1, $\{a_i\}$ are also generators of the coinvariants ideal of $\bC[\cM(G_\bv, {\bf{N}})^{\times}_{F}]^{\wedge (t_\bv, f)}$. Hence, taking completion of~\eqref{eq: ses for taking b-algebra} yields that the completion of the B-algebra is the B-algebra of completion.

Now the result follows from Proposition \ref{completion of K-coulomb branch}.
\end{proof}

\subsection{Generators of K-theoretic Coulomb branches} \label{subsec: monopoles generate k-branches}
For the homological case, Weekes in \cite{Wee19} showed that the quantized Coulomb branch of a quiver gauge theory is generated by dressed minuscule monopole operators (this argument also appears in \cite[proof of Theorem 4.32]{FT19b}). We need generation in the K-theoretic case by a particular set of monopoles, which is implicit in \cite{Wee19}.



Recall from Section~\ref{subsubsec: formulas for upsilon on monopoles} the definition of dressed minuscule monopole operators $M^\times_{{\boldsymbol{\la}}, p} \in \cA_{Q, F}^\times$. 
We call a minuscule coweight ${\boldsymbol{\la}}$ of $G_\bv$ {\it positive} if only 0's and 1's appear in its standard expression; we call it {\it negative} if only 0's and $(-1)$'s appear in its standard expression (see \cite[Section~3.1]{Wee19}). Note that there are minuscule coweights which are neither positive nor negative.

\begin{prop} \label{monopoles generate K-Coulomb branch}
The algebra $\cA^\times_{Q, F}$ 
is generated by $K^{G_\bv \times F \times \torus}(\pt)$ and dressed minuscule monopole operators $M^\times_{{\boldsymbol{\la}}, p}$ such that ${\boldsymbol{\la}}$ is positive or negative.
\end{prop}
\begin{proof}
The proof for homological case of \cite[Proposition 3.1]{Wee19} actually shows that positive and negative coweights are sufficient for generating. It works without changes for the case of K-theory, see \cite[proof of Theorem~4.32]{FT19b} and \cite[Theorem~4.2.6]{VV25}. Both these papers go under the running assumption that $Q$ is of finite type, but this particular argument works for arbitrary quiver.


\end{proof}

\begin{rem}
In \cite{BDG17}, the monopole operators $M_{{\boldsymbol{\la}}, p}$ in homological Coulomb branch are considered for all (not necessarily minuscule) weights ${\boldsymbol{\la}}$. Mathematically, they make sense as element of the associated graded to the Coulomb branch algebra (see \cite[Remark 6.5]{BFN18}), and it is unclear if one can canonically lift them to elements of $\bC[\cM_{Q}]$.

At the same time, for K-theoretic Coulomb branches, such elements were constructed in \cite{CW23} as classes of simple objects in the heart of the Koszul-perverse t-structure. They form a basis of $\bC[\cM^\times_{Q}]$.
\end{rem}

Now we propose an application to the K-theoretic Hikita conjecture. We formulate it in the non-quantum case, since that is the case of our interest in what follows.
\begin{cor} \label{corollary: surjections for K-hikita}
Both homomorphisms
\begin{equation} \label{surjections from K-theory of point to both sides}
\begin{tikzcd}
	& {K^{G_{\bold v} \times F}(\mathrm{pt})} \\
	{K^{F}(\widetilde{\mathfrak M}(G_{\bold v}, {\bf{N}}))} && {\mathbb C[(\mathcal M({G_{\bold v}}, {\bf{N}})^\times_{F})^{\nu}]}
	\arrow["{\phi^\times_1}"', from=1-2, to=2-1]
	\arrow["{\phi^\times_2}", from=1-2, to=2-3]
\end{tikzcd}
\end{equation}
are surjective.

In particular, the K-theoretic Hikita conjecture \eqref{deformed K-theoretic hikita conjecture} is equivalent to the claim $\ker \phi_1^\times = \ker \phi_2^\times$.
\end{cor}

\begin{proof}
For $\phi^\times_1$, this is the K-theoretic Kirwan surjectivity, proved for quiver varieties in~\cite{MN18}. 

For $\phi_2^\times$, the proof is the same as for homological case in \cite[Proposition~8.7]{KS25}, using generating property of monopole operators, Proposition~\ref{monopoles generate K-Coulomb branch}.
\end{proof}

\subsection{K-theoretic Hikita conjecture from homological}
The main result of this section is Theorem~\ref{k-hikita from homo-hikita}, which explains how to deduce the K-theoretic Hikita conjecture from homological for a larger set of gauge theories. However, first we go in the opposite direction and show that for a fixed quiver gauge theory our K-theoretic Hikita conjecture is actually a stronger statement than the homological one.

\begin{prop} \label{prop: k-hikita implies homo-hikita}
K-theoretic Hikita conjecture implies homological Hikita conjecture. Namely, if for a particular quiver theory $(G_\bv, {\bf{N}})$ one has an isomorphism $K_{G_\bv \times F}(\pt)$-algebras
\begin{equation*}
K^{F} (\wti \fM(G_\bv, {\bf{N}})) \simeq \bC[(\cM(G_\bv, {\bf{N}})^\times_{F})^{\nu}],
\end{equation*}
then one also has an isomorphism of $H_{G_\bv \times F}(\pt)$-algebras
\begin{equation*}
H_{F} (\wti \fM(G_\bv, {\bf{N}})) \simeq \bC[\cM(G_\bv, {\bf{N}})_{\ff}^{\nu}].
\end{equation*}
\end{prop}

\begin{proof}
As explained in Section \ref{subsection: homological Hikita}, both morphisms
\begin{equation} \label{surjective morphisms for homo-hikita}
\begin{tikzcd}[column sep=tiny]
	{} & H_{G_\bv \times F}(\pt) \\ 
	H_{F}(\widetilde{\mathfrak M}({G_{\bold v}}, {\bf{N}})) && {\mathbb C[\mathcal M({G_{\bold v}}, {\bf{N}})_{{\mathfrak f}}^{\nu}]}
	\arrow["{\phi_1}"', from=1-2, to=2-1]
	\arrow["{\phi_2}", from=1-2, to=2-3]
\end{tikzcd}
\end{equation}
are surjective, and we need to prove that $\ker \phi_1 = \ker \phi_2$. 
Note that $\phi_1$ and $\phi_2$ are graded homomorphisms, and hence it is sufficient to show that they coincide after completion at $(0, 0) \in (\ft_\bv / S_\bv) \times \ff$ (the completion w.r.t. the grading).
Taking the completion at $(1, 1) \in (T_\bv / S_\bv) \times F$ of \eqref{surjections from K-theory of point to both sides}, yields precisely the completed at $(0, 0)$ version of \eqref{surjective morphisms for homo-hikita} due to Corollary~\ref{completion of B-algebra at any point} and Lemma~\ref{chern induces isomorphism in completion for quiver varieties}, hence the result.
\end{proof}

For what follows, we first need a few lemmata from commutative algebra.

\begin{lem} \label{interstion of completion with ring is ideal}
Let $(R, \fm)$ be a Noetherian local ring without zero divisors and let $I \subset R$ be an ideal. Then 
\begin{equation*}
I^{\wedge \fm} \cap R = I,
\end{equation*}
where the intersection is taken in $R^{\wedge \fm}$.
\end{lem}
\begin{proof}
The inclusion $I^{\wedge \fm} \cap R \supset I$ is evident. We show the inclusion in the other direction.

Take $a \in I^{\wedge \fm} \cap R$. Denote $\phi_k: R \rightarrow R/{\fm^k}$ the natural surjection. Since $a \in I^{\wedge \fm}$, we have $\phi_k(a) \in \phi_k(I)$ for any $k$, hence, $a \in (I + \fm^k)$ for any $k$. Let $[a]$ be the image of $a$ in the local ring $(R/I, [\fm])$. We get $[a] \in [\fm]^k$ for any $k$. By the Krull intersection theorem, we have $\bigcap_{k \geq 1} [\fm]^k = 0$. It follows that $[a] = 0$, hence $a \in I$, as required.
\end{proof}

\begin{lem} \label{if ideals coincide at each completion, they coincide}
Let $A$ be a Noetherian integral domain, and $I, J \subset A$ be two ideals. 
Suppose for any maximal $\fm \subset A$, the ideals $I^{\wedge \fm}$ and $ J^{\wedge \fm}$ coincide as ideals of $A^{\wedge \fm}$. Then $I$ and $J$ coincide as ideals of $A$. 
\end{lem}
\begin{proof}
Let $A_\fm \subset A^{\wedge \fm}$ be the localization of $A$ at $\fm$ and $I_{\fm}, J_\fm \subset A_\fm$ be localizations of $I,J$. Applying Lemma~\ref{interstion of completion with ring is ideal} to the local ring $A_\fm$, we get:
\begin{equation*}
I_\fm = I^{\wedge \fm} \cap A_{\fm} = J^{\wedge \fm} \cap A_\fm = J_\fm,
\end{equation*}
where the intersections are taken in $A^{\wedge \fm}$.

It is well-known that $\bigcap_{\fm} A_{\fm} = A$, where $\fm$ runs over all maximal ideals, and the intersection is taken in the fraction field of $A$. One easily sees that $\bigcap_\fm I_\fm = I$ and $\bigcap_\fm J_\fm = J$ inside these intersections. The result follows.
\end{proof}

We know turn to the main result of this section.

\begin{thm} \label{k-hikita from homo-hikita}

Suppose for any $(t_\bv, f) \in (T_\bv / S_\bv) \times F$, there is an isomorphism of $H_{Z_{G_\bv}(t_\bv) \times F}(\pt)$-algebras
\begin{equation*}
H_{F} (\wti \fM(Z_{G_\bv}(t_\bv), {\bf{N}}^{(t_\bv, f)})) \simeq \bC[\cM(Z_{G_\bv}(t_\bv), {\bf{N}}^{(t_\bv, f)})_{\ff}^{\nu}].
\end{equation*}
(homological equivariant Hikita conjecture for $(Z_{G_\bv}(t_\bv), {\bf{N}}^{(t_\bv, f)})$). 

Then there is an isomorphism of $K_{{G_\bv} \times F}(\pt)$-algebras
\begin{equation*}
K^{F} (\wti \fM({G_\bv}, {\bf{N}})) \simeq \bC[(\cM({G_\bv}, {\bf{N}})^\times_{F})^{\nu}]
\end{equation*}
(K-theoretic equivariant Hikita conjecture for $({G_\bv}, {\bf{N}})$).



\end{thm}

\begin{proof}

Due to Corollary \ref{corollary: surjections for K-hikita}, our task is to show that $\ker \phi^\times_1 = \ker \phi^\times_2$ for homomorphisms~\eqref{surjections from K-theory of point to both sides}. By Lemma \ref{if ideals coincide at each completion, they coincide}, it is sufficient to show that $\ker \phi^\times_1$ and $\ker \phi^\times_2$ coincide after completing  at every maximal ideal. That is what we show.

Pick $(t_\bv, f) \in (T_\bv / S_\bv) \times F$. Take completion at this point of \eqref{surjections from K-theory of point to both sides}. By Proposition \ref{completion of K-theory of quiver variety at any point} and Corollary \ref{completion of B-algebra at any point}, it is nothing else than
\[\begin{tikzcd}[column sep=tiny]
	{} & {H_{Z_{G_{\bold v}}(t_{\bold v}) \times F}(\mathrm{pt})^{\wedge (0, 0)}} \\
	{H_{F}(\widetilde{\mathfrak M}(Z_{G_{\bold v}}(t_{\bold v}), {\bf{N}}^{(t_{\bold v}, t_{f})}))^{\wedge (0, 0)}} && {\mathbb C[\mathcal M(Z_{G_{\bold v}}(t_{\bold v}), {\bf{N}}^{(t_{\bold v}, t_{f})})_{{\mathfrak f}}^{\nu}]^{\wedge (0, 0)}}
	\arrow["{\phi_1^{\wedge (t_{\bold v}, t_{f})}}"', from=1-2, to=2-1]
	\arrow["{\phi_2^{\wedge (t_{\bold v}, t_{f})}}", from=1-2, to=2-3].
\end{tikzcd}\]
Now the equality $\ker \phi_1^{\wedge (t_\bv, f)} = \ker \phi_2^{\wedge
(t_\bv, f)}$ is nothing else but the (completed) homological Hikita conjecture for $(Z_{G_\bv}(t_\bv), {\bf{N}}^{(t_\bv, f)})$, which holds by assumptions. The result follows.
\end{proof}

\begin{cor} \label{cor: k-hikita for ade}
Equivariant K-theoretic Hikita conjecture holds for quiver of type ADE under same assumptions as in Corollary \ref{homo-hikita for ADE quivers}.
\end{cor}

\begin{proof}
By Proposition \ref{fixed points of quiver theory is quiver theory} for $(G_\bv, {\bf{N}})$ of ADE type quiver and any $(t_\bv, f) \in T_\bv \times F$, the theory $(Z_{G_\bv}(t_\bv), {\bf{N}}^{(t_\bv, f)})$ is the sum of quiver theories of the same type. Thus Theorem \ref{k-hikita from homo-hikita} and Corollary \ref{homo-hikita for ADE quivers} imply the result.
\end{proof}

\begin{cor}[Weak form of the equivariant K-theoretic Hikita conjecture for Jordan quiver] \label{cor: k-hikita jordan quiver}
Let $Q$ be the Jordan quiver, take any dimension number $v$ and framing number $w$. Take the flavor torus $T_w = (\torus)^w \subset GL_w$. The K-theoretic Hikita conjecture holds for this data.
\end{cor}
\begin{proof}
By Proposition \ref{fixed points of quiver theory is quiver theory}, only quivers of the same type appear as the fixed points set. So, using Theorem \ref{k-hikita from homo-hikita}, we deduce the result from the homological case, which is the main result of \cite{KS25}.
\end{proof}

We call it the weak form of the conjecture because we do not include the additional one-dimensional torus $\torus_{\mathrm{loop}}$ in the flavor group (see proof of Corollary \ref{cor: hiktia for A-quivers}). If we include this torus, then affine type A quivers appear as the fixed points, see \cite[Proposition~6]{KS22}. If one proves the homological Hikita for affine type A, then one gets the strong form of K-theoretic Hikita. Conversely, if one proves K-theoretic Hikita conjecture for the full flavor torus for the Jordan quiver, then one gets both the homological and K-theoretic Hikita conjecture for affine type A, using a K-theoretic version of Proposition \ref{prop: hikita from one theory to its centralizer} and Proposition \ref{prop: k-hikita implies homo-hikita}.

\subsection{Application: torus fixed points on Coulomb branches} \label{subsec: corollaries of K-Hikita}
K-theoretic Hikita conjecture connects geometry of quiver varieties and K-theoretic Coulomb branches in a non-trivial way. While K-theory of quiver varieties is a well-studied object, relatively little is known about geometry of K-theoretic Coulomb branches. We expect our conjecture to give new information about it. Below we list some immediate applications to the description of torus fixed points on deformed Coulomb branches. We fix $Q, \bv, \bw$, and denote the corresponding quiver variety $\fM$, and deformations of Coulomb branches $\cM_{\ff}, \cM^\times_{F}$. 
From now on and until the end of this section we assume both homological and K-theoretic Hikita conjecture to hold for $(Q, \bv, \bw)$ (for example, $Q$ may be of type ADE, and $\bw$ subject to assumptions of Corollary~\ref{homo-hikita for ADE quivers}). We also assume that $\nu$ is the character of $G_{\bf{v}}$ given by the product of determinants and that ${\bf{w}} \neq 0$.

\subsubsection{Fixed points: non-deformed case}

First, recall a general result of Crawley-Boevey \cite[Section~1]{CB}: 
the quiver variety $\mathfrak{M}$ is always connected. 

For $Q$ without edge loops we recall a representation-theoretic characterization of ${\bf{v}}, {\bf{w}}$ such that the corresponding quiver variety $\mathfrak{M}$ is nonempty.

Let $\mathfrak{g}_Q$ be the symmetric Kac-Moody Lie algebra corresponding to $Q$. Let $\alpha_i$, $\omega_i$ ($i \in Q_0$) be simple roots and fundamental weights for $\mathfrak{g}_Q$. Set:
\begin{equation*}
\lambda = \sum_{i \in Q_0}w_i\omega_i, \qquad \mu = \lambda-\sum_{i \in Q_0} v_i\alpha_i.
\end{equation*}
Let $V(\lambda)$ be the integrable highest weight representation of $\mathfrak{g}_Q$ with highest weight $\lambda$. Let $V(\lambda)_\mu$ be the weight $\mu$ subspace of $V(\lambda)$. It follows from the main results of \cite{Nak98}, as well as \cite[Theorem 2.15]{nak branching} (see also \cite[Remark~3.5]{HeLi14}) that for 
$Q$ with no edge loops, the variety $\mathfrak{M}$ is nonempty iff $V(\lambda)_\mu \neq 0$.\footnote{Actually, for an \underline{arbitrary} quiver $Q$, it is known when the variety  $\mathfrak{M}$ is nonempty (see \cite[Theorem~1.3]{BS}). We are grateful to Gwyn Bellamy and Pavel Shlykov for explaining this to us and providing the reference.}  

Combining facts above with (non-equivariant) K-theoretic Hikita conjecrture we obtain the following corollary (compare with \cite[Conjecture~3.25(1)]{bfn_slices}). 
\begin{cor}\label{cor: fixed coulomb}
\begin{enumerate}[(a)]
    \item \label{cor: fixed coulomb a}  We have $\mathcal{M}^\nu(\bC) =(\mathcal{M}^\times)^\nu(\bC)
$ is a single point if $\mathfrak{M} \neq \varnothing$ and is empty otherwise.

\item If $Q$ has no edge loops, then 
$
\mathcal{M}^\nu(\bC)=(\mathcal{M}^\times)^\nu(\bC)
$ is a single point if $V(\lambda)_\mu \neq 0$ and is empty otherwise.
\end{enumerate}
\end{cor}
\begin{proof}
We have identifications:
\begin{equation}\label{eq: hikita hom k repeat}
\mathbb{C}[\mathcal{M}^\nu] \simeq H^*(\mathfrak{M}), \qquad \mathbb{C}[(\mathcal{M}^\times)^\nu] \simeq K(\mathfrak{M}). 
\end{equation}
Note now that the algebras $H^*(\mathfrak{M})$, $K(\mathfrak{M})$ are isomorphic via the Chern character, so we obtain the identification of algebras $\mathbb{C}[\mathcal{M}^\nu] \simeq \mathbb{C}[(\mathcal{M}^\times)^\nu]$ that implies the identification $\mathcal{M}^\nu(\bC)=(\mathcal{M}^\times)^\nu(\bC)$. Note now that $\mathcal{M}^\nu(\bC)$ is nothing else but the spectrum of the quotient of $\mathbb{C}[\mathcal{M}^\nu]$ by the radical. So, applying~\eqref{eq: hikita hom k repeat}, we conclude that $\mathcal{M}^\nu(\bC)$ is isomorphic to the spectrum of $H^*(\mathfrak{M})$ modulo the radical. Clearly, this quotient is isomorphic to $H^0(\mathfrak{M})$. Now, the claims of \cite{CB}, \cite{Nak98} cited above imply the claim. 
\end{proof}

\begin{rem} \label{rem: iso of fixed point algebras}
Let's  point out that the proof of Corollary~\ref{cor: fixed coulomb} also implies that the algebras $\mathbb{C}[\mathcal{M}^\nu]$, $\mathbb{C}[(\mathcal{M}^\times)^\nu]$ are isomorphic. Assume for a second that $\mathfrak{M} \neq \varnothing$. Then, the isomorphisms~\eqref{eq: hikita hom k repeat} suggest that the unique $\nu$-fixed points of $\mathcal{M}^\times$, $\mathcal{M}$ are nonsingular iff $H^*(\mathfrak{M})=\mathbb{C}$ (for example, when $\mathfrak{M}$ is a point or, more generally, is $T^*\mathbb{A}^k$).

In fact, this isomorphism of algebras can be proved without assuming the Hikita conjecture, using Theorem~\ref{thm_iso_after_completions} and an unpublished result of Kamnitzer--Weekes\footnote{We thank Kifung Chan for asking us this question.}.
\end{rem}

\subsubsection{Fixed points: deformed case}

Let's now describe the fixed points of deformed K-theoretic and homological Coulomb branches. Using the same computation as in the proof of Corollary~\ref{localization of quiver side of homological hikita over T_w}, this  reduces to the nondeformed case above (but for different quiver gauge theory).

Fix an element $f^\times \in F$ and let $\mathcal{M}^\times_{f^\times}$ be the fiber of $\mathcal{M}^\times_F$ over $f^\times$.
For an element $t_{\bf{v}}^\times \in T_{\bf{v}}/W$,
let $Q_{t_{\bf{v}}^\times,f^\times}$ be the quiver as in the proof of Proposition \ref{prop: fixed points fo quiver for arbitrary F}, namely the one that corresponds to $(Z_{G_{\bf{v}}}(t_{\bf{v}}^{\times}),N^{(t^{\times}_{\bf{v}},f^\times)})$.
Let $\mathfrak{M}_{(t_{\bf{v}}^{\times},f^\times)}$ be the corresponding quiver variety. 
We say that $t_{\bf{v}}^\times$ is {\emph{relevant}} to $f^\times$ if the corresponding quiver variety $\mathfrak{M}_{(t_{\bf{v}}^{\times},f^\times)}$ is nonempty.  Similarly, for $f \in \mathfrak{f}$, we say that $t_{\bf{v}} \in \mathfrak{t}_{\bf{v}}/S_\bv$ is relevant to $f$ if the quiver variety corresponding to $(Z_{G_{\bf{v}}}(t_{\bf{v}}),N^{(t_{\bf{v}},f)})$ is nonempty.  For a variety $X$, let $\on{Comp}(X)$ denote the set of its connected components.

\begin{cor}
(a) There are canonical bijections:
\begin{equation*}
(\mathcal{M}^\times_{f^\times})^\nu(\mathbb{C}) \leftrightarrow  \on{Comp}(\mathfrak{M}^{f^\times}) \leftrightarrow \{t_{\bf{v}}^{\times} \in T_{\bf{v}}\,|\, t_{\bf{v}}^{\times}~\text{is relevant for}~f^\times\}/S_\bv,
\end{equation*}
\begin{equation*}
(\mathcal{M}_{f})^\nu(\mathbb{C})  \leftrightarrow \on{Comp}(\mathfrak{M}^f) \leftrightarrow \{t_{\bf{v}} \in \mathfrak{t}_{\bf{v}}\,|\, t_{\bf{v}}~\text{is relevant for}~f\}/S_\bv. 
\end{equation*}

(b) If $f^\times$ is generic in some $A \subset F$ and $f$ is generic in $\operatorname{Lie}A$, there is canonical bijection: 
\begin{equation*}
(\mathcal{M}^\times_{f^\times})^\nu(\mathbb{C}) \leftrightarrow (\mathcal{M}_{f})^\nu(\mathbb{C}).
\end{equation*}
\end{cor}
\begin{proof}
Let's prove part (a) for $\mathcal{M}^\times_{f^\times}$, the argument for $\mathcal{M}_f$ is identical. 
The argument from the proof of Corollary~\ref{localization of quiver side of homological hikita over T_w} combined with equivariant K-theoretic Hikita shows that the algebra $\bC[(\mathcal{M}^\times_{f^\times})^\nu]$ is isomorphic to $K(\mathfrak{M}^{f^\times})$ which is isomorphic to the direct sum of K-theories of quiver varieties $\mathfrak{M}_{(t_{\bf{v}}^\times,f^\times)}$ for relevant $t_{\bf{v}}^\times$. The same argument as in the proof of Corollary~\ref{cor: fixed coulomb} finishes the proof.  Part (b) follows from part (a) together with $\mathfrak{M}^f=\mathfrak{M}^{f^\times}$. In fact, part (b) is true in general (without assuming Hikita conjecture), compare with Remark~\ref{rem: iso of fixed point algebras}. 
\end{proof}

\begin{rem}
Note that that every element of the set $\on{Comp}(\mathfrak{M}^{f^\times})$ is an algebraic variety. In particular, we have the dimension function $\on{Comp}(\mathfrak{M}^{f^\times}) \rightarrow \mathbb{Z}_{\geqslant 0}$. 
It would be interesting to describe this function via the Coulomb branch perspective.
When $Q$ is the Jordan quiver, similar objects were considered and studied in~\cite{Pae25}.
\end{rem}

\subsubsection{Quantization}
Finally, omitting the details, let us mention that assuming the quantized equivariant K-theoretic Hikita conjecture \eqref{quantized hikita conjecture} holds for $(Q,{\bf{v}},{\bf{w}})$, one would obtain the bijection: 
\begin{equation*}
\on{Irr}(\mathcal{O}_\nu(\mathcal{A}^\times_{q,f^\times}({\bf{v}},{\bf{w}})))    \longleftrightarrow \on{Comp}(\mathfrak{M}({\bf{v}},{\bf{w}})^{(q,f^\times)})
\end{equation*}
where $\on{Irr}(\mathcal{O}_\nu(\mathcal{A}^\times_{q,f^\times}))$ is the set of irreducible objects in the category $\mathcal{O}$ for the algebra
\begin{equation*}
\mathcal{A}^\times_{q,f^\times}({\bf{v}},{\bf{w}}) := (K^{G_{\bf{v}} \times F \times \bC^\times}(\mathcal{R}_{G_{\bf{v}},{\bf{N}}}))|_{(q,f^\times)}.
\end{equation*}
In particular, fixing ${\bf{w}}$ but allowing ${\bf{v}}$ to vary, we obtain the bijection:
\begin{equation}\label{eq: bij irr comp K theory}
\bigsqcup_{\bf{v}}\on{Irr}(\mathcal{O}_\nu(\mathcal{A}^\times_{q,f^\times}({\bf{v}},{\bf{w}})))    \longleftrightarrow \bigsqcup_{\bf{v}}\on{Comp}(\mathfrak{M}({\bf{v}},{\bf{w}})^{(q,f^\times)}).
\end{equation}

The homological version of the bijection (\ref{eq: bij irr comp K theory}) in case of ADE quivers is explained in \cite{KTWWY19, KTWWY19b}. Both sides are equipped with a structure of $\mathfrak{g}_Q$-crystal called monomial crystal (on the quiver variety side this is done by Nakajima in \cite{Nak01b} and on the Coulomb side this is one of the main results of the aforementioned papers). It is proved that the bijection induces an isomorphism of crystals. It would be very interesting to extend this to the K-theoretic setting as in homological setting monomial crystals proved to be a very useful tool to study category $\mathcal{O}$ for (truncated) shifted Yangians.

\appendix

\section{On homological Hikita conjecture in types ADE} \label{appendix: ade}
The Appendix concerns homological Hikita conjecture for type ADE quivers.
According to \cite{bfn_slices}, the Coulomb branch in this case is isomorphic to a generalized affine Grassmannian slice $\cW^\la_\mu$.
The conjecture thus establishes a relation between $\cW_\mu^\la$  and  the quiver variety $\wti \fM(\la - \mu, \la)$. Note that for the case when $\mu$ is dominant (that is, $\cW^\la_\mu$ is an honest ``non-generalized'' slice), the conjecture was proved in \cite[Theorem~8.1]{KTWWY19}.

In the main body of this paper, we prove two statements about ADE theories: equivariant version of homological conjecture (Corollary \ref{homo-hikita for ADE quivers}) and K-theoretic equivariant conjecture (Corollary \ref{cor: k-hikita for ade}). Both of the proofs use inductive argument, and even if one wants to prove the final result for the case when $\mu$ is dominant, our proofs use the result for smaller~$\la, \mu$, when $\mu$ is not necessarily dominant.

The goal of this Appendix is twofold. First, we prove the non-equivariant homological conjecture for the case when $\mu$ is not necessarily dominant (Theorem~\ref{thm: hikita for generalized slices}). Second, we  provide a direct geometric argument to give a different proof of Corollary \ref{homo-hikita for ADE quivers}, as we believe it is of independent interest (Theorem~\ref{Equivariant Hikita conjecture for ADE quivers}). Both of these arguments heavily rely on the proof of \cite[Theorem~8.1]{KTWWY19}.

In Section~\ref{sec: appendix slices and repellers} we study generalized slices in affine Grassmannian and prove required for us facts about repeller subschemes in them. In Section~\ref{sec: appendix non-eqvi hikita} we prove the non-equivariant Hikita conjecture.
In Section~\ref{sec: appendix equi-hikita} we prove the equivariant version by a direct geometric argument (reproving Corollary~\ref{homo-hikita for ADE quivers}).


\subsection{Generalized slices and repeller subschemes} \label{sec: appendix slices and repellers}
For this subsection, let $G$ be an arbitrary reductive group with Cartan torus $T$, opposite Borel subgroups $B, B_-$ and their unipotent radicals $U, U_-$. The affine Grassmannian $\Gr_G = G((z)) / G[[t]]$ is an ind-scheme, parametrizing pairs $(\cP, \sigma)$, where $\cP$ is a $G$-bundle on $\bP^1$, and $\sigma\colon \cP^{\mathrm{triv}}_{\bP^1 \setminus \{0\}} \iso \cP|_{\bP^1 \setminus \{0\}}$ is a trivialization of $\cP$ away of $0 \in \bP^1$.

The thick affine Grassmannian for $G$ is a scheme (of infinite type) defined as $\operatorname{Gr}^{\mathrm{thick}}_G = G((z^{-1}))/G[z]$. 
It  is the moduli space of pairs $(\cP, \sigma_{D_\infty})$, where $\cP$ is a $G$-bundle on $\bP^1$, and $\sigma_{D_\infty} \colon \mathcal{P}_{D_\infty}^{\mathrm{triv}} \iso \mathcal{P}|_{D_\infty}$ is a trivialization; here $D_\infty = \operatorname{Spec}\bC[[z^{-1}]]$. Recall that the Beilinson--Drinfeld Grassmannian $\operatorname{Gr}_{G,\bA^N}$ parametrizes triples $(\underline{z},\mathcal{P},\sigma)$, where $\underline{z}=(z_1,\ldots,z_N) \in \bA^N$ is a collection of points, $\mathcal{P}$ is a $G$-bundle and $\sigma\colon \mathcal{P}^{\mathrm{triv}}_{\bP^1 \setminus \{z_1,\ldots,z_N\}} \iso \mathcal{P}|_{\bP^1 \setminus \{z_1,\ldots,z_N\}}$ is a trivialization. Starting with a collection $\underline{\lambda}=(\la_1,\ldots,\la_N)$ we define $\ol{\on{Gr}}^{\ul{\la}}_{G,\BA^N} \subset \operatorname{Gr}_{G,\bA^N}$ to be a reduced subvariety which $\mathbb{C}$-points are  triples $(\underline{z},\mathcal{P},\sigma)$ such that the pole of $\sigma$ at $z_i$ is $\leqslant \lambda_i$.
Denote also by $'\operatorname{Bun}_G(\bP^1)$ the moduli stack of $G$-bundles on $\bP^1$ with a $B$-structure at $\infty$. Let $\operatorname{Bun}^{w_0 \mu}_B(\bP^1)$ be the  the moduli stack of degree $w_0\mu$ $B$-bundles on $\bP^1$.

For fixed dominant coweight $\la$ and arbitrary coweight $\mu$, generalized slices in affine Grassmannian are defined as:
\begin{equation*}
\cW^{\lambda}_\mu := \ol{\on{Gr}}^{\la}_{G} \times_{'\operatorname{Bun}_G(\bP^1)} \operatorname{Bun}_{B}^{w_0\mu}(\mathbb{P}^1). 
\end{equation*}
For collection of dominant coweights $\ula$ and arbitrary $\mu$, the deformation of the generalized slice in the affine Grassmannians is defined as (see \cite[Section~2.]{bfn_slices}):
\begin{equation*}
\cW^{\ul{\la}}_{\mu,\bA^N} := \ol{\on{Gr}}^{\ul{\la}}_{G,\BA^N} \times_{'\operatorname{Bun}_G(\bP^1)} \operatorname{Bun}_{B}^{w_0\mu}(\mathbb{P}^1).
\end{equation*}

We have a map
\begin{equation*}
\operatorname{Gr}_{G,\bA^N} \rightarrow \operatorname{Gr}^{\mathrm{thick}}_G \times \bA^N,~(\mathcal{P},\sigma) \mapsto (\mathcal{P},\sigma|_{D_\infty},\underline{z}), 
\end{equation*}
it restricts to the closed embedding of schemes: 
\begin{equation*}
\ol{\on{Gr}}^{\ul{\la}}_{G,\BA^N} \hookrightarrow \operatorname{Gr}^{\mathrm{thick}}_G \times \bA^N.
\end{equation*}

We have a natural projection $G((z^{-1})) \times \bA^N \rightarrow \operatorname{Gr}_G^{\mathrm{thick}} \times \bA^N$, let $(\overline{G[z]z^\lambda G[z]})_{\mathbb{A}^N}$ be the preimage of $\ol{\on{Gr}}^{\ul{\la}}_{G,\BA^N}$.

It follows from \cite[Section 2(xi)]{bfn_slices} that 
\begin{equation*}
\mathcal{W}^{\underline{\lambda}}_{\mu,\mathbb{A}^N} = (\overline{G[z]z^\lambda G[z]})_{\mathbb{A}^N} \cap   
(U[[z^{-1}]]_1 z^\mu T[[z^{-1}]]_1 U_{-,1}[[z^{-1}]]_1 \times \mathbb{A}^N),
\end{equation*}
where by intersection we mean the fiber product over $G((z^{-1})) \times \bA^N$.

Recall the following definition (see \cite[Definition 1.8.3]{dg}). Let $Z$ be a space equipped with an action of $\mathbb{C}^\times$.  We set: 
\begin{equation*}
X^- := {\bf{Maps}}^{\mathbb{C}^\times}(\mathbb{A}^1_-,X),
\end{equation*}
where $\bA^1_-$ is $\bA^1$ with an action of $\bC^\times$ given by $t \cdot x = t^{-1}x$.

We start with couple well-known results.
Recall the cocharacter $2\rho\colon \bC^\times \rightarrow T$, it induces the action $\bC^\times \curvearrowright \operatorname{Gr}_{G,\bA^N}$.
\begin{lem}
The natural morphism $\operatorname{Gr}_{B_-,\bA^N} \rightarrow \operatorname{Gr}_{G,\bA^N}$ induces an isomorphism:
\begin{equation*}
\operatorname{Gr}_{B_-,\bA^N} \iso \operatorname{Gr}_{G,\bA^N}^{-}.
\end{equation*}
\end{lem}
\begin{proof}
For $\on{Gr}_G$, this is \cite[Proposition~3.4]{HR21}. For a twisted version of BD Grassmannian, this is \cite[Proposition~5.6,iii)]{HR21}. For the usual BD Grassmannian $\on{Gr}_{G, \bA^N}$ the proof is the same.
\end{proof}
From now on, we use the identification $\operatorname{Gr}_{G,\bA^N}^{-} \simeq \operatorname{Gr}_{B_-,\bA^N}$.
Let us also recall the description of $\operatorname{Gr}_{B_-}^{\mathrm{thick}}=B_-((z^{-1}))/B_-[z]$. Denote $\Lambda=\operatorname{Hom}(\bC^\times,T)$.
\begin{lem}\label{descr of Gr thick B}
Connected components of the scheme $\operatorname{Gr}_{B_-}^{\mathrm{thick}}$ are labeled by $\Lambda$. The connected component  $\operatorname{Gr}_{B_-,\mu}^{\mathrm{thick}}$ corresponding to $\mu \in \Lambda$ is isomorphic to $z^\mu T[[z^{-1}]]_1U[[z^{-1}]]_1$ via the map $z^\mu g \mapsto [z^\mu g]$.
\end{lem}
\begin{proof}
It is enough to check that the natural multiplication morphism
\begin{equation*}
\bigsqcup_{\mu \in \Lambda} \left( z^\mu T[[z^{-1}]]_1U[[z^{-1}]]_1 \right) \times T[z]U[z] \rightarrow T((z^{-1}))U((z^{-1}))
\end{equation*}
is an isomorphism. 

This is equivalent to showing that  morphisms:
\begin{align*}
\bigsqcup_{\mu \in \Lambda}(z^\mu T[[z^{-1}]]_1) \times T[z] &\rightarrow T((z^{-1})),&  U[[z^{-1}]]_1 \times U[z] &\rightarrow U((z^{-1}))
\end{align*}
are isomorphisms. The second one is an isomorphism by \cite[Lemma 4.6]{k}. To prove the claim for the first one it is enough to assume that $T=\bC^\times$. So, our goal is to show that for any test local $\bC$-algebra $R$, the natural morphism 
\begin{equation*}
\bigsqcup_{k \in \bZ}(z^k+z^{k-1}R[[z^{-1}]]) \times R[z]^\times \rightarrow R((z^{-1}))^{\times}
\end{equation*}
is an isomorphism.

An element of $R((z^{-1}))$ is invertible iff it is of the form $\sum_{i=k}^{-\infty} a_iz^i$ such that for some $l \leqslant k$, the elements $a_k, a_{k-1}, \ldots, a_{l+1}$ are nilpotent and the element $a_l \in R$ is invertible. An element of $R[z]$ is invertible iff it is of the form $\sum_{i=0}^{p}b_iz^i$ with $b_0$ invertible and $b_i, i>0$ nilpotent.

Pick an element $\sum_{i=k}^{-\infty} a_iz^i \in R((z^{-1}))^{\times}$, we want to prove that it can be uniquely presented as:
\begin{equation}\label{pres as product}
\sum_{i=k}^{-\infty} a_iz^i  = z^m \Big(\sum_{i=0}^p b_iz^i\Big)(1+z^{-1} \cdot {\text{lower terms}}).
\end{equation}
First of all, note that the first non-nilpotent term of the RHS of (\ref{pres as product}) is in front of $z^m$ and is equal to $b_0$ plus some linear combination of $b_i$ ($i>0$). We conclude that $m=l$. Dividing by $z^m$, we can assume that $m=l=0$, $k \geqslant 0$. Dividing by $a_0$, we can assume that $a_0=1$. Now, we can consider the logarithm 
\begin{equation*}
\operatorname{ln}\Big(\sum_{i=k}^{-\infty} a_iz^i \Big) = \operatorname{ln}\Big(1 + \sum_{i \neq 0} a_iz^i \Big) \in R((z^{-1})),
\end{equation*}
which is clearly well-defined. 
We can uniquely decompose this logarithm as $c_+ + c_-$, where $c_+ \in R[z]$, $c_- \in z^{-1}R[[z^{-1}]]$. 

Note now that $c_+$ is of the form $\sum_{i=0}^r c_iz^i$, with $c_0 \in R^\times$ and $c_i, i>0$ are nilpotent. We conclude that $\operatorname{exp}(c_+) \in R[z]^{\times}$. So, we see that 
\begin{equation*}
\sum_{i=k}^{-\infty} a_iz^i = \operatorname{exp}(c_+)\operatorname{exp}(c_-),
\end{equation*}
where $\operatorname{exp}(c_+) \in R[z]^{\times}$, $\operatorname{exp}(c_-) \in 1+ z^{-1}R[[z^{-1}]]$ and moreover it is clear that this decomposition is unique (because the decomposition of $\operatorname{ln}\Big(\sum_{i=k}^{-\infty} a_iz^i \Big)$ into the sum $c_+ + c_-$ is unique).
\end{proof}

For $\mu \in \Lambda$, set $\operatorname{Gr}_{B_-,\mu,\bA^N} := \operatorname{Gr}_{B_-,\bA^N} \cap 
\operatorname{Gr}^{\mathrm{thick}}_{B_-,\mu}$.

\begin{lem}\label{closed rep W to B-}
The natural morphism:     
\begin{equation*}
\mathcal{W}^{\underline{\lambda},-}_{\mu,\bA^N} \rightarrow \operatorname{Gr}_{B_-,\mu,\bA^N} \cap \ol{\on{Gr}}^{\ul{\la}}_{G,\BA^N} 
\end{equation*}
is a closed embedding.
\end{lem}
\begin{proof}
Note that as $T$-spaces:
\begin{equation*}
U[[z^{-1}]]_1  z^\mu T[[z^{-1}]]_1  U_{-,1}[[z^{-1}]]_1 \simeq U[[z^{-1}]]_1 \times T[[z^{-1}]]_1  \times U_{-,1}[[z^{-1}]]_1. 
\end{equation*}
It follows that 
\begin{equation*}
(U[[z^{-1}]]_1 z^\mu T[[z^{-1}]]_1  U_{-,1}[[z^{-1}]]_1)^- = z^\mu T[[z^{-1}]]_1 U_{-,1}[[z^{-1}]].
\end{equation*}

So, we have
\begin{equation}\label{rep W matrix}
\mathcal{W}^{\underline{\la},-}_{\mu,\bA^N}=(\overline{G[z]z^\lambda G[z]})_{\mathbb{A}^N} \cap z^\mu T[[z^{-1}]]_1 U_{-,1}[[z^{-1}]].
\end{equation}

It now follows from Lemma \ref{descr of Gr thick B} that the composition 
\begin{equation*}
\mathcal{W}^{\underline{\lambda},-}_{\mu,\bA^N} \rightarrow \operatorname{Gr}_{B_-,\mu,\bA^N} \cap \ol{\on{Gr}}^{\ul{\la}}_{G,\BA^N} \rightarrow \operatorname{Gr}_{B_-,\mu}^{\mathrm{thick}}
\end{equation*}
is a closed embedding. Using that the morphism $\operatorname{Gr}_{B_-,\mu,\bA^N} \cap \ol{\on{Gr}}^{\ul{\la}}_{G,\BA^N}  \rightarrow \operatorname{Gr}_{B_-,\mu}^{\mathrm{thick}}$ is a closed embedding, hence, separated, we conclude that the morphism $\mathcal{W}^{\underline{\lambda},-}_{\mu,\bA^N} \rightarrow \operatorname{Gr}_{B_-,\mu,\bA^N}$ is also a closed embedding. 
\end{proof}

We are now ready to prove the main result of this section. We use the identification $\operatorname{Gr}_{G,\bA^N}^{-} \simeq \operatorname{Gr}_{B_-,\bA^N}  = \bigsqcup_{\mu \in \Lambda}\operatorname{Gr}_{B_-,\mu,\bA^N}$ discussed above.
\begin{prop}\label{rep are equal}
The natural morphism:
\begin{equation*}
(\mathcal{W}^{\underline{\lambda}}_{\mu,\bA^N})^{-} \rightarrow \operatorname{Gr}_{G,\bA^N}^{-}
\end{equation*}
is an isomorphism onto $\ol{\on{Gr}}^{\ul{\la}}_{G,\BA^N}  \cap  \operatorname{Gr}_{B_-,\mu,\bA^N}$.
\end{prop}
\begin{proof}
It follows from Lemma \ref{closed rep W to B-} that the morphism $(\mathcal{W}^{\underline{\lambda}}_{\mu,\bA^N})^{-} \hookrightarrow \ol{\on{Gr}}^{\ul{\la}}_{G,\BA^N}  \cap  \operatorname{Gr}_{B_-,\mu,\bA^N}$ is the closed embedding. It remains to construct a section of this morphism. This is done in completely same way as in \cite[Section 4.10]{k}.
\end{proof}

Denote $\nu = 2\rho$. 

\begin{cor} \label{cor: fixed points of deformed slices are same as of repellers}
The natural morphism of fixed points subschemes
\begin{equation*}
\bigsqcup_{\mu \leqslant \lambda}(\mathcal{W}^{\underline{\lambda}}_{\mu,\bA^N})^\nu \iso (\ol{\on{Gr}}^{\ul{\la}}_{G,\BA^N})^\nu
\end{equation*}
is an isomorphism. 
\end{cor}
\begin{proof}
The claim follows from Proposition \ref{rep are equal} using that for a scheme $Z$ with a $\bC^\times$-action we have $Z^{\bC^\times} = (Z^-)^{\bC^\times}$.
\end{proof}


\begin{rem}
Of course, we also get the non-deformed version of Proposition~\ref{rep are equal} and hence Corollary~\ref{cor: fixed points of deformed slices are same as of repellers}:
\begin{equation} \label{eq: fixed points of slices are same as of repellers}
\bigsqcup_{\mu} (\cW^\la_\mu)^\nu \simeq (\ol \Gr_G^\la)^\nu.
\end{equation}
At the level of $\bC$-points, this was already checked in \cite[Theorem 3.1(1)]{k}.
\end{rem}

\subsection{Hikita conjecture for generalized slices} \label{sec: appendix non-eqvi hikita}
We now turn to the following setting.

Let $G$ be the adjoint group with simple Lie algebra $\g$ with Dynkin diagram $Q$ of type ADE, let $\la$ be its dominant coweight, and $\mu$  an arbitrary coweight.
Let $\la = \sum_{i \in Q_0} w_i \om_i$ be the decomposition into the sum of fundamental weights (here and throughout of this section, we identify weights and coweights using that $\mathfrak{g}$ is simply-laced), and $\al := \la - \mu = \sum_{i \in I} v_i \al_i$ be the decomposition to sum of simple roots.

The Coulomb branch, associated with quiver $Q$, dimension vector~$(v_i)_{i \in Q_0}$ and framing vector~$(w_i)_{i \in Q_0}$, is isomorphic to the generalized slice $\cW^\la_\mu$, see \cite{bfn_slices}.

We also have the quiver variety $\wti \fM (\al, \la) := \wti \fM_Q (\bv, \bw)$, associated to this data. Let $\wti \fM(\la) = \bigsqcup_{\mu \in Q^+} \wti \fM(\la - \mu, \la)$.

Throughout the Appendix, we assume the same conditions on $w_i$ as in Corollary~\ref{homo-hikita for ADE quivers}.

For the case when $\mu$ is dominant, the Hikita conjecture was proved for this pair of dual symplectic singularities in  \cite[Theorem 8.1]{KTWWY19}:
\begin{equation} \label{nonequivariant hikita for ADE}
H^*(\wti\fM(\la - \mu, \la)) \simeq \bC[(\cW^{\la}_\mu)^{\nu}].
\end{equation}

In this section we prove the following result.
\begin{thm} \label{thm: hikita for generalized slices}
Let $\la$ be subject to conditions of Corollary~\ref{homo-hikita for ADE quivers}, and  $\mu$ be arbitrary (not necessarily dominant). There is an isomorphism of algebras over $H^*_{G_\bv}(\pt)$:
\begin{equation*} 
H^*(\wti\fM(\la - \mu, \la)) \simeq \bC[(\cW^{\la}_\mu)^{\nu}].
\end{equation*}
\end{thm}
\begin{proof}



Let $(\ol \Gr^\la)^\nu_\mu$ be the connected component of $(\ol \Gr^\la)^\nu$, corresponding to $\mu$. The argument in the proof of \cite[Theorem~8.1]{KTWWY19} actually shows that for any $\mu$ (not necessarily dominant) there is an isomorphism of  $H^*_{G_\bv}(\pt)$-algebras:
\begin{equation*}
H^*(\wti\fM(\la - \mu, \la)) \simeq \bC[(\ol \Gr^\la)^\nu_\mu].
\end{equation*}
We also have $\bC[ (\ol \Gr^\la)^\nu_\mu] \simeq \bC[(\cW^{\la}_{\mu})^{\nu}]$ from \eqref{eq: fixed points of slices are same as of repellers}. Note also that $H^*_{G_\bv}(\pt)$-action on these algebras comes from the $\ft[[z]]$-action (see Section \ref{app: u(t) vs cohom} below).
Isomorphism~\eqref{eq: fixed points of slices are same as of repellers} comes from the (non-deformed version of) morphism in Proposition~\ref{rep are equal}, thus it is $T[[z]]$-equivariant. The claim follows.
\end{proof}

\subsection{Equivariant Hikita conjecture} \label{sec: appendix equi-hikita}
Let now $\ula = (\la_i)$ be a tuple of fundamental coweights, $\sum_i \la_i = \la$.
There are Beilinson--Drinfeld deformations of $\ol \Gr_G^\la = \ov \Gr^\la$ and $\cW^{\la }_\mu$ over the affine space, which we can identify with $\ft_\bw$ in notations of precious section.
They are acted by the product of symmetric groups $S_\bw = \prod_{i \in I} S_{w_i}$, and we denote by $\ov \Gr^{\la}_{\ft_\bw / S_\bw}$ and $\cW^{\la }_{\mu, {\ft_\bw / S_\bw}}$ the quotients by this action (see Section~\ref{sec: appendix slices and repellers} for definitions). 

The main theorem of this section is the following statement.

\begin{thm}[Equivariant Hikita conjecture for ADE quivers] \label{Equivariant Hikita conjecture for ADE quivers}
There is an isomorphism of $\bC[(\ft_\bv / S_\bv) \times (\ft_\bw / S_\bw)]$-algebras:
\begin{equation*}
H^*_{G_\bw}(\wti \fM(\la - \mu, \la)) \simeq \bC[(  \cW^{\la}_{\mu,  \ft_\bw / S_\bw })^{\nu}].
\end{equation*}
\end{thm}

\begin{rem}
We show an isomorphism of algebras over $\bC[\ft_\bw / S_\bw]$; there is also a version over $\bC[\ft_\bw]$, which we deal with in the main body of the text:
\begin{equation*}
H^*_{T_\bw}(\wti \fM(\la - \mu, \la)) \simeq \bC[(  \cW^{\la}_\mu)^{\nu}_{\ft_\bw}].
\end{equation*}
In fact, they are equivalent. In one direction, one should take the $S_\bw$-invariants; in the opposite direction, one should base change from $\ft_\bw / S_\bw$ to $\ft_\bw$; see \cite[Theorem~6.1.22]{CG97}.
\end{rem}

In the proof of \eqref{nonequivariant hikita for ADE} in \cite{KTWWY19}, the step which connects the ``quiver side'' with the ``Coulomb side'', is the isomorphism of $\g[z]$-modules (\cite[Theorem 8.5]{KTWWY19}):
\begin{equation} \label{local isomorphism of g[t]-modules}
H^*(\wti \fM(\la)) \simeq \Gamma(\ov \Gr^\la, \cO(1)).
\end{equation}
Here the LHS is isomorphic to the dual local Weyl module by \cite[Proposition 4.4]{KN12}, the RHS is isomorphic to the dual affine Demazure module by \cite[Theorem 8.2.2]{Kum02}, and their isomorphism is \cite[Theorem A]{FL07}, or can be deduced from~\cite{Kas05}. 

Thus, our first step towards the proof of Theorem \ref{Equivariant Hikita conjecture for ADE quivers} is the following global analog of~\eqref{local isomorphism of g[t]-modules}:

\begin{prop} \label{global isomorphism of cohomology and sections}
There is an isomorphism of $U(\g[z])$-$\bC[\ft_\bw / S_\bw]$-bimodules
\begin{equation} \label{eq:global isomorphism of cohomology and sections}
H^*_{G_\bw}(\wti \fM(\la)) \simeq \Gamma((\ov \Gr^{\la })_{\ft_{\bw} / S_\bw}, \cO(1)).
\end{equation}
\end{prop}

The RHS of this Proposition is described in \cite{DFF21}. We are to describe the LHS and prove Proposition \ref{global isomorphism of cohomology and sections} after that.

Recall the action of the Yangian on the Borel--Moore homology $H_*^{G_\bw \times \bC^\times} (\wti \fM(\la))$, constructed in \cite{Var00}. Forgetting the contracting action, we get the $U(\g[z])$-action on $H_*^{G_\bw} (\wti \fM(\la))$. 
There is also an obvious commuting $H^*_{G_\bw} (\pt) = \bC[\ft_\bw / S_\bw]$-action. Recall the notion of global Weyl module $\bW(\la)$ over $\g[z]$. It admits the commuting action of the \textit{highest weight algebra}, isomorphic to $\bC[\ft_\bw / S_\bw]$ (see, e.g., \cite{CFK10}). $\bW(\la)$ is free over $\bC[\ft_\bw / S_\bw]$, as follows from~\cite[Corollary~B]{FL07}, or can be deduced from \cite{Kas02, BN04}, see introduction of~\cite{FL07}.

\begin{prop} \label{cohomology of quiver is global weyl}
$H^*_{G_\bw}(\fM(\la))$ is isomorphic to the $\bC[\ft_\bw / S_\bw]$-dual global Weyl module $\bW(\la)^\svee$ as a $U(\g[z])$-$\bC[\ft_\bw / S_\bw]$-bimodule.
\end{prop}

\begin{proof}
The proof is identical to the (non-equivariant) case of local Weyl module in \cite[Proposition 4.4]{KN12} and goes back to \cite{Nak01a}.

Namely, consider the Lagrangian subvariety $\fL(\al, \la) \subset \wti \fM(\al, \la)$ (pre-image of 0 under the resolution $\wti \fM(\al, \la) \rightarrow \fM(\al, \la)$), denote $\fL(\la) = \bigsqcup_\al \fL(\al, \la)$.
As in \cite[Proposition~13.3.1]{Nak01a}, one sees that $H_*^{G_\bw}(\fL(\la))$ is generated over $U(\g[z])$ by $H_*^{G_\bw}(\fL(0, \la))$, and that vectors from $H_*^{G_\bw}(\fL(0, \la))$ satisfy the defining relations of the global Weyl module. 
Hence, there is a surjection $\bW(\la) \twoheadrightarrow H_*^{G_\bw}(\fL(\la))$. Note that both $\bW(\la)$ and $H_*^{G_\bw}(\fL(\la))$ are free as modules over $\bC[\ft_\bw / S_\bw]$ (for the latter, see \cite[Theorem~7.5.3]{Nak01a}). Their ranks coincide by \cite[Proposition 4.4]{KN12}. Any surjective homomorphism between free modules of same rank is an isomorphism. Hence, $H_*^{G_\bw}(\fL(\la)) \simeq \bW(\la)$.  

By \cite[Theorem~7.3.5]{Nak01a}, there is a non-degenerate $\bC[\ft_\bw / S_\bw]$-linear pairing between $H_*^{G_\bw}(\fL(\la))$ and $H_*^{G_\bw}(\fM(\la))$.
Now the claim follows from the Poincar\'e duality for $\fM(\al, \la)$.
\end{proof}

We are now ready to prove Proposition \ref{global isomorphism of cohomology and sections}.

\begin{proof}[Proof of Proposition \ref{global isomorphism of cohomology and sections}]

By \cite[Theorem 4.5]{DFF21}, the RHS of \eqref{eq:global isomorphism of cohomology and sections} is isomorphic to $\bC[\ft_\bw / S_\bw]$-dual global Demazure module of level 1, $\bD(1, \la)^\vee$.

By Proposition \ref{cohomology of quiver is global weyl}, the LHS of \eqref{eq:global isomorphism of cohomology and sections} is isomorphic to $\bW(\la)^\vee$.

The isomorphism of $\bD(1, \la)$ and $\bW(\la)$ in types ADE is evident from the definition of $\bD(1, \la)$  and the isomorphism of corresponding local modules, see \cite[Section 3]{DF23}.
\end{proof}

We proceed by restricting to the schematic $T$-fixed points. 
\begin{lem} \label{lem: restriction to t-fixed points is an isomorphsim for BD}
Restriction to $T$-fixed points yields an isomorphism of $U(\ft[z])$-$\bC[\ft_\bw / S_\bw]$-bimodules:
\[
\Gamma(\ov \Gr^{\la }_{\ft_\bw / S_\bw}, \cO(1)) \simeq \Gamma((\ov \Gr^{\la }_{\ft_\bw / S_\bw})^T, \cO(1)).
\]
\end{lem}
\begin{proof}
$T$ acts on $\ov \Gr^{\la }_{\ft_\bw / S_\bw}$ fiberwise over $\ft_\bw / S_\bw$, and we have a restriction morphism of $U(\ft[z])$-$\bC[\ft_\bw / S_\bw]$-bimodules $\Gamma(\ov \Gr^{\la }_{\ft_\bw / S_\bw}, \cO(1)) \rightarrow \Gamma((\ov \Gr^{\la}_{\ft_\bw / S_\bw})^T, \cO(1))$. It is sufficient to check that it induces an isomorphism on each fiber over  $\ft_\bw / S_\bw$.

Indeed, a fiber of $\ov \Gr^{\la }_{\ft_\bw / S_\bw}$ over any point $p \in \ft_\bw / S_\bw$ is isomorphic to the product of affine Schubert varieties $\ov \Gr^{\nu_i}$ (depending on which of coordinates of $p$ are equal). For any $\ov \Gr^{\nu_i}$, restriction to $T$-fixed points yields an isomorphism of sections of $\cO(1)$ due to the main result of \cite{Zhu09}. The lemma follows.
\end{proof}

Our next goal is to prove the following
\begin{prop} \label{prop: O(1) trivializes on fixed points}
Line bundle $\cO(1)$ on $(\ol \Gr_{\ft_\bw / S_\bw}^\la)^T$ is 
trivial. One has an isomorphism 
\[
\Gamma((\ov \Gr_{\ft_\bw / S_\bw}^{\la })^T, \cO(1)) \simeq \Gamma((\ov \Gr_{\ft_\bw / S_\bw}^{\la })^T, \cO).
\]
\end{prop}

Consider the group schemes $T(\cK)_{\ft_\bw / S_\bw}$ and $T(\cO)_{\ft_\bw / S_\bw}$ over $\ft_\bw / S_\bw$ (for definitions, see \cite[(3.1.5),~(3.1.8)]{Zhu16}). The quotients of their (abelian) Lie algebras (or more formally, Lie rings) is naturally a Lie ring over $\bC[\ft_\bw / S_\bw]$, and we denote it by 
\[
\left. \ft_1[z^{-1}]_{\ft_{\bw} / S_\bw} = \Lie T(\cK)_{\ft_\bw / S_\bw} \right/ \Lie T(\cO)_{\ft_\bw / S_\bw}.
\]


\begin{proof}[Proof of Proposition \ref{prop: O(1) trivializes on fixed points}]
Recall that due to \cite[Proposition~3.4]{HR21}, one has $(\Gr_G)^T \simeq \Gr_T$, and similarly for Beilinson--Drinfeld Grassmannians $(\Gr_{G, \ft_\bw / S_\bw})^T \simeq \Gr_{T, \ft_\bw / S_\bw}$.

First, we prove that $\cO(1)$ is trivial on $(\Gr_{T, \ft_\bw / S_\bw})_\red$. Note that the irreducible components of $(\Gr_{T, \ft_\bw / S_\bw})_\red$ are parametrized by tuples of $T$-coweights $(\la_1, \hdots, \la_{|\bw|})$.
Each irreducible component is isomorphic to $\ft_\bw / S_\bw$, see \cite[Section~3.2.2]{Zhu09}\footnote{Note that in \cite[3.2.2]{Zhu09} it is claimed that these are the connected components of $((\Gr_T)_{\ft_\bw / S_\bw})_\red$. We assume this is a typo, and irreducible components are meant. Indeed, two such irreducible components have a common point over $0 \in {\ft_\bw / S_\bw}$ if and only if sums of coweights in corresponding tuples are equal.}.
Note that the irreducible components, for which the sum of coweights in the corresponding tuples are equal, have intersection over the locus of $\ft_\bw / S_\bw$, where some of coordinates are equal (such loci are unions of affine spaces of smaller dimensions, necessarily intersecting at $0 \in \ft / S_\bw$). 
Recall that any line bundle on an affine space is trivial. Moreover, the space of trivializations of a line bundle on an affine space can be identified with nonzero scalars. 
Hence, we can independently pick a trivialization of $\cO(1)$ on each irreducible component of $(\Gr_{T, \ft_\bw / S_\bw})_\red$, and then scaling these trivializations, make them agree over $0 \in \ft_\bw / S_\bw$. Since this determines a trivialization, they agree at the whole union of all irreducible components.

Next, we deal with the whole (non-reduced) $\ft_\bw / S_\bw$-scheme $\Gr_{T, \ft_\bw / S_\bw}$. Pick its connected component corresponding to a coweight $\mu$, denoted $(\Gr_{T, {\ft_\bw / S_\bw}})_\mu$. We can identify its tangent sheaf with the Lie ring $\ft_1[z^{-1}]_{\ft_\bw / S_\bw}$, see \cite[Proposition~3.1.9]{Zhu16} (informally, this means that $\ft_1[z^{-1}]_{\ft_\bw / S_\bw}$ acts freely and transitively on $(\Gr_{T, {\ft_\bw / S_\bw}})_\mu$). 
The line bundle $\cO(1)$ is naturally equivariant with respect to this Lie ring action, and hence it has a natural connection, flat over $\ft_\bw / S_\bw$. Giving a trivialization on the reduced part, this connection trivializes this line bundle on the whole scheme.

We showed that the line bundle $\cO(1)$ is trivial on $(\Gr_T)_{\ft_\bw / S_\bw} = (\Gr_{G, \ft_\bw / S_\bw})^T$. Restricting to the closed subscheme  $(\ol \Gr^\la_{\ft_\bw / S_\bw})^T$, we get the claim.
\end{proof}

Note that in \cite{KTWWY19} a local statement, similar to Proposition~\ref{prop: O(1) trivializes on fixed points} is proved by utilizing the free transitive $T_1[t^{-1}]$-action on each of the connected components of the fixed points. 
We identified the Lie ring of the global variant of $T_1[z^{-1}]$ with the tangent sheaf of our ind-scheme to trivialize the bundle in our case.
Due to the factorization property of $(\Gr_T)_{\ft_\bw / S_\bw}$, its tangent sheaf also factorizes, and hence the trivialization we perform coincides with (the product of ones)  in \cite{KTWWY19}, when restricted to any fiber over $\ft_\bw / S_\bw$.

Finally, we are ready to prove the main result of this subsection.
\begin{proof}[Proof of Theorem \ref{Equivariant Hikita conjecture for ADE quivers}]
Restricting an isomorphism of $U(\g[z])$-$\bC[\ft_\bw / S_\bw]$-bimodules of Proposition \ref{global isomorphism of cohomology and sections} to the weight $\mu$ subspace, we get an isomorphism of $U(\ft[z])$-$\bC[\ft_\bw / S_\bw]$-bimodules
\[
H^*_{G_\bw}(\wti \fM(\la - \mu, \la)) \simeq \Gamma((\ol \Gr^\la)_{\ft_\bw / S_\bw}, \cO(1))_\mu.
\]
Let us rewrite the right-hand side as
\begin{multline} \label{eq: chain of isomorphism for sections on BD}
\Gamma((\ol \Gr^\la)_{\ft_\bw / S_\bw}, \cO(1))_\mu \simeq \Gamma((\ol \Gr^\la)^T_{\ft_\bw / S_\bw}, \cO(1))_\mu \simeq \\
\Gamma((\ol \Gr^\la_{\ft_\bw / S_\bw})^T, \cO)_\mu \simeq \Gamma((\ol \Gr_{\ft_\bw / S_\bw}^\la)^\nu, \cO)_\mu \simeq \bC[(\cW^\la_{\mu, \ft_\bw / S_\bw})^\nu].
\end{multline}
The first isomorphism here is by Lemma \ref{lem: restriction to t-fixed points is an isomorphsim for BD}, the second one is due to  Proposition \ref{prop: O(1) trivializes on fixed points}, and the last one is due to Corollary \ref{cor: fixed points of deformed slices are same as of repellers}. We also used here an isomorphism $(\ol \Gr^\la_{\ft_\bw / S_\bw})^T \simeq (\ol \Gr^\la_{\ft_\bw / S_\bw})^\nu$: one of this schemes has evident closed embedding to the other, and an isomorphism can be checked fiberwise, which was done in \cite{KTWWY19}.

We claim that all isomorphisms in \eqref{eq: chain of isomorphism for sections on BD} are $\ft[z]$-equivariant, with respect to the natural action of this Lie algebra. Indeed, this claim can be checked fiberwise, over all points $p \in \ft_\bw / S_\bw$. For any such point, all isomorphisms of \eqref{eq: chain of isomorphism for sections on BD}, restricted to the fiber over $p$ are precisely the (product of) isomorphisms of \cite[Proposition~8.3, Theorem~8.4, Lemma~8.8]{KTWWY19}, which are proved to be $\ft[z]$-equivariant.

Up to this point, we proved the isomorphism of LHS and RHS of Theorem~\ref{Equivariant Hikita conjecture for ADE quivers} as $U(\ft[z])$-$ \bC[\ft_\bw / S_\bw]$-bimodules. It remains to deduce the isomorphism of algebras.
Recall that the action of $U(\ft[z]) \T \bC[\ft_\bw / S_\bw]$ on $H^*_{G_\bw}(\wti \fM(\la - \mu, \la))$ comes from the surjective Kirwan map of \cite{MN18}, $U(\ft[z]) \T \bC[\ft_\bw / S_\bw] \twoheadrightarrow H^*_{G_\bw}(\wti \fM(\la - \mu, \la))$. So we know that $\bC[(\cW^\la_{\mu, {\ft_\bw / S_\bw}})^\nu]$ is quotient of $U(\ft[z]) \T \bC[\ft_\bw / S_\bw]$ by the same subspace. These quotients respect the algebra structure, hence we indeed have the required isomorphism of algebras (compare with the usage of \cite[Lemma~8.8]{KTWWY19} in the proof of Theorem~8.1 \textit{loc. cit.}). The claim that the isomorphisms are $H^*_{G_{\bf{v}} \times G_{\bf{w}}}(\operatorname{pt})$-equivariant follows from Section \ref{app: u(t) vs cohom} below. 
\end{proof}

\subsection{Comparison of $U(\ft[z]) \T \bC[\bA^N]$ and $H_{G_\bv \times T_\bw}^*(\pt)$-actions}\label{app: u(t) vs cohom}

In the proof of Theorem~\ref{Equivariant Hikita conjecture for ADE quivers}, we identified the required algebras as quotients of $U(\ft[z]) \T \bC[\ft_\bw / S_\bw]$. In order to compare with the Coulomb realization (used, in particular, in Corollary~\ref{homo-hikita for ADE quivers}), we need to identify natural $U(\ft[z]) \T \bC[\ft_\bw / S_\bw]$ and $H_{G_\bv \times T_\bw}^*(\pt)$-actions. That is what we do in this section.

Let us recall the description of the integrable system on the (deformed) Coulomb branch after the identification with $\cW^{\underline{\lambda}}_{\mu,\bA^N}$. We have a natural projection: 
\begin{equation*}
\cW^{\underline{\la}}_{\mu,\bA^n} \rightarrow (T[[z^{-1}]]_1z^\mu) \times \mathbb{A}^N.
\end{equation*}
It induces the map 
\begin{equation}\label{map from funct torus to slice}
\bC[T[[z^{-1}]]_1z^\mu] \otimes \bC[\bA^N] \rightarrow \bC[\cW^{\underline{\la}}_{\mu,\bA^n}].
\end{equation}

Recall that $N=\sum_{i \in Q_0}w_i$. For $i \in Q_0$ and $s=1,\ldots,w_i$ let $c_{i,s} \in \mathfrak{t}_{w_i}^*$ be the weights of $T_{w_i}$ acting on $W_i$. We have the natural identification $\mathbb{C}[\bA^N]=\bC[c_{i,s} \ | \ i \in Q_0, 1\leq s \leq w_i]$. 

We set 
\begin{equation*}
p_i(u) = \prod_{s=1,\ldots,w_i}(u-c_{i,s}),
\end{equation*}
where $u$ is a formal variable. It follows from the definitions that $\operatorname{deg}p_i = w_i=\langle \lambda,\alpha_i\rangle$.

Now, we define a collection of elements $a_i^{(r)}, r \geqslant 0$ in $\bC[T[[z^{-1}]]_1z^\mu ][c_{i,s}]$. For $s \geqslant -\langle\mu,\alpha_i\rangle$ let $h_i^{(s)} \in \bC[T[[z^{-1}]]_1z^\mu]$ be the function sending $t \in T[[z^{-1}]]_1z^\mu$ to the  coefficient of $\alpha_i(t)$ in front of $z^{-s}$. Note that $h_i^{-\langle\mu,\alpha_i\rangle}=1$. We set 
\begin{equation*}
h_i(u) := \sum_{r \geqslant -\langle \mu,\alpha_i\rangle}h_i^{(r)}u^{-r} = u^{\langle \mu,\alpha_i\rangle} + \sum_{r>-\langle \mu,\alpha_i\rangle} h_i^{(r)}u^{-r}.
\end{equation*}
Then, $a_i^{(r)}$ are uniquely determined by requiring that the
following identity of formal series holds:
\begin{equation*}
h_i(u) = p_i(u) \cdot u^{-\langle \alpha_i,\lambda-\mu\rangle} \cdot \frac{\prod_{i \sim j}a_j(u)}{a_i(u)^2} = p_i(u) \cdot \frac{\prod_{j \sim i}u^{v_j}}{u^{2v_i}} \cdot \frac{\prod_{i \sim j}a_j(u)}{a_i(u)^2},
\end{equation*}
where $a_i(u)=\sum_{r \geqslant 0}a_i^{(r)}u^{-r}$ and for $i,j \in Q_0$ we write $i \sim j$ iff they are adjacent in $Q_1$.

\begin{rem}
Note that $\mathbb{C}[T[[z^{-1}]]_1z^\mu] \otimes \bC[c_{i,s}]$ has a  quantization $Y_\mu[c_{i,s}]$ over $\bC[\hbar]$ (algebra $Y_\mu$ was introduced in \cite[Sections 3.6, 3.7]{KWWY14}, see also \cite[Appendix B]{bfn_slices} and is called shifted Yangian). The elements $a_i^{(r)}$ are $\hbar=0$ specializations  of the elements $A_i^{(r)}$ (see, for example, \cite[(B.14)]{bfn_slices} for the definition of $A_i^{(r)}$).
\end{rem}

The map (\ref{map from funct torus to slice}) sends $a_i^{(r)}$ for $r > v_i$ to zero. We also have a (surjective) map 
\begin{equation*}
\BC[z^\mu T[[z^{-1}]]_1] \otimes \bC[\BA^N] \twoheadrightarrow H^*_{G_{\bf{v}} \times T_{\bf{w}}}(\on{pt}) 
\end{equation*}
given by $a_i^{(r)} \mapsto c_i(\cV_r)$. It follows from \cite[Theorems B.15, B.18 and Corollary B.28]{bfn_slices} that this map is compatible with the integrable system $H^*_{G_{\bf{v}} \times T_{\bf{w}}} \rightarrow \bC[\cW^{\underline{\lambda}}_\mu]$. In other words, functions $c_r(\mathcal{V}_i) \in \bC[\cW^{\underline{\lambda}}_\mu]$ are nothing else but the images of $a_i^{(r)}$ under (\ref{map from funct torus to slice}).

Let's now recall the action of $U(\mathfrak{t}[z]) \subset U(\mathfrak{g}[z])$ on $\mathbb{C}[(\cW^{\underline{\lambda}}_{\mu,\bA^N})^T]$.  We use the map $(\cW^{\underline{\lambda}}_{\mu,\bA^N})^T \rightarrow T[[z^{-1}]]_1z^\mu$.


We denote by $\langle\,,\,\rangle$ the residue pairing between $S^\bullet(\mathfrak{t}[z])$ and $\bC[z^{-1}\mathfrak{t}[[z^{-1}]]]$ that is given by:
\begin{equation*}
\langle x \otimes z^r, y \otimes z^k\rangle = r\delta_{r+k+1,0}(x,y),
\end{equation*}
where $(\,,\,)$ is the normalized invariant form on $\mathfrak{g}$. 

For $i \in Q_0$ and $r \geqslant 0$ set ${\bf{h}}_{i,r} := h_i \otimes z^r$, ${\bf{h}}_{i}(u):=\sum_{r \geqslant 0}{\bf{h}}_{i,r}u^{-r-1}$.
Now, we have the identification: 
\begin{equation}\label{CC symbol}
S^\bullet(\mathfrak{t}[z]) \iso \bC[T[[z^{-1}]]_1z^\mu], \qquad {\bf{h}}_i(u) \mapsto \operatorname{ln}(u^{-\langle \mu,\alpha_i\rangle}h_i(u))^{\prime}
\end{equation}
to be denoted $\eta$.
It follows from \cite[Lemma 8.8]{KTWWY19} together with the explicit formula for the Contou-Carr\'ere symbol (see, for example, \cite[Section 3.2.1]{Zhu09}) that the action of $S^\bullet(\mathfrak{t}[z])$ on $\mathbb{C}[(\cW^{\underline{\lambda}}_{\mu,\bA^N})^T]$ is induced by $\eta$ composed with the natural map $\mathbb{C}[T[[z^{-1}]]_1 z^\mu] \rightarrow \mathbb{C}[(\cW^{\underline{\lambda}}_{\mu,\bA^N})^T]$.

Let's now recall the action of $U(\mathfrak{t}[z]) \subset U(\mathfrak{g}[z])$ on $H^*_{T_{\bf{w}}}(\widetilde{\mathfrak{M}}(\lambda-\mu,\mu))$. We follow~\cite{Var00}. The action of $U(\mathfrak{t}[z])$ is obtained as $\hbar=0$ specialization of the Cartan subalgebra $H \subset Y(\mathfrak{g})$ acting on $H^*_{T_{\bf{w}} \times \BC^\times}(\widetilde{\mathfrak{M}}(\lambda-\mu,\mu))$. Algebra $H$ has generators ${\bf{H}}_{i,r}$, $i \in Q_0$, $r \geqslant 0$ (denoted by ${\bf{h}}_{i,r}$ in~\cite{Var00}) that specialize to ${\bf{h}}_{i,r} := h_i \otimes t^{r}$ when $\hbar=0$.

Let us recall more notations from \cite{Var00}. It is denoted by  $q$ the trivial bundle on $\widetilde{\mathfrak{M}}(\lambda-\mu,\mu)$ with the degree one action of $\bC^\times$. We have $\hbar = c_1(q^2)$.
\begin{equation*}
\CF_i({\bf{v}},{\bf{w}}) = q^{-2}\CW_i - (1+q^{-2})\CV_i + q^{-1}\sum_{j}\CV_j
\end{equation*}
is the virtual bundle on $\widetilde{\mathfrak{M}}(\lambda-\mu,\mu)$. By $\lambda_u(-)$ we denote the  equivariant Chern polynomial polynomial (for a line bundle $\mathcal{L}$, $\lambda_u(\mathcal{L})=1+c_1(\mathcal{L})u$). 
Directly from the definitions we have
\begin{equation*}
\lambda_{-1/u}(\CW_i) = u^{-w_i}p_i(u).
\end{equation*}
Set $a:=u^{-1}$ and denote $u^{-w_i}p_i(u)$ by $q(a)$.

We set $\tilde{A}_i(u) := \lambda_{-1/u}(\CV_i)$ and denote by $\tilde{a}_i(u)$ the specialization of $\tilde{A}_i(u)$ to $\hbar=0$. We form the generating function:
\begin{equation*}
{\bf{H}}_i(u) := \sum_{r \geqslant 0}{\bf{H}}_{i,r}u^{-r-1}.
\end{equation*}
 It follows from \cite{Var00} that the action of $\hbar {\bf{H}}_i(u)$ on $H^*_{T_{\bf{w}} \times \bC^\times}(\widetilde{\mathfrak{M}}(\lambda-\mu,\mu))$ is given via the multiplication by
\begin{multline*}
-1+\frac{\lambda_{-1/u}(\CF_i({\bf{v}},{\bf{w}}))}{\lambda_{-1/u}(q^2\CF_i({\bf{v}},{\bf{w}}))} = \\
= -1 + \frac{\la_{-1/u}(q^{-2}\CW_i) \la_{-1/u}((1+q^2)\CV_i) \prod_{j}\la_{-1/u}(q^{-1}\CV_j)}{\la_{-1/u}(\CW_i)\la_{-1/u}((1+q^{-2})\CV_i) \prod_{j}\la_{-1/u}(q\CV_j)} \\
= -1 +  \frac{q_i(a-\hbar) \cdot \tilde{A}_i(a+\hbar) \cdot \prod_{i \sim j}\tilde{A}_j(a-\frac{\hbar}{2})}{q_i(a) \cdot \tilde{A}_i(a-\hbar) \cdot \prod_{i \sim j}\tilde{A}_j(a+\frac{\hbar}{2})}
\end{multline*}

Recall that ${\bf{h}}_i(u)=(\frac{\hbar{\bf{H}}_i(u)}{\hbar})_{\hbar=0}$. Hence, the action of ${\bf{h}}_i(u)=\sum_{r \geqslant 0}{\bf{h}}_{i,r}u^{-r-1}$ is given by:
\begin{equation*}
\Big((\on{ln}q_i(u)) - 2\on{ln}(\tilde{a}_i(u)) +\sum_{i \sim j} \on{ln}(\tilde{a}_j(u)) \Big)' 
=\Big(\on{ln}\Big(\frac{q_i(u) \prod_{i \sim j}\tilde{a}_j(u)}{\tilde{a}_i(u)^2}\Big)\Big)',
\end{equation*}
where the derivative is taken w.r.t. the variable $u$.


\begin{prop}
The identification of $U(\mathfrak{t}[z])$-$\mathbb{C}[\bA^N]$-modules 
\begin{equation*}
\bC[(\cW^{\underline{\lambda}}_{\mu,\bA^N})^T] \simeq H^*_{T_{\bf{w}}}(\widetilde{\mathfrak{M}}(\lambda-\mu,\mu))
\end{equation*}
intertwines the $H^*_{G_{\bf{v}} \times T_{\bf{w}}}(\operatorname{pt})$-actions. 
\end{prop}
\begin{proof}
Our goal{\footnote{Note that when $\la$ is the sum of minuscule coweights, this can also be deduced from \cite[Section~8.3]{KTWWY19} combined with \cite{Nak01b}, using that both $\tilde{a}_i(u), a_i(u)$ diagonalize (after the base change to $\operatorname{Frac}H^*_{T_{\bf{w}}}(\operatorname{pt})$) in the same basis with the same eigenvalues.}} is to check that  $\eta(\tilde{a}_i(u))=a_i(u)$.
Recall that we have:
\begin{equation*}
{\bf{h}}_i(u) =     \Big(\on{ln}\Big(\frac{q_i(u) \prod_{i \sim j}\tilde{a}_j(u)}{\tilde{a}_i(u)^2}\Big)\Big)', \qquad u^{-\langle \alpha_i,\mu\rangle} \cdot h_i(u)=  \frac{q_i(u)\prod_{i \sim j} a_j(u)}{a_i(u)^2}.
\end{equation*}
Using  (\ref{CC symbol}) we conclude that: 
\begin{equation*}
\Big(\on{ln}\Big(\frac{q_i(u) \prod_{i \sim j}\eta(\tilde{a}_j(u))}{\tilde{a}_i(u)^2}\Big)\Big)'=\eta({\bf{h}}_i(u))=\Big({\on{ln}}\Big(\frac{q_i(u)\prod_{i \sim j}a_j(u)}{a_i(u)^2}\Big)\Big)'
\end{equation*}
so we have 
\begin{equation*}
\frac{\prod_{i \sim j}\eta(\tilde{a}_j(u))}{\tilde{a}_i(u)^2} = \frac{\prod_{i \sim j}a_j(u)}{a_i(u)^2},
\end{equation*}
which implies $\eta(\tilde{a}_i(u))=a_i(u)$, as desired. 
\end{proof}

\end{document}